\documentclass{amsart}
\usepackage{graphicx}
\usepackage{amssymb}
\usepackage{epstopdf}
\usepackage{nicefrac}
\usepackage{enumerate}
\newtheorem{thm}{THEOREM}[section]

\newtheorem{cor}[thm]{COROLLARY}
\newtheorem{defn}[thm]{DEFINITION}

\newtheorem{lemma}[thm]{LEMMA}
\newtheorem{prob}[thm]{PROBLEM}
\newtheorem{prop}[thm]{PROPOSITION}

\newcommand{\ds}{\displaystyle}
\newcommand{\F}{{\mathcal F}} % our favorite foliation
\newcommand{\wtF}{{\widetilde{\mathcal F}}} % some lift of our favorite foliation
\newcommand{\G}{\Gamma}

\newcommand{\e}{{\epsilon}} % some choice
\newcommand{\eF}{{\epsilon_{\mathcal F}}} % convex diameter defined in Lemma~\ref{lem-conhull}
\newcommand{\eU}{{\epsilon_{\cU}}} % Lebesgue number for given covering.
\newcommand{\eFU}{{\epsilon^{\F}_{\cU}}} % leafwise Lebesgue number defined in equation~\ref{eq-leafdiam}
\newcommand{\eTU}{{\epsilon^{\cT}_{\cU}}} % transverse Lebesgue number defined in equation~\ref{eq-transdiam}

\newcommand{\dFU}{\delta^{\F}_{\cU}} % leafwise radius defined by Equation~\ref{eq-Fdelta}
\newcommand{\dTU}{\delta^{\cT}_{\cU}} % equicontinuous uniform domain size defined in Proposition~\ref{prop-uniformdom}
\newcommand{\lF}{{\lambda_{\mathcal F}}} % leafwise convexity diameter defined in Lemma~\ref{lem-stronglyconvex}
\newcommand{\rp}{\rho_{\pi}} % modulus of continuity function for projections defined in equation~\ref{eq-modpi}
\newcommand{\rt}{\rho_{\tau}} % modulus of continuity function for sections defined in equation~\ref{eq-modtau}
\newcommand{\cGF}{\cG_{\F}} % denotes the pseudogroup of F
\newcommand{\cRF}{\cR_{\F}} % denotes the equivalence relation of F - not used so far
\newcommand{\GF}{\Gamma_{\F}} % denotes the groupoid of F - not used so far

\newcommand{\dX}{\text{diam}_{\fX}} %
\newcommand{\wth}{{\widetilde{h}}}

\newcommand{\wtx}{{\widetilde{x}}}
\newcommand{\wtz}{{\widetilde{z}}}
\newcommand{\wty}{{\widetilde{y}}}

\newcommand{\wtL}{{\widetilde L}}
\newcommand{\wtM}{\widetilde{M}}
\newcommand{\wtN}{\widetilde{N}}

\newcommand{\wtU}{{\widetilde U}}
\newcommand{\wtV}{\widetilde{V}}

\newcommand{\wtpi}{{\widetilde{\pi}}}

\newcommand{\wttau}{{\widetilde{\tau}}}
\newcommand{\wtvp}{{\widetilde{\varphi}}}
\newcommand{\wtgamma}{{\widetilde{\gamma}}}
\newcommand{\wtlambda}{{\widetilde{\lambda}}}

\newcommand{\wtcN}{{\widetilde{\mathcal N}}}
\newcommand{\wtcP}{{\widetilde{\mathcal P}}}
\newcommand{\wtfN}{\widetilde{\mathfrak{N}}}
\newcommand{\wtfT}{\widetilde{\mathfrak{T}}}
\newcommand{\wtfU}{\widetilde{\mathfrak{U}}}

\newcommand{\wtcH}{\widetilde{\mathcal H}}

\newcommand{\whU}{{\widehat U}}

\newcommand{\whcP}{{\widehat{\mathcal P}}}

\newcommand{\oU}{{\overline{U}}}

\newcommand{\mG}{{\mathbb G}}

\newcommand{\mN}{{\mathbb N}}

\newcommand{\mR}{{\mathbb R}}

\newcommand{\mT}{{\mathbb T}}
\newcommand{\mZ}{{\mathbb Z}}

\newcommand{\cC}{{\mathcal C}}

\newcommand{\cG}{{\mathcal G}}
\newcommand{\cH}{{\mathcal H}}
\newcommand{\cI}{{\mathcal I}}
\newcommand{\cJ}{{\mathcal J}}
\newcommand{\cK}{{\mathcal K}}

\newcommand{\cN}{{\mathcal N}}
\newcommand{\cO}{{\mathcal O}}
\newcommand{\cP}{{\mathcal P}}

\newcommand{\cR}{{\mathcal R}}
\newcommand{\cS}{{\mathcal S}}
\newcommand{\cT}{{\mathcal T}}
\newcommand{\cU}{{\mathcal U}}
\newcommand{\cV}{{\mathcal V}}
\newcommand{\cW}{{\mathcal W}}

\newcommand{\fB}{{\mathfrak{B}}}

\newcommand{\fD}{{\mathfrak{D}}}

\newcommand{\fF}{\mathfrak{F}}

\newcommand{\fM}{{\mathfrak{M}}}

\newcommand{\fS}{{\mathfrak{S}}}
\newcommand{\fT}{{\mathfrak{T}}}
\newcommand{\fU}{{\mathfrak{U}}}

\newcommand{\fX}{{\mathfrak{X}}}

\newcommand{\vp}{{\varphi}}

\begin{document}

\title{Homogeneous matchbox manifolds}

\thanks{2010 {\it Mathematics Subject Classification}. Primary 57S10, 54F15, 37B10, 37B45; Secondary 57R05, }
\thanks{Both authors supported by NWO travel grant 040.11.132}
\author{Alex Clark}
\thanks{AC supported in part by EPSRC grant EP/G006377/1}
\address{Alex Clark, Department of Mathematics, University of Leicester, University Road, Leicester LE1 7RH, United Kingdom}
\email{adc20@le.ac.uk}

\author{Steven Hurder}
\address{Steven Hurder, Department of Mathematics, University of Illinois at Chicago, 322 SEO (m/c 249), 851 S. Morgan Street, Chicago, IL 60607-7045}
\email{hurder@uic.edu}
\thanks{Version date: June 28, 2010; revised   July 11, 2011}

\date{}

% \subjclass{Primary 57R30, 57S10, 54F15, 37B10, 37B45; Secondary 52C22 , 53C12}

\keywords{solenoids, matchbox manifold, laminations, equicontinuous foliation, Effros Theorem, foliations}

\maketitle

% \vfill \eject

% {\small \baselineskip2pt \tableofcontents }

\section{Introduction} \label{sec-intro}

A \emph{continuum} is a compact, connected, and non-empty
metrizable space. A topological  space $X$ is \emph{homogeneous} if for
every $x, y \in X$, there exists a homeomorphism $h \colon X \to X$
such that $h(x) = y$.

\begin{thm}[Bing  \cite{Bing1960}] \label{thm-bing}
Let $X$ be    a homogeneous, circle-like
continuum that contains an arc. Then either $X$ is homeomorphic to a circle, or to   a Vietoris solenoid.
\end{thm}

In the course of the proof of Theorem~\ref{thm-bing},
Bing raised the question: If $X$ is a homogeneous
continuum, and if every proper subcontinuum of $X$ is an arc, must
$X$ then be a circle or a solenoid?
An affirmative answer to this question was given by Hagopian~\cite{Hagopian1977}, and
 subsequent (simpler) proofs  in the framework of $1$-dimensional
matchbox manifolds were given by Mislove and Rogers \cite{MR1989} and by  Aarts,  Hagopian and Oversteegen~\cite{AHO1991}. In this paper,
we prove the generalization of this result to  $n$-dimensional matchbox manifolds, for all $n \geq 1$.

We introduce some notations required to state our main result precisely.
An \emph{$n$-dimensional
solenoid} is an inverse limit space
\begin{equation}\label{eq-solenoid}
\cS = \lim_{\leftarrow} ~ \{p_{\ell+1} \colon M_{\ell +1} \to M_{\ell}\}
\end{equation}
where for $\ell \geq
0$, $M_{\ell}$ is a compact, connected, $n$-dimensional manifold without boundary, and the maps
$p_{\ell+1} \colon M_{\ell +1} \to M_{\ell}$ are   \emph{proper}
covering maps.
 A \emph{Vietoris solenoid}  is   a $1$-dimensional solenoid, where each $M_{\ell}$ is a circle.

If all of the defined compositions of the covering
maps $p_{\ell}$ are normal coverings, then $\cS$ is said to be a
\emph{McCord solenoid}.
McCord solenoids are homogeneous  \cite{McCord1965},
and conversely,  Fokkink and Oversteegen showed in \cite{FO2002} that any
homogeneous $n$-dimensional solenoid is homeomorphic to a McCord
solenoid.

An \emph{$n$-dimensional foliated space} $\fM$ is a continuum
which has a local product structure \cite{CandelConlon2000,MS2006};
that is, every point of $\fM$ has an open neighborhood homeomorphic
to an open subset of $\mR^n$ times a compact metric space (the local
transverse model). The leaves of the foliation $\F$ of $\fM$ are the
maximal connected components with respect to the fine topology on
$\fM$ induced by the plaques of the local product structure. Precise
definitions are given in Section~\ref{sec-concepts}.

A \emph{matchbox manifold} is a foliated space $\fM$ such that the
local transverse models are totally disconnected. Intuitively, a
$1$-dimensional matchbox manifold $\fM$ has local coordinate charts
$U$ which are homeomorphic to a ``box of matches.'' Manifolds and
$n$-dimensional solenoids   provide examples of matchbox
manifolds.

As remarked above, every homogeneous
$1$-dimensional matchbox manifold is homeomorphic to a circle or a
solenoid \cite{AHO1991}.  Our primary
result is the generalization of this $1$-dimensional result to
$n$-dimensions, thereby proving a strong version of a conjecture of
Fokkink and Oversteegen \cite[Conjecture 4]{FO2002} under a
smoothness assumption, as clarified in Section~\ref{sec-concepts}.

\begin{thm}\label{thm-main1}
Let $\fM$ be an  homogeneous smooth matchbox manifold. Then $\fM$ is
homeomorphic to a McCord solenoid. In particular, $\fM$ is minimal.
\end{thm}

As a consequence of Theorem~\ref{thm-main1} and the impossibility of
codimension-$1$ embeddings of solenoids as shown
in~\cite{ClarkFokkink2004}, we obtain the following corollary, which
is a generalization of the result of Prajs~\cite{Prajs1990} that any
homogeneous continuum in $\mR^{n+1}$ which contains an $n$-cube is
an $n$-manifold.

\begin{cor}\label{cor-codimen1}
Let $\fM$ be an   homogeneous, smooth $n$-dimensional matchbox manifold
which embeds in a closed orientable $(n+1)$-dimensional manifold. Then $\fM$
is   a manifold.
\end{cor}

The work \cite{ClarkHurder2011} by the authors studies the problem of finding smooth embeddings of solenoids into foliated manifolds with codimension $q \geq 2$. It is an open problem, in general, to determine the lowest codimension $q > 1$ in which a given solenoid can be embedded, either into a compact manifold, or as a minimal set for a $C^r$-foliation of a compact manifold.

The proof of the main theorem involves drawing an important connection between homogeneity and equicontinuity, based on the fundamental result of Effros that transitive continuous actions of Polish groups are micro-transitive \cite{Ancel1987, Effros1965,vanMill2004, vanMill2008}.
As a step in the proof of Theorem~\ref{thm-main1}, we show in Theorem~\ref{thm-equic}  that Effros' Theorem implies that  a homogeneous matchbox manifold is \emph{equicontinuous}. Combining the results  Theorem~\ref{thm-imasolenoid} and Proposition~\ref{prop-equicex}, we obtain:

\begin{thm}\label{thm-main2}
A smooth matchbox manifold $\fM$ is
homeomorphic to an $n$-dimensional solenoid if and only if $\fM$ is equicontinuous.
\end{thm}

Examples of equicontinuous smooth matchbox manifolds   which are not homogeneous are given in  Section~\ref{sec-examples}, showing  the results of Theorem~\ref{thm-main1} and Theorem~\ref{thm-main2} are optimal.

 There is an analogy between Theorem~\ref{thm-main1}, and the classification theory for Riemannian foliations \cite{Molino1988,MoerdijkMrcunbook2003}. Recall that a Riemannian foliation $\F$ on a compact manifold $M$ is said to be {\it transversally parallelizable} (or TP) if the group of foliation-preserving diffeomorphisms of $M$ acts transitively. In this case, the minimal sets for $\F$ are principle $H$-bundles, where $H$ is the structural Lie group of the foliation. Theorem~\ref{thm-main1} is the analog of this result for matchbox manifolds. It is interesting to compare this result with the theory of equicontinuous foliations on compact manifolds, as in \cite{ALC2009}.

 However, if $\fM$ is equicontinuous but not homogeneous, then the analogy becomes more tenuous. Clark, Fokkink and Lukina introduce in \cite{CFL2010} the Schreier continuum for   solenoids, an invariant of the topology of $\fM$,  which they use to calculate the end structures of leaves. In particular, they show that there exists   non-McCord solenoids for which the number of ends of leaves can be between 2 and infinity. It is not known if such behavior is possible  for Riemannian foliations which are not transversally parallelizable. (See \cite{Zhukova2010} for a discussion of ends of leaves in Riemannian foliations.)

We say that a matchbox manifold $\fM$ is a \emph{Cantor bundle}, if there exists a base manifold $M_0$ and a fibration $\pi_0 \colon \fM \to M_0$ so that for each $b \in M_0$ the fiber $\fF_b = \pi_0^{-1}(b)$ is a Cantor set.
The proofs of Theorem~\ref{thm-main1} and \ref{thm-main2} are much simpler, at least technically, if we assume that $\fM$ is a Cantor bundle.
In fact, in   \cite{Clark2002} the first author gave a proof of Theorem~\ref{thm-main1} in the case where $\fM$ is a Cantor bundle with base   an $n$-torus $\mT^n$, for $n > 1$. The technical simplifications are due to two properties of Cantor bundles, one is that for each $b \in M_0$ the fiber $\fF_b \equiv \pi_0^{-1}(b) \subset \fM$ is a transversal to the foliation $\F$ of $\fM$. The second simplification is that the local holonomy maps along leaves of $\F$ are the restrictions of global automorphisms of a fixed fiber  $\fF_0 \equiv \pi_0^{-1}(b_0)$. The extension of the arguments of  \cite{Clark2002} from the case of a base manifold $M_0 = \mT^n$ to an arbitrary compact base manifold $M_0$  involves few technical complications. In   the general case, the main technical  difficulties arise  due to the absence of given uniform transversals to the foliation $\F$ on $\fM$, and the consequent need for  uniform  estimates on the domains and dynamical behavior of the leafwise holonomy maps.

Section~\ref{sec-concepts} introduces the basic concepts of matchbox manifolds, and   in Section~\ref{sec-holonomy} the   holonomy maps and their properties are considered.
Properties of equicontinuous matchbox manifolds are developed in Section~\ref{sec-mme},
and properties of homogeneous matchbox manifolds in Section~\ref{sec-hmm}.

Section~\ref{sec-codes} begins the proof of Theorem~\ref{thm-main1} in ernest, as we develop the notion of the orbit coding for an equicontinuous matchbox manifold. This leads to a ``Borel'' version of the results of the main theorem.
Section~\ref{sec-tubes} shows how to obtain  the covering quotient maps associated to the Borel structures obtained in the previous section. The results of Section~\ref{sec-projections} depend  upon Theorem~\ref{thm-foliate}, which is fundamental for the analysis of the general case, but whose proof is quite technical and long, and thus relegated to the companion work \cite{CHL2011a}.

Finally, in Section~\ref{sec-mccord} we show that the solenoid structure obtained in Section~\ref{sec-projections} is a McCord solenoid with the additional hypothesis of homogeneity. This completes the proof of the Theorem~\ref{thm-main1}.

Section~\ref{sec-examples} gives   examples of matchbox manifolds which are equicontinuous but not homogeneous.
Section~\ref{sec-codim1} discusses an application of the main theorem to codimension one embeddings of solenoids.
Finally, Section~\ref{sec-remarks} discusses a selection of   open problems.

The second author would like to thank the Department of Mathematics at Leicester University for the warm hospitality given during a visit in February 2009, when this work was begun. The first author would like to thank the Department of Mathematics at the University of Illinois at Chicago for its generous hospitality for a visit in November 2010 that was vital to the completion of this project. Both authors are grateful to Robbert Fokkink for his comments on this work, and the hospitality provided by the University of Delft for hosting both authors during a mini-workshop in August 2009, and to our colleague Olga Lukina for her many helpful comments and questions. Finally, the authors are grateful to the referees for this paper, whose detailed reading(s) and comments greatly improved the exposition.

\section{Foliated spaces} \label{sec-concepts}

In this section, we discuss the basic concepts of foliated spaces. A more detailed discussion with examples can be found in \cite[Chapter 11]{CandelConlon2000} and \cite[Chapter 2]{MS2006}.
\begin{defn} \label{def-fs}
A continuum $\fM$ is a \emph{foliated space of dimension $n$} if there exists a compact separable metric space $\fX$, and
for each $x \in \fM$ there is a compact subset $\fT_x \subset \fX$, open subset $U_x \subset \fM$, and homeomorphism defined on its closure
$\vp_x \colon \oU_x \to [-1,1]^n \times \fT_x$ such that $\vp_x(x) = (0, w_x)$ where $w_x \in int(\fT_x)$. The
subspace $\fT_x$ of $\fX$ is called the \emph{local transverse model} at $x$.

Let $\pi_x \colon \oU_x \to \fT_x$ denote the composition of $\vp_x$ with projection onto the second factor.

For $w \in \fT_x$ the set $\cP_x(w) = \pi_x^{-1}(w) \subset \oU_x$ is called a \emph{plaque} for the coordinate chart $\vp_x$. We adopt the notation, for $z \in \oU_x$ that $\cP_x(z) = \cP_x(\pi_x(z))$, so that $z \in \cP_x(z)$. Note that each plaque $\cP_x(w)$ is given the topology so that the restriction $\vp_x \colon \cP_x(w) \to [-1,1]^n \times \{w\}$ is a homeomorphism. Then $int (\cP_x(w)) = \vp_x^{-1}((-1,1)^n \times \{w\})$.

Let $U_x = int (\oU_i) = \vp_x^{-1}((-1,1)^n \times int(\fT_x))$.
We require, in addition, that if $z \in U_x \cap U_y$, then $int(\cP_x(z)) \cap int( \cP_y(z))$ is an open subset of both
$\cP_x(z) $ and $\cP_y(z)$.
\end{defn}

The collection of sets
$$\cV = \{ \vp_x^{-1}(V \times \{w\}) \mid x \in \fM ~, ~ w \in \fT_x ~, ~ V \subset (-1,1)^n ~ {\rm open}\}$$
forms the basis for the \emph{fine topology} of $\fM$. The connected components of the fine topology are called leaves, and define the foliation $\F$ of $\fM$.
For $x \in \fM$, let $L_x \subset \fM$ denote the leaf of $\F$ containing $x$.

Note that in the above definition, the collection of transverse models
$\{\fT_x \mid x \in \fM\}$ need not have union equal to $\fX$. This is similar to the situation for a smooth foliation of codimension $q$, where each foliation chart projects to an open subset of $\mR^q$, but the collection of images need not cover $\mR^q$.

A \emph{smooth foliated space} is a foliated space $\fM$ as above, such that there exists a choice of local charts $\vp_x \colon \oU_x \to [-1,1]^n \times \fT_x$ such that for all $x,y \in \fM$ with $z \in U_x \cap U_y$, there exists an open set $z \in V_z \subset U_x \cap U_y$ such that $\cP_x(z) \cap V_z$ and $\cP_y(z) \cap V_z$ are connected open sets, and the composition
$$\psi_{x,y;z} \equiv \vp_y^{-1} \circ \vp_x \colon \vp_x(\cP_x (z) \cap V_z) \to \vp_y(\cP_y (z) \cap V_z)$$
is a smooth map, where $\vp_x(\cP_x (z) \cap V_z) \subset \mR^n \times \{w\} \cong \mR^n$ and $\vp_y(\cP_y (z) \cap V_z) \subset \mR^n \times \{w'\} \cong \mR^n$. Moreover, we require that the  maps $\psi_{x,y;z}$   depend continuously on $z$ in the $C^{\infty}$-topology on maps.

A closed saturated  subset $\fM \subset M$ of a smooth foliation $\F$ of a compact Riemannian manifold $M$ defines a smooth foliated space; but there are many other types of examples of smooth foliated spaces, as  discussed in \cite[Chapter 11]{CandelConlon2000} and also in this paper.

A map $f \colon \fM \to \mR$ is said to be \emph{smooth} if for each flow box
$\vp_x \colon \oU_x \to [-1,1]^n \times \fT_x$ and $w \in \fT_x$ the composition
$y \mapsto f \circ \vp_x^{-1}(y, w)$ is a smooth function of $y \in (-1,1)^n$, and depends continuously on $w$ in the $C^{\infty}$-topology on maps of the plaque coordinates $y$. As noted in \cite{MS2006} and \cite[Chapter 11]{CandelConlon2000}, this allows one to define smooth partitions of unity, vector bundles, and tensors for smooth foliated spaces. In particular, one can define leafwise Riemannian metrics. We recall a standard result, whose basic idea dates back to the work of Plante \cite{Plante1975b} if not before. The proof for foliated spaces can be found in \cite[Theorem~11.4.3]{CandelConlon2000}.
\begin{thm}\label{thm-riemannian}
Let $\fM$ be a smooth foliated space. Then there exists a leafwise Riemannian metric for $\F$, such that for each $x \in \fM$, $L_x$ inherits the structure of a complete Riemannian manifold with bounded geometry, and the Riemannian geometry depends continuously on $x$ . \hfill $\Box$
\end{thm}

In this paper, all foliated spaces are assumed to be smooth, equipped   with a   leafwise Riemannian metric as in Theorem~\ref{thm-riemannian}.

\begin{defn} \label{def-mm}
A \emph{matchbox manifold} is a continuum with the structure of a
smooth foliated space $\fM$, such that for each $x \in \fM$, the
transverse model space $\fT_x \subset \fX$ is totally disconnected.
\end{defn}

\subsection{Metric properties}\label{subsec-metricprops}

Bounded geometry on the leafwise metric for $\F$ implies  that for each $x \in \fM$, there is a leafwise exponential map
$\exp^{\F}_x \colon T_x\F \to L_x$ which is a surjection, and the composition $\iota \circ \exp^{\F}_x \colon T_x\F \to L_x \subset \fM$ depends continuously on $x$ in the compact-open topology.

The study of the dynamics of a foliated space $\fM$ requires
generalizing various concepts for flows, and group actions more generally,
 about the orbits of points in $\fM$,
 to the properties of   leaves $L$ of a foliation $\F$. On a
technical level, it is very useful in developing these
generalizations to have a strong local convexity property for the
leaves, generalizing the local convexity of the
orbit of a flow.

Another  nuance about the definition of foliated spaces, and matchbox manifolds in particular, is that for given $x \in \fM$,
the neighborhood $\oU_x$ in Definition~\ref{def-fs} need not be ``local''. As the transversal model $\fT_x$ need not be connected, the set $\oU_x$ need not be connected, and \emph{a priori} its connected components need not be contained in a metric ball around $x$.

The following technical procedures ensure that we can always choose the local charts for a matchbox manifold $\fM$ to satisfy strong  local convexity, as well as other metric regularity properties.

Let $d_{\fM} \colon \fM \times \fM \to [0,\infty)$ denote the metric on $\fM$, and $d_{\fX} \colon \fX \times \fX \to [0,\infty)$ the metric on $\fX$.

For $x \in \fM$ and $\e > 0$, let $D_{\fM}(x, \e) = \{ y \in \fM \mid d_{\fM}(x, y) \leq \e\}$ be the closed $\e$-ball about $x$ in $\fM$, and $B_{\fM}(x, \e) = \{ y \in \fM \mid d_{\fM}(x, y) < \e\}$ the open $\e$-ball about $x$.

Similarly, for $w \in \fX$ and $\e > 0$, let $D_{\fX}(w, \e) = \{ w' \in \fX \mid d_{\fX}(w, w') \leq \e\}$ be the closed $\e$-ball about $w$ in $\fX$, and $B_{\fX}(w, \e) = \{ w' \in \fX \mid d_{\fX}(w, w') < \e\}$ the open $\e$-ball about $w$.

Each leaf $L \subset \fM$ has a complete path-length metric induced from the leafwise Riemannian metric. That is, for $x,y \in L$ define
$$d_{\F}(x,y) = \inf \left\{\| \gamma\| \mid \gamma \colon [0,1] \to L ~ {\rm is } ~ C^1~, ~ \gamma(0) = x ~, ~ \gamma(1) = y  \right\}$$
and where $\| \gamma \|$ denotes the path length of the $C^1$-curve $\gamma(t)$. If $x,y \in \fM$   are not on the same leaf, then set $d_{\F}(x,y) = \infty$.

For each $x \in \fM$ and $r > 0$, let $D_{\F}(x, r) = \{y \in L_x \mid d_{\F}(x,y) \leq r\}$.
The Gauss Lemma implies that there exists $\lambda_x > 0$ such that $D_{\F}(x, \lambda_x)$ is a \emph{strongly convex} subset for the metric $d_{\F}$. That is, for any pair of points $y,y' \in D_{\F}(x, \lambda_x)$ there is a unique shortest geodesic segment in $L_x$ joining $y$ and $y'$ and it is contained in $D_{\F}(x, \lambda_x)$ (cf. \cite{BC1964}, \cite[Chapter 3, Proposition 4.2]{doCarmo1992}). Note then, that for all $0 < \lambda < \lambda_x$ the disk $D_{\F}(x, \lambda)$ is also strongly convex.

\begin{lemma}\label{lem-stronglyconvex}
There exists $\lF > 0$ such that for all $x \in \fM$, $D_{\F}(x, \lF)$ is strongly convex.
\end{lemma}
\proof $\fM$ is compact and the leafwise metrics have uniformly bounded geometry.
\endproof
If $\F$ is defined by a flow without periodic points, so
that every leaf is diffeomorphic to $\mR$, then the entire leaf is
strongly convex, so $\lF > 0$ can be chosen arbitrarily. For a
foliation with leaves of dimension $n > 1$, the constant $\lF$ must
be less than the injectivity radius for each of the leaves.

\subsection{Regular covers}\label{subsec-regcovers}
We next define  a ``regular covering'' of $\fM$ by foliation charts, which  is a finite collection of foliation charts which are well-adapted to the metrics $d_{\fM}$ on $\fM$ and $d_{\fX}$ on $\fX$, and   the leafwise metric $d_{\F}$. The definition is somewhat  technical, but this seems to be a necessary aspect of working with foliated spaces, as the usual metric properties of charts which hold for smooth foliations need not hold in general, and are replaced by the estimates imposed below on the charts.

\begin{lemma} \label{lem-conhull}
There exists $\eF > 0$ such that for all $x \in \fM$, there exists a
compact set $\oU' \subset \fM$ such that $D_{\fM}(x, 3 \eF) \subset
int( \oU')$, and for each leaf $L$ of $\F$, each connected component
of $L \cap \oU'$ is a strongly convex subset of $L$.
\end{lemma}
\proof
For each $x \in \fM$,
let $\vp_x \colon \oU_x \to [-1,1]^n \times \fT_x$ be a foliation chart with $\vp_x(x) = (0, w_x)$ as above. Then there exists $\e_x > 0$ such that $D_{\fM}(x, \e_x) \subset U_x$.
By the continuity of $\vp_x$ and the assumption that $w_x \in int(\fT_x)$, there exists $\e'_x > 0$
such that $D_{\fX}(w_x, \e'_x) \subset int(\fT_x)$ and
\begin{equation}\label{eq-localtrans}
\cT_x' \equiv \vp_x^{-1}\left(\{0\}\times D_{\fX}(w_x, \e'_x)\right) ~ \subset ~ B_{\fM}(x, \e_x) ~ .
\end{equation}

As $\cT_x'$ is compact and $B_{\fM}(x, \e_x)$ is open, there exists $0 < \delta_x' \leq \lF$ such that for each $y \in \cT_x'$
the strongly convex disk
$D_{\F}(y, \delta_x') \subset B_{\fM}(x, \e_x)$. Let
$$
\oU_x' = \bigcup_{y \in \cT_x'} ~ D_{\F}(y, \delta_x') ~ \subset ~ B_{\fM}(x, \e_x) ~ .
$$
The image $\vp_x(\oU_x') \subset [-1,1]^n \times \fT_x$ contains $(0, w_x)$ in its interior, so the collection
$\cU' = \{ \oU_x' \mid x \in \fM\}$ forms a covering. Let $\eU' > 0$ be a Lebesgue number for this covering.
Set $\eF = \eU'/3$.
\endproof

Next, introduce  coordinate charts with diameter bounded above by   $\eF$.
For each $x \in \fM$,
let $\vp_x \colon \oU_x \to [-1,1]^n \times \fT_x$ be a foliation chart with $\vp_x(x) = (0,  w_x)$ as above.
Then there exists $\e''_x > 0$
such that $D_{\fX}(w_x, \e''_x) \subset int(\fT_x)$ and
$$ \cT_x'' \equiv \vp_x^{-1}\left(\{0\}\times D_{\fX}(w_x, \e''_x)\right) ~ \subset ~ B_{\fM}(x, \eF) ~ .$$

As $\cT_x''$ is compact and $B_{\fM}(x, \eF)$ is open, there exists $0 < \delta_x \leq \lF/4$ such that for each $y \in \cT_x''$
the strongly convex disk
$D_{\F}(y, \delta_x) \subset B_{\fM}(x, \eF)$. Let
\begin{equation}\label{eq-prodcover}
\oU_x'' = \bigcup_{y \in \cT_x''} ~ D_{\F}(y, \delta_x) ~ \subset ~ B_{\fM}(x, \eF) ~ .
\end{equation}
The restriction of $\vp_x$ to $\oU_x''$ can be smoothly modified to $\vp_x''$ (for example, using the inverse of the leafwise exponential map followed by a smooth map from the $r_x$-ball in $T_xL_x$ to the unit cube) so that $\vp_x'' \colon \oU_x'' \to [-1,1]^n \times D_{\fX}(w_x, \e_x'')$ is a homeomorphism onto.

Replace $\fT_x$ with $D_{\fX}(w_x, \e_x'')$, $\oU_x$ with $\oU_x''$, and $\vp_x$ with $\vp_x''$.
Thus, for each $x \in \fM$, we can assume there are given $\e_x'', \delta_x > 0$, $w_x \in \fX$ and a foliation chart $\vp_x \colon \oU_x \to [-1,1]^n \times \fT_x$ such that $\oU_x \subset B_{\fM}(x, \eF)$, and the plaques of $\vp_x$ are leafwise strongly convex subsets with diameter $2\delta_x \leq \lF/2$.

The collection of open sets
$$ \{U_x \equiv int(\oU_x) = \vp_x^{-1}\left( (-1,1)^n \times B_{\fX}(w_x, \e''_x)\right) \mid x \in \fM \}$$
forms an open cover of the compact space $\fM$, so there exists a finite subcover ``centered'' at the points $\{x_1, \ldots , x_{\nu}\}$ where $\vp_{x_i}(x_i) = (0,  w_{x_i})$ for $w_{x_i} \in \fX$. Set
\begin{equation}\label{eq-Fdelta}
\dFU = \min \{\delta_{x_1} , \ldots , \delta_{x_{\nu}}\} ~ .
\end{equation}
Each open set $\oU_{x_j}$ can be covered by a finite collection of foliation charts of the form (\ref{eq-prodcover}) with leafwise radius $\dFU$. Thus, we can assume without loss that each $\oU_{x_i}$ is defined by (\ref{eq-prodcover}) where $\delta_{x_i} = \dFU$.
This covering by foliation coordinate charts will be fixed and used throughout, so we simplify notation.

For $1 \leq i \leq \nu$, set $\oU_i = \oU_{x_i}$, $U_i = U_{x_i}$, and $\e_i = \e_{x_i}''$.
Let $\cU = \{U_{1}, \ldots , U_{\nu}\}$ denote the corresponding open covering of $\fM$.
Then there are corresponding coordinate maps
$$
\vp_i = \vp_{x_i} \colon \oU_i \to [-1,1]^n \times \fT_i \quad , \quad
\pi_i = \pi_{x_i} \colon \oU_i \to \fT_i \quad , \quad \lambda_i \colon \oU_i \to [-1,1]^n ~ .
$$
For $z \in \oU_i$, the plaque of the chart $\vp_i$ through $z$ is denoted by $\cP_i(z) = \cP_i(\pi_i(z)) \subset \oU_i$. Note that the restriction $\lambda_i \colon \cP_i(z) \to [-1,1]^n$ is a homeomorphism onto.

Also, define sections
\begin{equation}\label{def-tau}
\tau_{i, \xi} \colon \fT_i \to \oU_i ~ , ~   \tau_{i, \xi}(w) = \vp_i^{-1}(\xi,  w)  ~; ~  \tau_i = \tau_{i,\vec{0}} ~.
\end{equation}
 Note that $\pi_i(\tau_{i, \xi}(w)) = w$.  Let $\cT_i$ denote the image of $\tau_i$ and  set $\cT = \cT_1 \cup \cdots \cup \cT_{\nu} \subset \fM$.

Let $\ds \fT_* = \fT_1 \cup \cdots \cup \fT_{\nu} \subset ~ \fX$; note that $\fT_*$ is compact, and if each $\fT_i$ is totally disconnected, then $\fT_*$ will also be totally disconnected.

\begin{defn}\label{def-regcover}
A \emph{regular covering} of a smooth foliated space $\fM$ is a covering by foliation charts satisfying   the above conditions: locality; that is, each $\oU_i \subset B_{\fM}(x_i , \eF)$,  and local convexity.
\end{defn}
We assume   that such a covering $\cU = \{\vp_i \colon \oU_i \to [-1,1]^n \times \fT_i \mid 1 \leq i \leq \nu\}$ of $\fM$ has been chosen.

If $\F$ is a smooth foliation of a compact manifold $M$, and $\fM \subset M$ is a closed saturated set, then the restriction to $\fM$ of a regular covering for $\F$ on $M$ (as defined for example in  \cite[Chapter 2]{CandelConlon2000}) provides a regular covering of the foliated space $\fM$ in the sense of Definition~\ref{def-regcover}.

\begin{lemma}\label{lem-regular} Suppose that $z \in U_i \cap U_j$ then $\cP_i(z) \cap \cP_j(z)$ is a strongly convex subset of $L_z$.
\end{lemma}
\proof
Our assumptions imply that $U_i \cup U_j$ has diameter at most $2\eF$ hence there exists $\whU \subset \fM$ as in Lemma~\ref{lem-conhull} such that each plaque $\cP_i(z)$ and $\cP_j(z)$ is contained in a strongly convex subset of $L_z \cap \whU$. As these sets intersect, they must be contained in the same connected component of $L_z \cap \whU$, which is strongly convex, and thus $\cP_i(z) \cap \cP_j(z)$ is also strongly convex.
\endproof

 Lemma~\ref{lem-regular} eliminates  the possibility that one of the charts $\oU_i$ might contain ``very long'' leaf segments, which could intersect another chart $\oU_j$ in more than one connected component.

\subsection{Local estimates}\label{subsec-locestimates}
We next introduce a number of constants based on the  choices made in Section~\ref{subsec-regcovers}, which will be used throughout the paper when making metric estimates.

Let $\eU > 0$ be a Lebesgue number for the covering $\cU$. That is, given any $z \in \fM$ there exists some index $1 \leq i_z \leq \nu$ such that the open metric ball $B_{\fM}(z, \eU) \subset U_{i_z}$.

The local projections $\pi_i \colon \oU_i \to \fT_i$ and sections
$\tau_{i, \xi} \colon \fT_i \to \oU_i$ are continuous maps of compact spaces, so admit uniform metric estimates as follows.
\begin{lemma}\label{lem-modpi}
There exists a continuous increasing function $\rp$ (the \emph{modulus of continuity} for the projections $\pi_i$) such that:
\begin{equation}\label{eq-modpi}
\forall ~ 1 \leq i \leq \nu ~ \text{and}~  x,y \in \oU_i \quad , \quad d_{\fM}(x,y) <\rp(\e) ~ \Longrightarrow ~ d_{\fX}(\pi_i(x), \pi_i(y)) < \e ~.
\end{equation}
\end{lemma}
\proof
Set $\ds\rp(\e) = \min \left\{\e, \min \left\{ d_{\fM}(x,y) \mid 1 \leq i \leq \nu ~ , ~ x,y \in \oU_i ~ , ~ d_{\fX}(\pi_i(x), \pi_i(y)) \geq \e\right\}\right\}$.
\endproof

\begin{lemma}\label{lem-modtau}
There exists a continuous   function $\rt$ (the \emph{modulus of continuity} for the sections $\tau_{i, \xi}$) such that:
\begin{equation}\label{eq-modtau}
\forall ~ \xi \in [-1,1]^n ~, ~ \forall ~ 1 \leq i \leq \nu ~, ~ \forall  w, w' \in \fT_i ~ , ~ d_{\fX}(w,w') <\rt(\e) ~ \Longrightarrow ~ d_{\fM}(\tau_{i,\xi}(w), \tau_{i, \xi}(w')) < \e ~ .
\end{equation}
\end{lemma}
\proof
Set $\ds\rt(\e) =   \min \left\{ d_{\fX}(w,w') \mid \xi \in [-1,1]^n  ~, ~ 1 \leq i \leq \nu ~ , ~ w,w' \in \fT_i ~ , ~ d_{\fM}(\tau_{i,\xi}(w), \tau_{i, \xi}(w')) \geq \e\right\}$, unless the set of points $(w,w')$ satisfying these restraints is empty, in which case we set $\rt(\e) = \e$.
\endproof

Finally, we introduce two additional constants, derived from the Lebesgue number $\eU$ chosen above.

The first is  derived from a ``converse'' to the modulus function $\rp$. Set:
\begin{equation}\label{eq-transdiam}
\eTU = \max \left\{\e \mid \forall ~ 1 \leq i \leq \nu, ~\forall ~ x \in \oU_i ~, ~ D_{\fM}(x, \eU/2) \subset \oU_i ~ , ~ D_{\fX}(\pi_i(x),\e) \subset \pi_i\left( D_{\fM}(x, \eU/2)\right)\right\}.
\end{equation}
Note that $\eTU \geq \rt(\eU/2 )$.

 Introduce a form of  ``leafwise Lebesgue number'',   defined by
\begin{equation}\label{eq-leafdiam}
\eFU(y) = \sup \left\{ \e \mid \forall ~ y \in \fM ~, ~ D_{\F}(y, \e) \subset D_{\fM}(y, \eU/8)\right\} ~ , ~ \eFU = \min \left\{ \eFU(y) \mid ~ \forall ~ y \in \fM \right\}.
\end{equation}
Thus, for all $y \in \fM$, $D_{\F}(y, \eFU) \subset D_{\fM}(y, \eU/8)$.
Note that for all $r > 0$ and $z' \in D_{\F}(z, \eFU)$, the triangle inequality implies that
$B_{\fM}(z', r) \subset B_{\fM}(z, r + \eU/8)$.

\section{Holonomy of foliated spaces} \label{sec-holonomy}

The holonomy pseudogroup   of a foliated manifold $(M, \F)$ generalizes the discrete cascade associated to a section of a flow. The holonomy pseudogroup for a matchbox manifold $(\fM, \F)$ is defined analogously, although there are delicate issues of domains which must be considered.

A pair of indices $(i,j)$, $1 \leq i,j \leq \nu$, is said to be \emph{admissible} if the \emph{open}  coordinate charts  satisfy $U_i \cap U_j \ne \emptyset$.
For $(i,j)$ admissible, define $\fD_{i,j} = \pi_i(U_i \cap U_j) \subset \fT_i \subset \fX$. Then the closure
$\overline{\fD_{i,j}} = \pi_i(\oU_i \cap \oU_j)$.
The hypotheses on foliation charts imply that plaques are either disjoint, or have connected intersection. This   implies that there is a well-defined homeomorphism $h_{j,i} \colon \fD_{i,j} \to \fD_{j,i}$ with domain $D(h_{j,i}) = \fD_{i,j}$ and range $R(h_{j,i}) = \fD_{j,i}$. The map $h_{j,i}$ admits a unique continuous extension to
$\overline{h}_{j,i} \colon \overline{\fD_{i,j}} \to \overline{\fD_{j,i}}$.

The maps $\cGF^{(1)} = \{h_{j,i} \mid (i,j) ~{\rm admissible}\}$ are the transverse change of coordinates defined by the foliation charts. By definition they satisfy $h_{i,i} = Id$, $h_{i,j}^{-1} = h_{j,i}$, and if $U_i \cap U_j\cap U_k \ne \emptyset$ then $h_{k,j} \circ h_{j,i} = h_{k,i}$ on their common domain of definition. The \emph{holonomy pseudogroup} $\cGF$ of $\F$ is the topological pseudogroup modeled on $\fX$ generated by compositions of the elements of $\cGF^{(1)}$.

A sequence $\cI = (i_0, i_1, \ldots , i_{\alpha})$ is \emph{admissible}, if each pair $(i_{\ell -1}, i_{\ell})$ is admissible for $1 \leq \ell \leq \alpha$, and the composition
\begin{equation}\label{eq-defholo}
 h_{\cI} = h_{i_{\alpha}, i_{\alpha-1}} \circ \cdots \circ h_{i_1, i_0}
\end{equation}
 has non-empty domain.
The domain $D(h_{\cI})$ is the \emph{maximal open subset} of $\fD_{i_0 , i_1} \subset \fT_{i_0}$ for which the compositions are defined.

Given any open subset $U \subset D(h_{\cI})$ we obtain a new element $h_{\cI} | U \in \cGF$ by restriction.  Introduce
\begin{equation}\label{eq-restrictedgroupoid}
\cGF^* = \left\{ h_{\cI} |  U \mid   \cI ~ {\rm admissible} ~ \& ~ U \subset D(h_{\cI}) \right\} \subset \cGF ~ .
\end{equation}
 The range of $g = h_{\cI} |  U$ is the open set $R(g) = h_{\cI}(U) \subset \fT_{i_{\alpha}} \subset \fX$. Note that each map $g \in \cGF^*$ admits a
continuous extension $\overline{g} \colon \overline{D(g)} = \overline{U} \to \fT_{i_{\alpha}}$.

We introduce the standard  notation for the orbits of the pseudogroup $\cGF$, where for $w \in \fX$, set
\begin{equation}\label{eq-orbits}
\cO(w) = \{g(w) \mid g \in \cGF^* ~, ~ w \in D(g) \} \subset \fT_* ~ .
\end{equation}

Given an admissible sequence $\cI  = (i_0, i_1, \ldots , i_{\alpha})$,   for each $0 \leq \ell \leq \alpha$,   set   $\cI_{\ell} = (i_0, i_1, \ldots , i_{\ell})$  and
\begin{equation}\label{eq-pcmaps}
h_{\cI_{\ell}} = h_{i_{\ell} , i_{\ell -1}} \circ \cdots \circ h_{i_{1} , i_{0}}~.
\end{equation}
Given $\xi \in D(h_{\cI})$ we adopt the notation $\xi_{\ell} = h_{\cI_{\ell}}(\xi) \in \fT_{i_{\ell}}$. So $\xi_0 = \xi$ and
$h_{\cI}(\xi) = \xi_{\alpha}$.

\medskip

Given $\xi \in D(h_{\cI})$,  let $x = x_0 = \tau_{i_0}(\xi_0) \in L_x$. Introduce the   \emph{plaque chain}
$$\cP_{\cI}(\xi) = \{\cP_{i_0}(\xi_0), \cP_{i_1}(\xi_1), \ldots , \cP_{i_{\alpha}}(\xi_{\alpha}) \} ~ .$$
 For each $0 \leq \ell < \alpha$, we have  $int(\cP_{i_{\ell}}(\xi_{\ell})) \cap int( \cP_{i_{\ell +1}}(\xi_{\ell+1})) \ne \emptyset$.
 Moreover,  each $\cP_{i_{\ell}}(\xi_{\ell})$ is a strongly convex subset of the   leaf $L_x$ in the leafwise metric $d_{\F}$.
 Recall that $\cP_{i_{\ell}}(x_{\ell}) = \cP_{i_{\ell}}(\xi_{\ell})$, so  we also adopt the notation $\cP_{\cI}(x) \equiv \cP_{\cI}(\xi)$.

  Intuitively, a plaque chain $\cP_{\cI}(\xi)$ is a sequence of successively overlapping convex ``tiles'' in $L_0$ starting at $x_0 = \tau_{i_0}(\xi_0)$, ending at
$x_{\alpha} = \tau_{i_{\alpha}}(\xi_{\alpha})$, and with each $\cP_{i_{\ell}}(\xi_{\ell})$ ``centered'' on the point $x_{\ell} = \tau_{i_{\ell}}(\xi_{\ell})$.

\medskip

\subsection{Leafwise path holonomy}
A \emph{leafwise path}  is a continuous map $\gamma \colon [0,1] \to \fM$ with image in some leaf $L$ of $\F$. The construction of the holonomy map $h_{\gamma}$ associated to a leafwise path $\gamma$ is a standard construction in foliation theory
(\cite{Reeb1952}, \cite{Haefliger1984}, \cite{CN1985}, \cite[Chapter 2]{CandelConlon2000}). We describe this in detail below, paying particular attention to domains and metric estimates.

Let $\cI$ be an admissible sequence. We say that $(\cI , w)$ \emph{covers} $\gamma$,
if  there exists a partition $0 = s_0 < s_1 < \cdots < s_{\alpha} = 1$ such that for the plaque chain
$\cP_{\cI}(w) = \{\cP_{i_0}(w_0), \cP_{i_1}(w_1), \ldots , \cP_{i_{\alpha}}(w_{\alpha}) \}$ we have
\begin{equation}\label{eq-cover}
\gamma([s_{\ell} , s_{\ell + 1}]) \subset int (\cP_{i_{\ell}}(w_{\ell}) )~ , ~ 0 \leq \ell < \alpha, \quad \& \quad \gamma(1) \in int( \cP_{i_{\alpha}}(w_{\alpha}))
\end{equation}
It follows that $w_0 = \pi_{i_0}(\gamma(0)) \in D(h_{\cI})$.

Now suppose we have two admissible sequences, $\cI = (i_0, i_1, \ldots, i_{\alpha})$ and
$\cJ = (j_0, j_1, \ldots, j_{\beta})$, such that both $(\cI, w)$ and $(\cJ, v)$ cover the leafwise path $\gamma \colon [0,1] \to \fM$.
Then
$$\gamma(0) \in int( \cP_{i_0}(w_0)) \cap int( \cP_{j_0}(v_0)) \quad , \quad \gamma(1) \in int(\cP_{i_{\alpha}}(w_{\alpha})) \cap int( \cP_{j_{\beta}}(v_{\beta})).$$
Thus both $(i_0 , j_0)$ and $(i_{\alpha} , j_{\beta})$ are admissible, and
$v_0 = h_{j_{0} , i_{0}}(w_0)$, $w_{\alpha} = h_{i_{\alpha} , j_{\beta}}(v_{\beta})$.

\medskip

\begin{prop}\label{prop-copc}
The maps $h_{\cI}$ and
$\ds h_{i_{\alpha} , j_{\beta}} \circ h_{\cJ} \circ h_{j_{0} , i_{0}}$
agree on their common domains.
\end{prop}
\proof
Let $\xi \in D(h_{\cI}) \cap D( h_{i_{\alpha} , j_{\beta}} \circ h_{\cJ} \circ h_{j_{0} , i_{0}})$.
Set $\xi' = h_{\cI}(\xi)$, $\zeta = h_{j_{0} , i_{0}}(\xi)$ and $\zeta' = h_{\cJ}(\zeta)$. We must show that
$\xi' = h_{i_{\alpha} , j_{\beta}}(\zeta')$.

Let $0 = s_0 < s_1 < \cdots < s_{\alpha} = 1$ and $0 = r_0 < r_1 < \cdots < r_{\beta} = 1$ be the partitions associated to $\cI$ and $\cJ$, respectively. The condition (\ref{eq-cover}) is open, so without loss of generality, we can assume that the two partitions have no points in common except endpoints.
Let $0 = t_0 < t_1 < \cdots < t_{\omega} = 1$ be the partition obtained by forming the common refinement of the two partitions: for each $\ell$, either $t_{\ell} = s_m$ for some $0 \leq m \leq \alpha$, or $t_{\ell} = r_{m'}$ for some $0 \leq m' \leq \beta$.

For each $0 \leq \ell \leq \omega$ we are given that $\gamma(t_{\ell}) \in U_{i_{m_{\ell}}} \cap U_{j_{m_{\ell}'}}$ where $m_{\ell}$ is the largest $m$ with $s_m \leq t_{\ell}$ and $m_{\ell}'$ is the largest $m'$ with $r_{m'} \leq t_{\ell}$.
Re-index the plaque chains $\cP_{\cI}(\xi)$ and $\cP_{\cJ}(\zeta)$ as follows:

Let $\xi_{\ell} = \xi_{m_{\ell}} = h_{\cI_{m_{\ell}}}(\xi)$,
so that $\cP_{i_{m_{\ell}}}(\xi_{\ell})$ denotes the plaque of $\oU_{i_{m_{\ell}}}$
corresponding to $\xi_{m_{\ell}}$.

Let $\zeta_{\ell} = \zeta_{m_{\ell}'} = h_{\cJ_{m_{\ell}'}}(\zeta)$,
so that $\cP_{j_{m_{\ell}'}}(\zeta_{\ell})$ denotes the plaque of $\oU_{j_{m_{\ell}'}}$
corresponding to $\zeta_{m_{\ell}'}$.

We inductively construct a plaque chain $\whcP = \{\whcP_0 , \whcP_1 , \ldots , \whcP_{\omega}\}$ which covers both plaque chains $\cP_{\cI}(\xi)$ and $\cP_{\cJ}(\zeta)$, so that
$\cP_{i_{\alpha}}(\xi_{\alpha}) \cup \cP_{j_{\beta}}(\zeta_{\beta}) \subset \whcP_{\omega}$ and thus
$\xi' = h_{i_{\alpha} , j_{\beta}}(\zeta')$.

For $\ell = 0$, the plaques $\cP_{\cI}(\xi_0) \cap \cP_{\cJ}(\zeta_0) \ne \emptyset$ as $\xi \in D(h_{j_{0} , i_{0}})$.
Thus, the diameter of the set $\oU_{i_0} \cup \oU_{j_0}$ is at most $2\eF$. By Lemma~\ref{lem-conhull}, there exists a coordinate chart $\whU_{0}$ such that
$\oU_{i_0} \cup \oU_{j_0} \subset int( \whU_{0})$.
Let $\whcP_0$ be the plaque of $\whU_0$ containing the connected set $\cP_{\cI}(\xi_0) \cup \cP_{\cJ}(\zeta_0)$.

Now proceed by induction. Assume that coordinate charts $\ds \left\{ \whU_0 , \whU_1 , \ldots , \whU_{k}\right\}$ have been chosen so that $\oU_{i_{\ell}} \cup \oU_{j_{\ell}} \subset int( \whU_{\ell})$ for $0 \leq \ell \leq k$, and
a plaque chain $\{\whcP_0 , \whcP_1 , \ldots , \whcP_{k}\}$ defined with
$\whcP_{\ell} \subset \whU_{\ell}$ and for $0 < \ell \leq k$,
$$ \cP_{i_{m_{\ell -1}}}(\xi_{\ell -1}) \cup \cP_{j_{m_{\ell -1}'}}(\zeta_{\ell -1}) \cup \cP_{i_{m_{\ell}}}(\xi_{\ell}) \cup \cP_{j_{m_{\ell}'}}(\zeta_{\ell}) \subset \whcP_{\ell} ~.$$

There are now two cases: either $m_{k} \ne m_{k+1}$ and $m_{k}' = m_{k+1}'$,
or $m_{k} = m_{k+1}$ and $m_{k}' \ne m_{k+1}'$. Consider the first case, so that
$\oU_{j_{m_{k}'}} = \oU_{j_{m_{k+1}'}}$ and $\cP_{i_{m_{k}}}(\xi_{k}) \cap \cP_{i_{m_{k+1}}}(\xi_{k+1}) \ne \emptyset$.
We also have $\gamma(t_{k +1}) \in U_{i_{m_{k +1}}} \cap U_{j_{m_{k +1}'}}$, from which it follows that the union $\ds \{ \oU_{i_{k}} \cup \oU_{j_{k}} \cup \oU_{i_{k+1}} \cup \oU_{j_{k+1}} \}$ is connected with
diameter at most $3 \eF$. By Lemma~\ref{lem-conhull}, there exists a coordinate chart $\whU_{k+1}$ containing the union in its interior. Let $\whcP_{k+1}$ be the plaque of $\whU_{k+1}$ containing the connected set
$$ \cP_{i_{m_{k}}}(\xi_{k}) \cup \cP_{j_{m_{k}'}}(\zeta_{k}) \cup \cP_{i_{m_{k+1}}}(\xi_{k+1}) \cup \cP_{j_{m_{k+1}'}}(\zeta_{k+1}) \subset \whcP_{k+1} ~ .$$
This completes the induction.
The resulting plaque chain $\{\whcP_0 , \whcP_1 , \ldots , \whcP_{\omega}\}$ thus covers both plaque chains $\cP_{\cI}(\xi)$ and $\cP_{\cJ}(\zeta)$.
In particular,
$$\cP_{i_{\alpha}}(\xi_{\alpha}) \cup \cP_{j_{\beta}}(\zeta_{\beta}) = \cP_{\omega} \cap \left( \oU_{i_{\alpha}} \cup \oU_{j_{\beta}}\right) \Longrightarrow \cP_{i_{\alpha}}(\xi_{\alpha}) \cap \cP_{j_{\beta}}(\zeta_{\beta}) = \cP_{\omega} \cap \left( \oU_{i_{\alpha}} \cap \oU_{j_{\beta}}\right) \ne \emptyset ~ .$$
Now,
$\xi' = \pi_{i_{\alpha}}(\cP_{i_{\alpha}}(\xi_{\alpha}))$ and
$\zeta' = \pi_{j_{\beta}}( \cP_{j_{\beta}}(\zeta_{\beta}) )$ so $\xi' = h_{i_{\alpha} , j_{\beta}}(\zeta')$ follows.
\endproof

The interested reader can compare the above argument to the proof of \cite[Proposition~2.3.2]{CandelConlon2000} where it is shown that the germinal holonomy along a path is well-defined. The two proofs are essentially the same; yet a detailed proof of Proposition~\ref{prop-copc} is included, as the study of equicontinuous maps depends fundamentally on having equality on   domains of fixed size, and not just germinal equality.

\medskip

\subsection{Admissible sequences}\label{subsec-admissible}
Given a leafwise path $\gamma \colon [0,1] \to \fM$, we next construct an admissible sequence
$\cI = (i_0, i_1, \ldots, i_{\alpha})$ with $w \in D(h_{\cI})$ so that $(\cI , w)$ covers $\gamma$, and has ``uniform domains''.

Inductively, choose a partition of the interval $[0,1]$, $0 = s_0 < s_1 < \cdots < s_{\alpha} = 1$ such that for each $0 \leq \ell \leq \alpha$,
$\gamma([s_{\ell}, s_{\ell + 1}]) \subset D_{\F}(x_{\ell}, \eFU)$ where $x_{\ell} = \gamma(s_{\ell})$. As a notational convenience,  we have let
$s_{\alpha+1} = s_{\alpha}$, so that $\gamma([s_{\alpha}, s_{\alpha + 1}]) = x_{\alpha}$.

For each $0 \leq \ell \leq \alpha$, choose an index $1 \leq i_{\ell} \leq \nu$ so that $ B_{\fM}(x_{\ell}, \eU) \subset U_{i_{\ell}}$.
Note that, for all $s_{\ell} \leq t \leq s_{\ell +1}$, $B_{\fM}(\gamma(t), \eU/2) \subset U_{i_{\ell}}$, so that
$x_{\ell+1} \in U_{i_{\ell}} \cap U_{i_{\ell +1}}$. It follows that $\cI_{\gamma} = (i_0, i_1, \ldots, i_{\alpha})$ is an admissible sequence.
Set $h_{\gamma} = h_{\cI_{\gamma}}$. Then $h_{\gamma}(w) = w'$, where $w = \pi_{i_0}(x_0)$ and
$w' = \pi_{i_{\alpha}}(x_{\alpha})$.

The construction of the admissible sequence $\cI_{\gamma}$ above has an important special property.
For $0 \leq \ell < \alpha$, note that $x_{\ell+1} \in D_{\F}(x_{\ell +1}, \eFU)$ implies that for some $s_{\ell} < s_{\ell + 1}' < s_{\ell + 1}$, we have that
$\gamma([s_{\ell + 1}', s_{\ell + 1}]) \subset D_{\F}(x_{\ell +1}, \eFU)$. Hence,
\begin{equation} \label{eq-unifest}
B_{\fM}(\gamma(t), \eU/2) \subset U_{i_{\ell}} \cap U_{i_{\ell +1}} ~ , ~ {\rm for ~ all} ~ s_{\ell + 1}' \leq t \leq s_{\ell + 1} ~.
\end{equation}
Then  for all $s_{\ell + 1}' \leq t \leq s_{\ell + 1}$, the uniform estimate defining $\eTU > 0$ in (\ref{eq-transdiam}) implies that
\begin{equation} \label{eq-domains}
B_{\fX}(\pi_{i_{\ell}}(\gamma(t)), \eTU ) \subset \fD_{i_{\ell} , i_{\ell +1}} \quad \& \quad
B_{\fX}(\pi_{i_{\ell +1}}(\gamma(t)), \eTU ) \subset \fD_{i_{\ell +1} , i_{\ell}} ~ .
\end{equation}
 For the admissible sequence $\cI_{\gamma} = (i_0, i_1, \ldots, i_{\alpha})$,
    recall that $x_{\ell} = \gamma(s_{\ell})$ and we set $w_{\ell} = \pi_{i_{\ell}}(x_{\ell})$.
 Then by the definition (\ref{eq-defholo}) of $\ds h_{\cI_{\gamma}}$ the condition (\ref{eq-domains}) implies that
$D_{\fX}(w_{\ell} , \eTU) \subset D(h_{\ell})$.

That is, $h_{\cI_{\gamma}}$ is the composition of generators of $\cGF^*$ which have uniform estimates on the radii of the metric balls contained in their domains, where $ \eTU$ is independent of $\gamma$.

\medskip

There is a converse to the above construction, which associates to an admissible sequence a leafwise path.
Let $\cI = (i_0, i_1, \ldots, i_{\alpha})$ be admissible, with corresponding holonomy map $h_{\cI}$,
and choose $w \in D(h_{\cI})$ with $x = \tau_{i_0}(w)$.

For each $1 \leq \ell \leq \alpha$, recall that
$\cI_{\ell} = (i_0, i_1, \ldots, i_{\ell})$, and let $h_{\cI_{\ell}}$ denote the corresponding holonomy map. For $\ell = 0$, let $\cI_0 = (i_0 , i_0)$.
Note that $h_{\cI_{\alpha}} = h_{\cI}$ and $h_{\cI_{0}} = Id \colon \fT_0 \to \fT_0$.

For each $0 \leq \ell \leq \alpha$, set $w_{\ell} = h_{\cI_{\ell}}(w)$ and
$x_{\ell}= \tau_{i_{\ell}}(w_{\ell})$. By assumption, for $\ell > 0$, there exists $z_{\ell} \in \cP_{\ell -1}(w_{\ell -1}) \cap \cP_{\ell}(w_{\ell})$.

Let
$\gamma_{\ell} \colon [(\ell -1)/\alpha , \ell / \alpha] \to L_{x_0}$ be the leafwise piecewise geodesic segment from $x_{\ell -1}$ to $z_{\ell}$ to $x_{\ell}$. Define the leafwise path $\gamma^x_{\cI} \colon [0,1] \to L_{x_0}$ from $x_0$ to $x_{\alpha}$ to be the concatenation of these paths.
If we then cover $\gamma^x_{\cI}$ by the charts determined by the given admissible sequence $\cI$, it follows that $h_{\cI} = h_{\gamma^x_{\cI}}$.

Thus,  given an admissible sequence $\cI = (i_0, i_1, \ldots, i_{\alpha})$ and $w \in D(h_{\cI})$ with $w' = h_{\cI}(w)$, the choices above determine an initial chart $\vp_{i_0}$ with ``starting point''
$x = \tau_{i_0}(w) \in U_{i_0} \subset \fM$. Similarly, there is a terminal chart $\vp_{i_{\alpha}}$ with
  ``terminal point'' $x' = \tau_{i_{\alpha}}(w') \in U_{i_{\alpha}} \subset \fM$. The leafwise path $\gamma^x_{\cI}$ constructed above starts at $x$, ends at $x'$, and has image contained in the plaque chain $\cP_{\cI}(x)$.

On the other hand, if we start with a leafwise path $\gamma \colon [0,1] \to \fM$, then the initial point $x = \gamma(a)$ and the terminal point $x' = \gamma(b)$ are both well-defined. However, there need not be a unique index $j_0$ such that $x \in U_{j_0}$ and similarly for the index $j_{\beta}$ such that $x' \in U_{j_{\beta}}$. Thus, when one constructs an admissible sequence $\cJ = (j_0, \ldots , j_{\beta})$ from $\gamma$, the initial and terminal charts need not be well-defined. This was observed already in the proof of Proposition~\ref{prop-copc}, which proved that
$$h_{\cI} | U = h_{i_{\alpha} , j_{\beta}} \circ h_{\cJ} \circ h_{j_{0} , i_{0}} | U \quad {\rm for} \quad U = D(h_{\cI}) \cap D(h_{i_{\alpha} , j_{\beta}} \circ h_{\cJ} \circ h_{j_{0} , i_{0}}) ~ .$$
We introduce the following definition, which gives a uniform estimate of the effect of this ambiguity.
\begin{lemma}\label{lem-modgen}
There exists a continuous function $\kappa \colon (0,\infty)  \to (0,\infty)$   such that for all admissible $(i,j)$ there is a uniform estimate:
\begin{equation}\label{eq-modgen}
d_{\fX}(h_{j,i}(w) , h_{j,i}(w')) \leq \kappa(r)  ~ {\rm for ~ all} ~ w , w' \in \fD_{i,j} ~ ~ {\rm with }~ d_{\fX}(w,w') \leq r ~.
\end{equation}
Moreover, $\ds \lim_{r \to 0} ~  \kappa(r) = 0$.
\end{lemma}
\proof
For $(i,j)$ admissible, the holonomy map $h_{j,i}$ extends to a homeomorphism of the closure of its domain,
$\overline{h}_{j,i} \colon \overline{\fD_{i,j}} \to \overline{\fD_{j,i}}$. Thus, for $r > 0$, the product map $\overline{h}_{j,i} \times \overline{h}_{j,i}$ is continuous on the compact set
$\overline{\fB_{i,j}^r} = \{ (w,w') \mid w,w' \in \overline{\fD_{i,j}} ~ , ~ d_{\fX}(w,w') \leq r\}$, hence we obtain a finite upper bound
\begin{equation}\label{eq-modgendef}
\kappa(r) = \max \left\{  d_{\fX}(h_{j,i}(w) , h_{j,i}(w'))  \mid (i,j) ~ {\rm admissible} ~, ~ (w,w') \in\overline{\fB_{i,j}^r}\right\}~ .
\end{equation}
  Note that $\ds \lim_{r \to 0} ~   \kappa(r) = 0$ follows  from continuity of the maps  $h_{i,j}$.
\endproof

We conclude this discussion with a useful observation which yields a   key technical point,   that the holonomy along a path is independent of ``small deformations'' of  the path. First, we recall a standard definition:

Let $h \colon U \to V$ be a homeomorphism, where $U, V \subset \fT_*$ are open subsets,  and let $w \in U$. Given a second homeomorphism
$h' \colon U' \to V'$ be a homeomorphism, where $U', V' \subset \fT_*$ are also open subsets,  with $w \in U'$.
Then define an equivalence relation, where $h \sim h'$ if there exists an open set $w \in V \subset U \cap U'$ such that $h | V= h' |V$.
\begin{defn}\label{def-germ}
The \emph{germ of $h$ at $w$} is the equivalence class $[h]_w$ under this relation, which is also called the \emph{germinal class of $h$ at $w$}. The  map  $h \colon U \to V$ is called a \emph{representative} of  $[h]_w$.
The point $w$ is called the source of  $[h]_w$ and denoted $s([h]_w)$, while $w' = h(w)$ is called the range of  $[h]_w$ and denoted $r([h]_w)$.
\end{defn}

  Let $\cI = (i_0, i_1, \ldots, i_{\alpha})$ be admissible, with associated holonomy map $h_{\cI}$.  Given   $w, u \in D(h_{\cI})$, then   the germs of $h_{\cI}$ at $w$ and $u$   admit a common representative, namely $h_{\cI}$. Thus, if $\gamma$, $\gamma'$ are leafwise paths defined as above from the plaque chains associated to $(\cI, w)$ and $(\cI,u)$ then the germinal holonomy maps along $\gamma$ and $\gamma'$ admit a common representative by Proposition~\ref{prop-copc}. This is the basic idea behind the following technically useful result.

 \begin{lemma} \label{lem-domainconst}
 Let  $\gamma, \gamma' \colon [0,1] \to \fM$ be leafwise paths. Suppose that $x = \gamma(0), x' = \gamma'(0) \in U_i$ and
  $y = \gamma(1), y' = \gamma'(1) \in U_j$. If $d_{\fM}(\gamma(t) , \gamma'(t)) \leq \eU/4$ for all $0 \leq t \leq 1$, then the induced holonomy maps $h_{\gamma}, h_{\gamma'}$ agree on their common domain $D(h_{\gamma}) \cap D(h_{\gamma'}) \subset \fT_i$.
\end{lemma}
\proof  Choose a partition of the interval $[0,1]$, $0 = s_0 < s_1 < \cdots < s_{\alpha} = 1$ such that for each $0 \leq \ell \leq \alpha$, both paths satisfy the conditions
$$ d_{\fM}(\gamma(s_{\ell}) , \gamma(s_{\ell +1})) < \eFU \quad , \quad  d_{\fM}(\gamma'(s_{\ell} ) , \gamma'(s_{\ell +1})) < \eFU ~ .$$
Set $x_{\ell} = \gamma(s_{\ell})$  and $x'_{\ell} = \gamma'(s_{\ell})$ for   $0 \leq \ell \leq \alpha$, and where for  notational convenience,  we   let   $s_{\alpha+1} = s_{\alpha}$.

Then note that for all $s_{\alpha -1} \leq   t' \leq s_{\alpha +1}$ we have
\begin{equation}\label{eq-pathests}
d_{\fM}(\gamma(s_{\alpha}) , \gamma'(t')) \leq
d_{\fM}(\gamma(s_{\alpha}) , \gamma'(s_{\alpha})) + d_{\fM}(\gamma'(s_{\alpha}) , \gamma'(t')) \leq \eU/4 + \eU/8 < \eU/2 ~ .
\end{equation}

For each $0 \leq \ell \leq \alpha$, choose an index $1 \leq i_{\ell} \leq \nu$ so that $ B_{\fM}(x_{\ell}, \eU) \subset U_{i_{\ell}}$.

Then   for all $s_{\ell} \leq t \leq s_{\ell +1}$, $B_{\fM}(\gamma(t), \eU/2) \subset U_{i_{\ell}}$, so that
$x_{\ell+1} \in U_{i_{\ell}} \cap U_{i_{\ell +1}}$. It follows that $\cI = (i_0, i_1, \ldots, i_{\alpha})$ is an admissible sequence. Set $h_{\gamma} = h_{\cI}$.

Also, by (\ref{eq-pathests}) we have $\gamma'([s_{\ell -1}, s_{\ell +1}]) \subset B_{\fM}(\gamma(s_{\ell}), \eU/2)$ so that $\cI$ is also an admissible sequence defining $h_{\gamma'} = h_{\cI}$.
Thus, $x,x' \in D(h_{\cI}) \subset \fT_{i_0}$.

As the domains for $h_{\gamma}$ and $h_{\gamma'}$ are defined to be the maximal subsets of $ \fT_{i_0}$ where the maps are defined, this shows they agree on the subset $ D(h_{\cI}) \subset D(h_{\gamma}) \cap D(h_{\gamma'})$,  so we are done by Proposition~\ref{prop-copc}.
 \endproof

 \medskip

\subsection{Homotopy independence}
Two leafwise paths $\gamma , \gamma' \colon [0,1] \to \fM$ are homotopic if there exists a family of leafwise paths $\gamma_s \colon [0,1] \to \fM$ with $\gamma_0 = \gamma$ and $\gamma_1 = \gamma'$. We are most interested in the special case when $\gamma(0) = \gamma'(0) = x$ and $\gamma(1) = \gamma'(1) = y$. Then $\gamma$ and $\gamma'$ are \emph{endpoint-homotopic}
if they are homotopic with $\gamma_s(0) = x$ for all $0 \leq s \leq 1$, and similarly
$\gamma_s(1) = y$ for all $0 \leq s \leq 1$. Thus, the family of curves $\{ \gamma_s(t) \mid 0 \leq s \leq 1\}$ are all contained in a common leaf $L_{x}$. The following property then follows from an  inductive application of Lemma~\ref{lem-domainconst}:

\begin{lemma}\label{lem-homotopic}
Let $\gamma, \gamma' \colon [0,1] \to \fM$ be endpoint-homotopic leafwise paths. Then their holonomy maps $h_{\gamma}$ and $h_{\gamma'}$ agree on some open subset $U \subset D(h_{\gamma}) \cap D(h_{\gamma'}) \subset \fT_*$. In particular, they determine the same germinal holonomy maps.
\end{lemma}
\proof
Let $H(s,t) \colon [0,1] \times [0,1] \to \fM$ with $H(0,t) = \gamma(t)$, $H(1,t) = \gamma'(t)$, and $H(s,0) = \gamma(0)$, $H(s,1) = \gamma(1)$.
Choose a partition of the interval $[0,1]$, $0 = t_0 < t_1 < \cdots < t_{\alpha} = 1$ such that for all $0 \leq s \leq 1$ the leafwise distance estimate holds,
$$ d_{\fM}(H(s, t_{\ell}) , H(s, t_{\ell +1})) < \eFU ~ .$$
Then choose a partition $0 = s_0 < s_1 < \cdots <s_{\beta} = 1$ so that for all $0 \leq \ell < \beta$ we have the uniform estimate
$$d_{\fM}(H(s_{\ell}, t) , H(s_{\ell +1}, t)) \leq \eU/4 \quad {\rm for~ all} \quad 0 \leq t \leq 1 ~ .$$
Then apply Lemma~\ref{lem-domainconst} inductively to the paths $t \mapsto H(s_{\ell} , t)$ and $t \mapsto H(s_{\ell +1} , t)$ for $0 \leq \ell < \beta$, and the conclusion follows.
\endproof
The following is another consequence of the total convexity of the plaques in the foliation covering:
\begin{lemma} \label{lem-close}
Suppose that $\gamma, \gamma' \colon [0,1] \to \fM$ are leafwise paths for which $\gamma(0) = \gamma'(0) = x$ and $\gamma(1) = \gamma'(1) = x'$, and suppose that $d_{\fM}(\gamma(t), \gamma'(t)) < \eU/4$ for all $a \leq t \leq b$.
Then $\gamma, \gamma' \colon [0,1] \to \fM$ are endpoint-homotopic. \hfill $\Box$
\end{lemma}
Given $g \in \cGF^*$ and $w \in D(g)$, let $[g]_w$ denote the germ of the map $g$ at $w$. Set
\begin{equation}\label{eq-holodef}
\G_{\F,w} = \{ [g]_w \mid g \in \cGF^* ~ , ~ w \in D(g) ~ , ~ g(w) = w\} ~ .
\end{equation}
Given $x \in U_i$ with $w = \pi_i(x) \in \fT_*$, the elements of $\G_{\F,w}$ form a group, and by Lemma~\ref{lem-homotopic} there is a well-defined homomorphism $h_{\F,x} \colon \pi_1(L_x , x) \to \G_{\F,w}$ which is called the \emph{holonomy group} of $\F$ at $x$.

\subsection{Non-trivial   holonomy}
Note that if $y \in L_x$ then the homomorphism
$h_{\F , y}$ is conjugate (by   an element of $\cGF^*$) to the homomorphism $h_{\F , x}$.
A leaf $L$ is said to have \emph{non-trivial germinal holonomy} if for some $x \in L$, the homomorphism $h_{\F , x}$ is non-trivial. If the homomorphism $h_{\F , x}$ is trivial, then we say that $L_x$ is a \emph{leaf without holonomy}. This property depends only on $L$, and not the basepoint $x \in L$.
The foliated space $\fM$ is said to be \emph{without holonomy} if for every $x \in \fM$,  the leaf $L_x$ is   without germinal holonomy.
\begin{lemma} \label{lem-noholo}
Let $\fM$ be a foliated space without holonomy. Fix a regular covering for $\fM$ as above. Let $\cI$, $\cJ$ be two plaques chains such that $w \in Dom(h_{\cI}) \cap Dom(h_{\cJ})$ with $h_{\cI}(w) = w' = h_{\cJ}(w)$. Then $h_{\cI}$ and $h_{\cJ}$ have the same germinal holonomy at $w$. Thus, for each $w' \in \cO(w)$ in the  $\cGF^*$ orbit of $w$,
there is a well-defined holonomy germ $h_{w,w'}$.
\end{lemma}
\proof
The composition $g = h_{\cJ}^{-1} \circ h_{\cI}$ satisfies $g(w) = w$, so by assumption there is some open neighborhood $w \in U$ for which $g | U$ is the trivial map. That is, $ h_{\cI} | U = h_{\cJ} | U$.
\endproof
We introduce a mild generalization of the notion of a foliation without holonomy.
\begin{defn}\label{def-finiteholo}
The foliated space $\fM$ is said to have \emph{finite holonomy} if there is a (compact) foliated space without holonomy $\widetilde{\fM}$ with foliation $\wtF$, and a finite-to-one foliated map $\Pi \colon \widetilde{\fM} \to \fM$ which is a surjection, and the restrictions of $\Pi$ to leaves of $\wtF$ are covering maps onto leaves of $\F$.
\end{defn}

Finally, we recall a basic result of Epstein, Millet and Tischler \cite{EMT1977} for foliated manifolds, whose proof applies verbatim in the case of foliated spaces.
\begin{thm} \label{thm-emt}
The union of all leaves without holonomy in a foliated space $\fM$ is a dense $G_{\delta}$ subset of $\fM$. In particular, there exists at least one   leaf without germinal holonomy. \hfill $\Box$
\end{thm}

\section{Matchbox manifolds and equicontinuity}\label{sec-mme}

Let  $\fM$ be a matchbox manifold. Then   the local transverse models for $\F$ are totally disconnected, and the leaves of $\F$ are defined to be the path components for the induced fine topology on $\fM$. These remarks are the basis for several elementary but important observations.
\begin{lemma}\label{lem-pc}
Every continuous map $\gamma \colon [0,1] \to \fM$ is a leafwise path.
\end{lemma}
\proof
Let $a \leq c \leq b$ and choose a local chart $\vp_i \colon U_i \to (-1,1)^n \times \fT_i$ with $\gamma(c) \in U_i$. The image path $\pi_i(\gamma(t)) \in \fT_i$ must be constant for $t$ near to $c$, as $\fT_i$ is assumed to be totally disconnected. Thus, by standard arguments, $\gamma(t)$ lies in the leaf $L_{x}$ of $\F$ containing the initial point $x = \gamma(a)$.
\endproof
\begin{cor} \label{cor-cpc}
Let   $X$ be a path connected
topological space, and $h \colon X \to \fM$ a continuous map. Then
there exists a leaf $L_h \subset \fM$ for which $h(X) \subset L_h$.
\end{cor}
\proof
Let $x \in X$ and $L_h$ be the leaf of $\F$ containing $h(x)$. Then apply Lemma~\ref{lem-pc}.
\endproof
\begin{cor} \label{cor-foliated1}
Let $\fM$ and $\fM'$ be matchbox manifolds, and $h \colon \fM' \to \fM$ a continuous map. Then $h$ maps the leaves of $\F'$ to leaves of $\F$.
\end{cor}
\proof
The leaves of $\F'$ are path-connected, so their images under $h$ are contained in leaves of $\F$.
\endproof
\begin{cor} \label{cor-foliated2}
A homeomorphism $h \colon \fM \to \fM$ of a matchbox manifold is a foliated map. \hfill $\Box$
\end{cor}
Let $\cH(\fM)$ denote the group of homeomorphisms of $\fM$, and $\cH(\fM, \F)$ the subgroup of $\cH(\fM)$ consisting of homeomorphisms which preserve the foliation $\F$; that is, every leaf of $\F$ is mapped to some leaf of $\F$.
Then Corollary~\ref{cor-foliated2} states that $\cH(\fM) = \cH(\fM, \F)$.

\medskip

\subsection{Equicontinuous pseudogroups} The following is  one of the main concepts used in this work.

\begin{defn} \label{def-equicontinuous}
The holonomy pseudogroup $\cGF$ of $\F$ is \emph{equicontinuous} if for all $\epsilon > 0$, there exists $\delta > 0$ such that for all $g \in \cGF^*$, if $w, w' \in D(g)$ and $d_{\fX}(w,w') < \delta$, then $d_{\fX}(g(w), g(w')) < \epsilon$.
\end{defn}

\medskip

We note that equicontinuity is a strong hypothesis on a pseudogroup. In particular, as noted by Plante \cite[Theorem 3.1]{Plante1975a}, Sacksteder proved:
\begin{thm}[Sacksteder \cite{Sacksteder1965}] \label{thm-sacksteder} If $\cGF$ is an equicontinuous pseudogroup modeled on a compact Polish space $X$, then there exists a Borel probability measure $\mu$ on $X$ which is $\cGF$-invariant. \hfill $\Box$
\end{thm}

We also introduce the notion of a \emph{distal} pseudogroup. While not used directly in this work, we refer to this in discussing open questions in Section~\ref{sec-remarks}.
\begin{defn} \label{def-distal}
The holonomy pseudogroup $\cGF$ of $\F$ is \emph{distal} if for all $w, w' \in \fT_*$, if $w \ne w'$ then there exists $\delta_{w,w'} > 0$ such that for all $g \in \cGF^*$ with  $w,  w' \in D(g)$, then $d_{\fX}(g(w), g(w')) \geq \delta_{w,w'}$.
\end{defn}
Distal and equicontinuous pseudogroups are closely related \cite{AAG2008,EEN2001,Furstenberg1963,Lindenstrauss1999,Veech1970}.

\medskip
We next prove the fundamental result, that the equicontinuity hypothesis on $\cGF$   gives uniform control over the domains of arbitrary compositions of the generators of $\cGF^*$.
\begin{prop}\label{prop-uniformdom}
Assume the holonomy pseudogroup $\cGF$ of $\F$ is equicontinuous.  Then there exists $\dTU > 0$ such that for every leafwise path
$\gamma \colon [0,1] \to \fM$, there is a corresponding admissible sequence $\cI_{\gamma} = (i_0 , i_1 , \ldots , i_{\alpha})$ so that
$B_{\fX}(w_0, \dTU) \subset D(h_{\cI_{\gamma}})$, where $x = \gamma(0)$ and $w_0 = \pi_{i_0}(x)$.

Moreover, for all $0 < \e_1 \leq \eTU$ there exists $0 < \delta_1 \leq \dTU$ independent of the path $\gamma$, such that
$h_{\cI_{\gamma}}(D_{\fX}(w_0, \delta_1)) \subset D_{\fX}(w', \e_1)$ where $w' = \pi_{i_{\alpha}}(\gamma(1))$.

Thus,  $\cGF^*$ is equicontinuous as a family of local group actions.
\end{prop}
\proof
Recall that $\eTU> 0$ is defined by (\ref{eq-transdiam}).
Let $\dTU > 0$ be the modulus associated to $\e = \eTU$ by Definition~\ref{def-equicontinuous}. Note that $\dTU \leq \eTU$ as $\cGF^*$ contains the identity map for every open subset of $\fT_*$.

Given $\gamma \colon [0,1] \to \fM$, let $0 = s_0 < s_1 < \cdots < s_{\alpha} = 1$ be a partition of the interval $[0,1]$ as in Section~\ref{sec-holonomy},
so that for each $0 \leq \ell \leq \alpha$,
$\gamma([s_{\ell}, s_{\ell + 1}]) \subset D_{\F}(x_{\ell}, \eFU)$ where $x_{\ell} = \gamma(s_{\ell})$. Moreover, for the associated admissible sequence $\cI_{\gamma} = (i_0, i_1, \ldots, i_{\alpha})$, we have that
for all $0 \leq t \leq 1$, $B_{\fM}(\gamma(t), \frac{1}{2} \eU) \subset U_{i_{\ell}}$.

For each $1 \leq \ell \leq \alpha'$, set $x_{\ell} = \gamma(s_{\ell})$ and let $w_{\ell} = \pi_{i_{\ell}}(x_{\ell})$.
Set $\cI_{\ell} = (i_0, i_1, \ldots, i_{\ell})$ with corresponding holonomy map $h_{\cI_{\ell}}$.
Then $h_{\cI_{\ell}}(w_0) = w_{\ell}$. Let $h_{\ell} = h_{i_{\ell +1} , i_{\ell}}$ so that $h_{\ell} \circ h_{\cI_{\ell}} = h_{\cI_{\ell +1}}$
and $h_0 = h_{\cI_1}$.

We use induction on $\ell$ to show that $B_{\fX}(w_0, \dTU) \subset D(h_{\cI_{\gamma}})$.
First, note that $\dTU \leq \eTU$ implies $B_{\fX}(w_0, \dTU) \subset B_{\fX}(w_0, \eTU)$.
By the remarks following (\ref{eq-unifest}) we have that
$B_{\fX}(w_0, \eTU) \subset D(h_0) = D(h_{\cI_1})$.
By the definition of $\dTU$ we have that
$h_{\cI_1}(B_{\fX}(w_0 , \dTU)) \subset B_{\fX}(w_{1}, \eTU)$.
Then $B_{\fX}(w_{1}, \eTU) \subset D(h_1)$. Thus $B_{\fX}(w_0, \eTU) \subset D(h_{\cI_2})$.

Suppose that $B_{\fX}(w_0, \eTU) \subset D(h_{\cI_{\ell}})$.
As before, we have that
$$h_{\cI_{\ell}}(B_{\fX}(w_0 , \dTU)) \subset B_{\fX}(w_{\ell}, \eTU) \subset D(h_{\ell}) ~ .$$
Thus, $B_{\fX}(w_0, \eTU) \subset D(h_{\cI_{\ell +1}})$ and
$h_{\cI_{\ell +1}}(B_{\fX}(w_0 , \dTU)) \subset B_{\fX}(w_{\ell +1}, \eTU)$.

This completes the induction.
The last assertion on the existence of $\delta_1$ given $\e_1$ is just a restatement of equicontinuity for $h_{\cI_{\gamma}}$.
\endproof

Note that similar techniques can be used to prove the following, which implies that equicontinuity is a property of the foliation $\F$ of $\fM$, and does not depend on the particular covering chosen:
\begin{prop}\label{prop-indep} Let $\fM$ be a foliated space, with a regular covering $\cU$ such that
$\cGF$ is an equicontinuous pseudogroup. Then for any other choice of regular covering $\cU'$ of $\fM$, the
resulting pseudogroup $\cGF'$ will also be equicontinuous. \hfill $\Box$
\end{prop}

We say that $\fM$ is equicontinuous if for some regular covering of $\fM$, the groupoid $\cGF$ is   equicontinuous.

\subsection{Minimal foliations}

\begin{defn} \label{def-minimal}
A foliated space $\fM$ is minimal if each leaf $L \subset \fM$ is dense.
\end{defn}
The following is an immediate consequence of the definitions:
\begin{lemma} \label{lem-minimal}
A foliated space $\fM$ is minimal if and only if for some regular covering of $\fM$,
the holonomy pseudogroup $\cGF$ of $\F$ is minimal; that is, for all $w \in \fT_*$, the $\cGF$ orbit $\cO(w)$ of $w$ is dense.
\end{lemma}

A standard argument shows that equicontinuity of the action of $\cGF$ on $\fT_*$ implies that for each $w \in \fT_*$ the closure $\overline{\cO(w)}$ of its orbit  is a minimal set. This argument applies also in the case of an equicontinuous foliation of a compact manifold $M$ (see \cite{ALC2009}), and implies  that the ambient space $M$ is a disjoint union of minimal sets. However, as seen from the case of Riemannian foliations, where the closures of the leaves in $M$ can form a non-trivial fibration, this does not imply that the foliated manifold itself is minimal.
 Thus, the following result is, at first glance, very surprising: An equicontinuous action of the holonomy pseudogroup $\cGF$ on the \emph{totally disconnected} transverse space associated to matchbox manifolds is minimal. This has been previously shown in the context of   flows on homogeneous matchbox manifolds in \cite[page 5]{AHO1991}, and for homogeneous $\mR^n$-actions in \cite[page 275]{Clark2002}, and previously a version for equicontinuous group actions on compact Hausdorff spaces appears in  J. Auslander \cite{Auslander1988}. The proof below is a technical generalization of   these proofs, and extends the previous results to an equicontinuous action of holonomy pseudogroup $\cGF$ of $\F$ for a matchbox manifold.
In fact, the result can be thought of as a partial generalization to foliated spaces  of a well-known result of Sacksteder \cite{Sacksteder1965} for codimension-one foliations.

\begin{thm}\label{thm-minimal}
If $\fM$ is an equicontinuous matchbox manifold, then $\fM$ is minimal.
\end{thm}
\proof
The assumption that $\fM$ is a continua implies that it is connected. Thus, if $\fM$ is the \emph{disjoint} union of  open saturated subsets $U, V$ , then one of them must be empty. As the $\F$-saturation of disjoint open subsets of $\fT_*$ are disjoint and open in $\fM$, this implies that a clopen subset $\cW \subset \fT_*$ which is $\cGF$-invariant must be all of $\fT_*$.  We show below that if there exists a $\cGF$-invariant open non-empty proper subset  $\cW \subset \fT_*$ then $\fT_*$  contains a proper clopen subset, which contradicts that $\fM$ is connected. Thus, if $w \in \fT_*$ and its orbit closure $\overline{\cO(w)} \ne \fT_*$, then the complement $\cW = \fT_* - \overline{\cO(w)}$ is an open, non-empty proper subset,  which leads to  a contradiction. Thus, the closure of every orbit of $\cGF$ must be all of $\fT_*$.

Let $\cW \subset \fT_*$ be a $\cGF$-invariant open proper  subset  and $w \in \cW$. Let $i_{\alpha}$ be an index such that $w \in \fT_{i_{\alpha}}$.
The assumption $\fT_*$ that is totally disconnected implies that its topology has a basis of clopen subsets. Thus $w$ has a neighborhood system consisting of sets which are both open and closed,   hence compact.

Let $W_0 \subset \cW \cap \fT_{i_{\alpha}}$ be a clopen neighborhood of $w$. The $\cGF$-saturation of $W_0$ is the set
\begin{equation}\label{eq-saturate}
\cO(W_0) = \bigcup ~ \{g(W_0 \cap D(g)) \mid g \in \cGF^* ~, ~ D(g) \cap W_0 \ne \emptyset \} \subset \fT_* ~ .
\end{equation}
Since each map $g \colon D(g) \to R(g)$ is a homeomorphism, and $D(g)$ is open,  the set $\cO(W_0)$ is open.

We claim  that $\cO(W_0)$ is   closed. If not,  then there exists  $w_* \in \overline{\cO(W_0)} - \cO(W_0) \subset \fT_*$.
Choose $\{w_{\ell} \in \cO(W_0) \mid \ell = 1, 2, \ldots\}$ such that $\ds \lim_{\ell \to \infty} w_{\ell} = w_*$.
For each $\ell \geq 1$, there exists an admissible sequence $\cI^{(\ell)}$
and $\xi_{\ell} \in W_0 \cap D(h_{\cI^{(\ell)}}) \subset \fT_{i_{\alpha}}$ such that $w_{\ell} = h_{\cI^{(\ell)}}(\xi_{\ell})$.
As $\fT_*$ is the finite union of the compact sets $\{\fT_i \mid i =1, \ldots , \nu\}$, by passing to a subsequence, we can assume that there exists an index $i_{\beta}$ such that $w_{\ell} \in \fT_{i_{\beta}}$ for all $\ell \geq 1$.
Moreover, as $W_0$ is compact, we can also assume without loss of generality, that $\ds \lim_{\ell \to \infty} ~ \xi_{\ell} = \xi_* \in W_0$.

Let $\dTU > 0$ be the constant of Proposition~\ref{prop-uniformdom}.
As $W_0$ is open, there exists $0 < \e_1 < \eTU$ such that $B_{\fX}(\xi_*, 2 \e_1) \subset W_0$.
Let $0 < \delta_1 \leq \dTU$ be the   constant of equicontinuity for the action of $\cGF^*$ corresponding to  $\e_1$ which exists by Proposition~\ref{prop-uniformdom} as well.

Let $\gamma_{\ell}$ denote the leafwise path from $y_{\ell} = \tau_{i_{\beta}}(w_{\ell})$ to $x_{\ell} = \tau_{i_{\alpha}}(\xi_{\ell})$ determined by the \emph{reverse} of the admissible sequence $\cI^{(\ell)}$. Thus, $\gamma_{\ell}(0) = y_{\ell}$ and $\gamma_{\ell}(1) = x_{\ell}$.
By Proposition~\ref{prop-uniformdom}, for each $\ell \geq 1$, the path $\gamma_{\ell}$ defines an admissible sequence $\cJ^{(\ell)}$ such that $h_{\cJ^{(\ell)}}(w_{\ell}) = \xi_{\ell}$, $B_{\fX}(w_{\ell}, \dTU) \subset D(h_{\cJ^{(\ell)}})$
and $h_{\cJ^{(\ell)}}(B_{\fX}(w_{\ell}, \delta_1)) \subset B(\xi_{\ell}, \e_1)$.

Choose $\ell_0$ sufficiently large so that $\ell \geq \ell_0$ implies
$d_{\fX}(\xi_* , \xi_{\ell}) < \e_1$ and $d_{\fX}(w_* , w_{\ell}) < \delta_1$.
Then
$d_{\fX}(w_* , w_{\ell}) < \delta_1$ implies $w_* \in B_{\fX}(w_{\ell}, \delta_1)$,
and
$d_{\fX}(\xi_* , \xi_{\ell}) < \e_1$ implies $B_{\fX}(\xi_{\ell}, \e_1) \subset W_0$.

Thus,
$h_{\cJ^{(\ell)}}(B_{\fX}(w_{\ell}, \delta_1)) \subset W_0$,
so $h_{\cJ^{(\ell)}}(w_*) \in W_0$ hence $w_* \in \cO(W_0)$, contrary to its choice.

Thus, $\cO(W_0)$ is  a clopen subset of $\cW$, and is proper in $\fT_*$ as   $\cW$ is proper.
\endproof

\section{Homogeneous matchbox manifolds}\label{sec-hmm}

We next draw a connection between the homogeneity of a matchbox
manifold and the dynamics of its associated foliation. To do so, we
first recall a fundamental result of Effros~\cite{Effros1965},
presented in the spirit of~\cite{Ancel1987,vanMill2004}.   All topological
spaces considered here are assumed to be separable and metrizable.

\subsection{Micro-transitive actions}

Let $G$ be a topological group with identity $e$. An action $A$ of
$G$ on the space $X$  is a continuous map $A(g,x)=gx$ from $G\times
X$ to $X$ such that $ex=x$ for all $x\in X$,  and $f(gx)=(fg)x$ for
all $f,g\in G$ and $x\in X$. For $U\subseteq G$ and $x\in X$, let
$Ux=\{gx \mid g\in U\}$. An action of $G$ on $X$ is
\emph{transitive} if $Gx=X$ for all $x\in X$. It is
\emph{micro-transitive} if for every $x\in X$ and every neighborhood
$U \subset G$ of $e$, $Ux$ is a neighborhood of $x$. According to the theorem
of Effros, if a completely metrizable group $G$ acts transitively on
a second category space $X$, then it acts micro-transitively on $X$.
This result is a form of the ``Open Mapping Principle''  \cite{vanMill2004}.

Now consider the homeomorphism group $\cH(X)$ of a separable, locally
compact, metric space $X$ with the metric $d_{\cH}$ on $\cH(X)$ induced by the metric
$d_X$ on $X$:
$$
d_{\cH} \left( f, g\right) := \sup \left\{d_X \left( f(x), g(x)
\right) \mid x \in X\right\} +\sup \left\{ d_X \left( f^{-1}(x),
g^{-1}(x)\right) \mid x \in X\right\} ~ .
$$
With this metric, $\cH(X)$ is complete and acts continuously on $X$ in
the natural way: for $h \in \cH(X)$ and $x \in X$, $hx=h(x)$. Notice
that $X$ is homogeneous if and only if this action is transitive.
Effros' Theorem applied to this action states that if it is transitive, then it is also micro-transitive.
In the special case that $X$ is compact, we obtain that for any given $\e>0$ there
is a corresponding $\delta > 0$ so that if $d_X(x,y)< \delta$, there
is a homeomorphism $h \colon X\rightarrow X$ with $d_{\cH}(h,id_X)<\e$ and
$h(x)=y$.

\medskip

Let the homeomorphism group $\cH(\fM)$ have the metric
$d_{\cH}$ induced from the metric $d_{\fM}$. Then $\cH(\fM)$ is complete, so we can
apply the theorem of Effros to obtain:
\begin{cor}[Effros] \label{cor-effros}
Let $\fM$ be a homogeneous foliated space.
Given $\e^* > 0$, there is a corresponding $0 < \delta^* \leq \e^*$ so that for any $x,y \in \fM$ with $d_{\fM}(x,y) < \delta^*$, there is a
homeomorphism $h: \fM\rightarrow \fM$ with $d_{\cH}(h,id_{\fM}) <\e^*$ and $h(x)=y$. \hfill $\Box$
\end{cor}

\subsection{A key application}
The papers \cite{AHO1991, ALC2009, Ungar1975} give applications and examples of Effros Theorem related to the dynamics of flows. The fact that $\cH(\fM) = \cH(\fM , \F)$ by Corollary~\ref{cor-foliated2} is a key fact in these applications, and for the following application to foliated spaces.

\begin{thm}\label{thm-equic}
If $\fM$ is a homogeneous matchbox manifold, then $\fM$ is equicontinuous.
\end{thm}
\proof
The idea of the proof is simple in principle, though somewhat technical to show precisely. Basically, given a point $w \in \fT_*$ and an element $h_{\cI} \in \cGF^*$ with $w \in D(h_{\cI})$,  let $\gamma$ be a path defining the admissible sequence $\cI$. Then for a point $\eta \in D(h_{\cI})$, the value of $h_{\cI}(\eta)$ is defined by a path $\gamma_{\eta}$ starting at $\eta$ and shadowing the path $\gamma$.
Using Corollary~\ref{cor-effros}, for $\e > 0$ given, and $\eta$  sufficiently close to $w$, we can find such a path  $\gamma_{\eta} = \gamma^h$ starting at $\eta$ which is a conjugate of $\gamma$ by a homeomorphism $h$ of $\fM$ which is $\e$ close to $\gamma$. It follows that the endpoint of $\gamma^h$ is within $\e$ of $h_{\cI}(w)$, hence the action is equicontinuous. We give the details below.

Fix a regular cover of $\fM$, and let $\cGF$ be the associated groupoid, with notations as above.

Let $\e > 0$ be given.
Recall that $\kappa$ as defined by (\ref{eq-modgendef}) satisfies $\lim_{r \to 0} ~ \kappa(r) =0$, so there exists $r_0 > 0$ so that for all $0 < r \leq r_0$ we have $\kappa(r) < \e$.
Choose $0 < \e' < \e$ so that $\kappa(\e') < \e$.

Set $\e^* = \min \{\rp(\e'), \eU/4\}$, where $\rp$ is the uniform modulus of continuity function for the projections $\pi_i$ introduced in Lemma~\ref{lem-modpi}.

Let $\delta^*$ be determined by $\e^*$ as in Corollary~\ref{cor-effros}, and also assume that $\delta^* \leq \e^*$.
Let $\delta =\rt(\delta^*)$, where $\rt$ is the uniform modulus of continuity function for the sections $\tau_i$ introduced in Lemma~\ref{lem-modtau}.

Let $g \in \cGF^*$ be defined by an admissible sequence $\cI = (i_0 , \ldots , i_{\alpha})$. That is, $g = h_{\cI} | U$ for some open $U \subset D(h_{\cI})$. We show that for $w, \xi \in D(h_{\cI})$ with $d_{\fX}(w, \xi) < \delta$,  then
$d_{\fX}(w', \xi') < \epsilon$ where $w'= h_{\cI}(w)$ and $\xi' = h_{\cI}(\xi)$.

Let $x = \tau_{i_0}(w)$ and $y = \tau_{i_0}(\xi)$, then $x,y \in U_{i_0}$ and $d_{\fM}(x,y) < \delta^* \leq \eU/2$ by the definition of $\rt$ in Lemma~\ref{lem-modtau}.
Then set $x' = \tau_{i_{\alpha}}(w')$ and $y' = \tau_{i_{\alpha}}(\xi')$ so that $x' ,y' \in \oU_{i_{\alpha}}$.

Let $\gamma^x_{\cI} \colon [0,1] \to \fM$ be a leafwise piecewise geodesic path from $x$ to $x'$ determined by
the plaque chain $\cP_{\cI}(x)$, as in Section~\ref{sec-holonomy}.
Let $0 = s_0 < s_1 < \cdots < s_{\beta} = 1$ be a partition, and $\cJ = (j_0, \ldots , j_{\beta})$ an admissible sequence covering $\gamma^x_{\cI}$ so that
by equation (\ref{eq-unifest}),
\begin{equation}\label{eq-unifest2}
B_{\fM}(\gamma^x_{\cI}(t), \eU/2) \subset U_{j_{\ell}} \quad {\rm for ~ all}~ s_{\ell} \leq t \leq s_{\ell +1} ~ .
\end{equation}

Proposition~\ref{prop-copc} implies that
$\ds h_{\cI} | U = h_{i_{\alpha} , j_{\beta}} \circ h_{\cJ} \circ h_{j_{0} , i_{0}} | U$
where $U = D(h_{\cI}) \cap D(h_{i_{\alpha} , j_{\beta}} \circ h_{\cJ} \circ h_{j_{0} , i_{0}})$.

For $s=0$, by (\ref{eq-unifest2}) we have that $x, y \in U_{j_0}$, hence $x, y \in U_{i_0} \cap U_{j_0}$. Also, $x' \in U_{j_{\beta}}$ by construction.

Let $\eta = \pi_{j_0}(y) = h_{j_{0} , i_{0}}(\xi)$ and $w'' = \pi_{j_{\beta}}(x') = h_{j_{\beta} , i_{\alpha}}(w')$.

\medskip

The essence of the proof of Theorem~\ref{thm-equic} is the following.
\begin{lemma}\label{lem-key}
Let $\eta \in D(h_{\cJ})$. Set $\eta' = h_{\cJ}(\eta)$, then $d_{\fX}(w'' , \eta') < \e'$.
\end{lemma}
\proof

As $d_{\fM}(x,y) < \delta^*$, by Corollary~\ref{cor-effros} there exists $h \in \cH(\fM)$ with $h(x) = y$ and $d_{\cH}(h,id_{\fM}) < \e^*$.

By Corollary~\ref{cor-foliated2}, the composition $\gamma_{\cI}^h(t) = h \circ \gamma^x_{\cI}(t)$ is a leafwise path from $y = h(x)$ to $z = h(x')$, which satisfies
\begin{equation}\label{eq-effrosest}
d_{\fM}(\gamma_{\cI}^h(t)), \gamma_{\cI}^x(t))<\e^* \leq \eU/4 ~ , ~ {\rm for ~all} ~ 0 \leq t \leq 1 ~ .
\end{equation}

The conditions $d_{\fM}(h(\gamma^x_{\cI}(t)), \gamma^x_{\cI}(t))< \eU/4$ for all $0 \leq t \leq 1$, and
$B_{\fM}(\gamma^x_{\cI}(t), \eU/2) \subset U_{j_{\ell}}$ for each $s_{\ell} \leq t \leq s_{\ell +1}$, imply that
$\gamma^h_{\cI}([s_{\ell}, s_{\ell +1}]) \subset U_{j_{\ell}}$ for all $0 \leq t \leq 1$. Thus, $\eta \in D(h_{\cJ})$ and the trace of $\gamma^h_{\cI}(t)$ is contained in the plaque chain $\cP_{\cJ}(y)$.

Set $z = \gamma^h_{\cI}(1) = h(x')$, and note that $\eta' = \pi_{j_{\beta}}(z)$.
Then by (\ref{eq-effrosest}) we have that $d_{\fM}(x', z) < \e^* <\rp(\e')$,
so by the definition of $\rp$ we have
$\ds d_{\fX}(w'' , \eta' ) = d_{\fX}(\pi_{j_{\beta}}(x') , \pi_{j_{\beta}}(z)) < \e'$.
This completes the proof of the Lemma.
\endproof
By Lemma~\ref{lem-key} and Proposition~\ref{prop-copc} we have that
$$
\xi' = h_{\cI}(\xi) = h_{i_{\alpha} , j_{\beta}} \circ h_{\cJ} \circ h_{j_{0} , i_{0}}(\xi)
= h_{i_{\alpha} , j_{\beta}} \circ h_{\cJ}(\eta)
= h_{i_{\alpha} , j_{\beta}}(\eta') ~ .
$$
Since $w' = h_{ i_{\alpha}, j_{\beta}}(w'')$ by the definition of $\kappa$ and $\e'$, we have that
$$d_{\fX}(w' , \xi') \leq \kappa(d_{\fX}(w'', \eta')) \leq \kappa(\e') < \e ~ .$$
This completes the proof of Theorem~\ref{thm-equic}.
\endproof

We make another observation about the dynamics of homogeneous foliated spaces, which imposes further restrictions on which matchbox manifolds can be homogeneous.
\begin{lemma}\label{lem-homonoholo}
If $\fM$ is a homogeneous foliated space, then $\F$ is without holonomy.
\end{lemma}
\proof
By Theorem~\ref{thm-emt} the set of leaves without holonomy is a dense $G_{\delta}$, hence there exists at least one leaf $L_x$ without holonomy. Given any other leaf $L'$ of $\F$, choose basepoints $x \in L_x$ and $y \in L'$. As $\fM$ is homogeneous, there exists a homeomorphism $h$ of $\fM$ such that $h(x) = y$, and $h$ is a foliated map by   Corollary~\ref{cor-foliated2}. Then  the restriction $h \colon L \to L'$ is a homeomorphism,  hence $L'$ also is without holonomy.
\endproof

\section{Transverse holonomy and orbit coding}\label{sec-codes}

Let $\fM$ be an equicontinuous matchbox manifold.
In this section, we show that the orbits of the equicontinuous pseudogroup associated to $\F$ admit  finite codings, which is the basis for the proof of Theorem~\ref{thm-main2}. The techniques are inspired by the   paper of Thomas \cite{EThomas1973} which showed a similar result for equicontinuous actions  of group $\mZ$ on a Cantor set. For  matchbox manifolds with dimension $ n \geq 2$, this requires extending the basic ideas from  group actions to pseudogroup actions.

Fix a regular covering $\cU = \{\vp_i
\colon \oU_i \to [-1,1]^n \times \fT_i \mid 1 \leq i \leq \nu \}$
and pseudogroup $\cGF$ as in Section~\ref{sec-holonomy}.
Let $\dTU > 0$ be the constant of Proposition~\ref{prop-uniformdom},
such that for every leafwise path
$\gamma \colon [0,1] \to \fM$, we have
$B_{\fX}(w_0, \dTU) \subset D(h_{\gamma})$
for the   holonomy map $h_{\gamma} \in \cGF^*$ where $w_0 \in \fT_i$ corresponds to $\gamma(0)$.

Theorem~\ref{thm-minimal} implies that for any open subset $W \subset \fT_*$ the $\cGF$-saturation
$\cO(W)$, as defined in (\ref{eq-saturate}), is all of $\fT_*$. We study the dynamics of $\cGF$ restricted to
sufficiently small clopen subsets of $\fT_*$.

\subsection{Holonomy groupoids} Fix a coordinate transversal, say $\fT_1$, and  a  basepoint  $w_0 \in int(\fT_1)$.

Choose $W \subset B_{\fX}(w_0, \dTU/4) \subset \fT_1$ such that $W$ is clopen and $w_0 \in int(W) = W$. The leaf $L_{w_0}$ of $\F$ through $w_0$ is dense in $\fM$, so  by the minimality of $\F$, the union of all images of $W$ under the holonomy along $L_{w_0}$ is all of $\fT_*$.
The goal is to  choose a sequence of clopen subsets
\begin{equation}\label{eq-chain}
w_0 \in \cdots \subset V^{\ell} \subset V^{\ell -1} \subset \cdots V^1 \subset V^0 \subset W
\end{equation}
all   which contain $w_0$, and such that the $\cGF^*$ orbits of  each $V^{\ell}$ has a ``finite order periodic coding''.

Let $\cRF \subset \fT_* \times \fT_*$ denote the equivalence relation on  $\fT_*$ induced by $\F$, where     $(w,w') \in \cRF$ if and only if $w,w'$ correspond to points on the same leaf of $\F$.

Let $(w,w') \in \cRF$, and $\gamma = \gamma_{x,x'} \colon [0,1] \to \fM$  denote a path from $x = \tau_{i_x}(w)$ to $x' = \tau_{i_{x'}}(w')$.
Recall from Definition~\ref{def-germ} that $[h_{\gamma}]_w$ denotes its germinal class, which by
 Lemma~\ref{lem-homotopic}  depends only on the endpoint-fixed homotopy class of $\gamma$.
 The holonomy groupoid $\GF$ of $\F$ is the collection of all such germs, and the source and range maps from Definition~\ref{def-germ} define a groupoid map  $s \times r \colon \GF \to \cRF$.

 Given an element  $g \in \GF$ with $s(g) = w$ and $r(g) = w'$, there exists a   path $\gamma = \gamma_{x,x'} \colon [0,1] \to \fM$   from $x = \tau_{i_x}(w)$ to $x' = \tau_{i_{x'}}(w')$, where $w \in W$, so that $g = [h_{\gamma_{x,x'}}]_w$. Note that by Proposition~\ref{prop-uniformdom},    there is a plaque chain covering $\gamma$ such that
$W \subset D(h_{\gamma_{x,x'}})$ for the associated holonomy map $h_{\gamma_{x,x'}}$. As the constant $\dTU > 0$ is independent of the choice of the path $\gamma$, given  $g \in\GF$ we abuse notation, and let $\gamma$ denote   the holonomy map $h_{\gamma_{x,x'}}$   whose germ is $g$ and satisfies this domain condition. Then given such $\gamma$ and any $u \in W$,  let $\gamma_u$ denote the germ of $\gamma$ at $u$. Thus, $\gamma_w =  g$ for example.

We have previously introduced in (\ref{eq-holodef})  the germinal holonomy subgroups $\ds \G_{\F,w} \equiv (s \times r)^{-1}(w,w)$ for $w \in \fT_*$. Also, consider
the following subsets of $\GF$:
\begin{eqnarray*}
\G_W & = & \left\{ \gamma_u \in \GF \mid u \in W \right\}   \\
\G_W^W & = & \left\{ \gamma_u \in \GF \mid u \in W ~, ~ r(\gamma_u) \in W\right\}   \\
\G_w^W & = & \left\{  \gamma_{u} \in \GF  \mid  u = w ~, ~ r(\gamma_u) \in W \right\}
\end{eqnarray*}

Note that $\G_W^W$ is a subgroupoid of $\GF$, with object space $W$. For each $w_0 \in W$, let  $*_{w_0}$ denote the constant path at $w_0$, and by abuse of notation, also let  $*_{w_0} \in \G_{w_0}^W$  denote the germ of the identity map at $w_0$.
The composition rule for $\GF$ is defined by the concatenation of paths. In the case where $\cGF^*$ is equicontinuous, a   stronger form of composition holds.

 \begin{lemma}
There is a well-defined     composition law, $* \colon \G_W^W \times \G_W^W \to \G_W$.
 \end{lemma}
 \proof
 Let $\gamma_w , \gamma_u' \in \G_W$ with $w,u \in W$. Then $w' = r(\gamma_w) \in W$, so $w' \in Dom(\gamma')$, hence  $w'' = \gamma'(w')$ is well-defined.
 Then define
 \begin{equation}\label{eq-complaw}
 \gamma_w * \gamma_u' \equiv  [\gamma' \circ \gamma]_w \quad {\rm so ~ that} \quad r(\gamma_w * \gamma_u') = w'' ~ . \qquad\qquad\qquad  \Box
\end{equation}

The composition law  (\ref{eq-complaw})  has a natural intuitive definition in terms of the definition of the holonomy along paths, which yields a slightly stronger conclusion. The holonomy map $\gamma'$ is defined as the holonomy along a leafwise path $\gamma_{y,y'}'$ between $y= \tau_1(u)$ and $y' = \tau_1(u')$.
 The action of $\gamma'$ on the point $w'$ is defined by a leafwise path $\gamma_{\xi,\xi'}'$ between $\xi = \tau_1(w')$ and $\xi' = \tau_1(w'')$ which ``shadows'' $\gamma_{y,y'}'$.
 Form the concatenation of the two paths, $\gamma_{x,x''}'' = \gamma_{\xi,\xi'}' * \gamma_{x,x'}$ which is leafwise path between $x = \tau_1(w)$ and $x'' = \xi' = \tau_1(w'')$.
 Then the product $\gamma_w * \gamma_u'$ is the germinal holonomy along the path
$\gamma_{x,x''}''$.
We will use this geometric description of the composition law in later arguments.
  Note that by Proposition~\ref{prop-uniformdom},    there is a representative for $h_{\gamma''}$ with    $W \subset D(h_{\gamma''})$.
However,  there is no assertion that $h_{\gamma''}(W) =W$.

\medskip

\subsection{Coding functions}

We use the equicontinuity of $\cGF$ and that $\fT_*$ is totally disconnected to define the subsets $V^{\ell} \subset W$ in (\ref{eq-chain}),
for which the holonomy action of $\G_{w_0}^W
$ on $V^{\ell}$ has ``uniform return times''. This procedure corresponds to the procedure in Thomas' work, in the case where $\F$ is defined by a flow, where the section $W$ is partitioned into sets with uniform return times.

We begin the inductive construction of the clopen sets $V^{\ell} \subset W$ for $\ell \geq 1$.

Let $\ell = 1$ and set $\e_1 = \dTU/4$. Choose a partition of $W$ into disjoint clopen subsets,
\begin{equation}
\cW_1 = \{W^1_1, \ldots , W^1_{\beta_1} \} \quad , \quad  W = W^1_1 \cup \cdots \cup W^1_{\beta_1} \quad , \quad w_0 \in W^1_1
\end{equation}
 where  each $W^1_i$ has diameter (in $\fX$) less than $\e_1$.
If $\e_1 >  \dX(W)$, one can choose  $W^1_1 = W$ and   $\beta_1 = 1$. However,  in the later stages of the inductive construction, $\e_{\ell} \to 0$ so $\beta_{\ell} \to \infty$.
 Introduce the first code space $\cC_1 = \left\{0, 1,\ldots, \beta_1\right\}$.

Let $\ds \eta_1 = \min\left\{d_{\fX}(W, \fT_*-W), d_{\fX}(W^1_i,W^1_j)\,|\,i \neq j\,\right\} > 0$, which is positive as $W$ is clopen in $\fX$.

Let $\delta_1 > 0$ be a constant of equicontinuity for $\cGF$ corresponding to $\eta_1$.

\medskip

Next, introduce the coding function corresponding to the partition $\cW_1$.
\begin{defn}\label{def-coding-1} For $w \in W$,
the $\cC^1_w$-code of $u \in W$ is the function $C^1_{w,u} \colon \G_w^W \to \cC_1$ defined as, for $\gamma \in \G_w^W$:
$$C^1_{w,u}(\gamma) = \begin{cases}
   i   &  \text{if  } ~  \gamma(u) \in W^1_i\\
   0   & \text{if  } ~ \gamma(u) \in \fT_* - W.
\end{cases}
$$
\end{defn}
The function $C^1_{w, u}$ encodes the terminal point  for the path starting at $u$, and shadowing the path $\gamma$ from $w$ to   $w' = r(\gamma_w) \in \cO(w)$.
In particular,   $C^1_{w,u}(\gamma) = 0$ corresponds to those points $u \in W$
such that $\gamma(u) \not\in W$. Thus, if $W$ is invariant under $\cGF^*$ then the code value $C^1_{w,u}(\gamma) = 0$ never occurs.

\medskip

\begin{lemma}\label{lem-equivariant-1}
Let $w, u \in W$ and $\gamma' \in \G_w^W$. Set $w' = r(\gamma')$ and suppose that $u' = \gamma'(u) \in W$, then
\begin{equation}\label{eq-equivariant-1}
  C^1_{w', u'}(\gamma) = C^1_{w,u}(\gamma * \gamma') ~ {\rm for ~ all} ~ \gamma\in \G_{w'}^W.
\end{equation}
\end{lemma}
\proof
The function $C^1_{w', u'}$ codes for a path $\gamma_{u', u''}$ starting at $u'$ and shadowing the path $\gamma = \gamma_{w', w''}$ from $w'$ to   $w'' = r(\gamma)  \in   \cO(w') = \cO(w)$. Pre-compose the  path $\gamma_{w', w''}$ with $\gamma'$, to obtain
paths
$$\gamma''_{w,w''} = \gamma_{w', w''} * \gamma'_{w, w'}  \in \G_w^W \quad , \quad \gamma''_{u, u''} = \gamma_{u', u''} * \gamma'_{u, u'}  \in \G_u^W$$
where the latter shadows the former.  We then have
$\gamma(u') = \gamma \circ \gamma'(u)$,
from which (\ref{eq-equivariant-1}) follows.
\endproof

\medskip

\begin{lemma} \label{lem-locconstant}
If $u,v \in W$ with $d_{\fX}(u,v) < \delta_1$ then $C^1_{w,u}(\gamma) = C^1_{w,v}(\gamma)$ for all $\gamma \in \G_w^W$.
Hence, the function $C^1_w$ defined by $C^1_w(u) = C^1_{w,u}$ is locally constant, and so $V^1$ is open.
\end{lemma}
\proof
Let  $\gamma \in \G^W_W$, and suppose that $u,v \in W$ with $d_{\fX}(u,v) < \delta_1$.
Set $u' = \gamma(u)$ and $v' = \gamma(v)$.
By the equicontinuity of $\cGF$, $d_{\fX}(u,v) < \delta_1$ implies $d_{\fX}(u',v') < \eta_1$.
Assume that  $u' = \gamma(u) \in W^1_i$ then
$d_{\fX}(u',v') < d_{\fX}(W, \fT_*-W)$ implies  $v' \in W$.
Moreover, $d_{\fX}(u',v') < d_{\fX}(W^1_i,W^1_j)$ for all $j \ne i$ implies that $v' \in W^1_i$.
Thus, $C^1_{w,u}(\gamma) = C^1_{w,v}(\gamma)$.
\endproof

We now begin the construction of the coding partitions. For the first stage, $\ell =1$, set:
 \begin{eqnarray}
V^1 & = &  \left\{ u \in W \mid C^1_{w_0,u}(\gamma) = C^1_{w_0 , w_0}(\gamma) ~ {\rm for ~ all} ~ \gamma \in \G_{w_0}^W
\right\} \label{eq-codedef-1} \\
& = & \bigcap \left\{  h_{\gamma}^{-1}(W^1_i)  \mid \gamma \in \G_{w_0}^W ~ , ~ \gamma(w_0) \in  W^1_i\right\} \subset W^1_1 ~ . \nonumber
\end{eqnarray}
Note that $w_0 \in V^1$, so that $V^1$   is non-empty, and it is open by Lemma~\ref{lem-locconstant}.

For $\gamma \in \G_{w_0}^W$ we have $\gamma(w_0)  \in W$  by definition of $\G_{w_0}^W$, and the function $u \mapsto  C^1_{w_0,u}(\gamma)$ is constant on $V^1$, hence $\gamma(V^1) \subset W$. Moreover,  the image $\gamma(V^1) \subset W$ is   open for each $\gamma \in \G_{w_0}^W$.

\medskip

\begin{lemma}\label{lem-fixset}
Let $\gamma \in \G_{w_0}^W$ and set
$V^1_{\gamma} = \gamma(V^1)$. If $V^1 \cap V^1_{\gamma} \ne \emptyset$,  then $V^1 = V^1_{\gamma}$.
\end{lemma}
\proof
By assumption, there exists $u' = \gamma(u) \in V^1 \cap V^1_{\gamma}$ for some $u \in V^1$.
We first show that all points $v' \in V^1_{\gamma}$ have the same code function as $u'  \in V^1$, so that $v' \in V^1$, hence
$V^1_{\gamma} \subset V^1$.

Let $w' = \gamma(w_0) \in W^1_i$.

Given $v' \in V^1_{\gamma}$   there exists $v \in V^1$ with $v' = \gamma(v)$.  Then $w_0, u,v \in V^1$ implies by (\ref{eq-codedef-1}) we have that  $C^1_{w_0,u}(\gamma) = C^1_{w_0,v}(\gamma) = C^1_{w_0,w_0}(\gamma)$ so that $v'$ and $u'$ lie in the same partition   $W^1_i$.

For $\gamma' \in \G_{w_0}^W$   we    claim that $C^1_{w_0,v'}(\gamma') = C^1_{w_0, u'}(\gamma')$.
Let   $\wtgamma'$ be a path starting at $w'$   shadowing $\gamma'$.

Define  $\gamma'' = \wtgamma' * \gamma \in \G_{w_0}^W$.
 Then by Lemma~\ref{lem-equivariant-1},   we have
\begin{equation}
C^1_{w_0,v'}(\gamma') =  C^1_{w',v'}(\wtgamma') =  C^1_{w_0,v}(\wtgamma' * \gamma) =  C^1_{w_0,u}(\wtgamma' * \gamma) =  C^1_{w',u'}(\wtgamma')  =  C^1_{w_0,u'}(\gamma')
\end{equation}
where we also use that the holonomy along the path $\gamma$ starting at $w_0$ and the path $\wtgamma$ starting at $w'$ agree by Lemma~\ref{lem-domainconst}.   Thus, $V^1_{\gamma} \subset V^1$.

The reverse inclusion $V^1 \subset V^1_{\gamma}$ follows by the same arguments applied to  $\gamma^{-1}$.
\endproof

\begin{lemma} \label{lem-equalsets}
Let $\gamma, \sigma\in \G_W^W$, and   $V^1_{\gamma} = \gamma(V^1)$, $V^1_{\sigma} = \sigma(V^1)$.
If $d_{\fX}(V^1_\gamma, V^1_{\sigma} ) < \delta_1$, then
$V^1_{\gamma} = V^1_{\sigma}$.
\end{lemma}
\proof
By assumption, there exists $\xi, \zeta \in V^1$ such that
$\xi' = \gamma(\xi) \in V^1_{\gamma}$ and $\zeta' =\sigma(\zeta) \in V^1_{\sigma}$ such that
$d_{\fX}(\xi' , \zeta') < \delta_1$. In particular, if $\xi' \in W^1_i$ then also $\zeta' \in W^1_i$.

Set $w' = \gamma(w_0)$, then by Lemma~\ref{lem-locconstant},
for all $\gamma' \in \G_{w_0}^W$ we have
$C^1_{w',\xi'}(\gamma') = C^1_{w',\zeta'}(\gamma')$.

Set  $\xi'' = \sigma^{-1}(\xi') = \sigma^{-1} \circ \gamma(\xi)$.
Then for all $\gamma'' \in \G_{w_0}^W$ we have
$C^1_{w_0,\xi''}(\gamma'') = C^1_{w_0,\zeta}(\gamma'')$.
As   $\zeta   \in V^1$ this implies $\xi'' \in V^1$ and thus $V^1 \cap \sigma^{-1} \circ \gamma(V^1) \ne \emptyset$.
By Lemma~\ref{lem-fixset} this implies $V^1 = \sigma^{-1} \circ \gamma(V^1)$ hence $V^1_{\sigma} =    V^1_{\gamma}$ as was to be shown.
\endproof

 \medskip
To complete the construction for the first stage, note that $V^1$ is an open set, hence the
 minimality of $\F$ implies that for each $\xi \in W$, we have $\cO(\xi) \cap V^1 \ne \emptyset$, hence there exists $\sigma \in \G_W^W$ with $\xi' = \sigma(\xi) \in V^1$. So for $\gamma = \sigma^{-1}$ we have $\xi \in V^1_{\gamma}$.
  Thus, the set $\left\{\gamma(V^1) \mid \gamma \in \G_{w_0}^W \right\}$ of all $\G_{w_0}^W$-translates of $V^1$ is an open cover of the clopen set $W$.
  As $W$ is compact, the cover admits a finite subcover
$\cV_1 \equiv \left\{V^1_1, \ldots , V^1_{n_1}\right\}$.

Moreover, $V^1_i \cap V^1_j = \emptyset$ for $i \ne j$ by Lemma~\ref{lem-equalsets}.
As $\cV_1$ is a finite covering of the compact set $W$ by disjoint open sets, each of the sets $V^1_i$  is also closed, hence is clopen.

For each $1 \leq i \leq n_1$ choose $\gamma_i^1 \in \G_{w_0}^W$ so that $V^1_i = \gamma_i^1(V^1)$.
Without loss of generality, we can assume that $\gamma_1^1 = *_{w_0}$ is the constant path, so that $V^1_1 = V^1$.
Set $w_i^1 = \gamma_i^1(w_0)$.

Note that by definition of the coding function used to define $V^1$, we have $V^1_i \subset W^1_j$ where $j = C^1_{w_0, w_0}(\gamma_i^1)$.
In particular,  $\dX(V^1_i) \leq \dX(W^1_j) < \e_1$ by construction.

Thus, the collection $\cV_1$ is a finite partition of $W$ by clopen subsets of diameter less that $\e_1$, which refines the initial partition $\cW_1$. This concludes the construction for the initial inductive step $\ell = 1$.

\medskip

Now consider the general inductive step, for $\ell > 1$.
Assume that for each $1 < \lambda < \ell$ there is given:
\begin{itemize}
\item a constant $\e_{\lambda}$ with $0 < \e_{\lambda} < \e_{\lambda-1}/2$,
\item a clopen set $V^{\lambda}$ with $w_0 \in V^{\lambda} \subset V^{\lambda -1}$,
\item a collection $ \{\gamma^{\lambda}_1, \ldots , \gamma^{\lambda}_{n_{\lambda}} \} \subset \G_{w_0}^W$ with $\gamma^{\lambda}_1 = *_{w_0}$,
\end{itemize}
 such that $\cV_{\lambda} = \{ V^{\lambda}_i \equiv \gamma^{\lambda}_i(V^{\lambda}) \mid 1 \leq i \leq n_{\lambda} \}$ is a finite partition of $W$ by disjoint clopen sets with $V^{\lambda}_1 = V^{\lambda}$, and the diameter of each $V^{\lambda}_i$ is less than $\e_{\lambda}$.
Moreover,   we assume that    the holonomy maps $\left\{ \gamma^{\lambda}_i \mid  1 \leq i \leq n_{\lambda}\right\}$ are such that the collection
$\cV_{\lambda}^* = \left\{ V^{\lambda}_i \mid 1 \leq j \leq \alpha_{\lambda}\right\}$ is a partition of the clopen set $V^{\lambda}$ into  disjoint clopen sets.
Furthermore,
 the covering $\cV_{\lambda}$ is assumed to be given by the  union of the translates of the partition  $\cV_{\lambda}^*$ by the maps $\{ \gamma^{\lambda-1}_j  \mid 1 \leq j \leq n_{\lambda -1} \}$.
That is, the covering $\cV_{\lambda}$ is a refinement of the covering $\cV_{\lambda -1}$ obtained by partitioning $V^{\lambda -1}$ into the clopen subsets of $\cV_{\lambda}^*$ and then translating them to obtain the covering $\cV_{\lambda}$ of $W$.
It follows that  $n_{\lambda} = \alpha_{\lambda} \cdot n_{\lambda -1}$ for some integer $\alpha_{\lambda} \geq 1$.

We begin the construction of the next partition $\cV_{\ell} =  \{ V^{\ell}_i \equiv \gamma^{\ell}_i(V^{\ell}) \mid 1 \leq i \leq n_{\ell} \}$, given the above data. First, set
\begin{equation}\label{eq-epsilonell}
\e_{\ell} = \frac{1}{2} \min \left \{ \dX(V^{\ell -1}_i) \mid 1 \leq i \leq n_{\ell -1} \right\} < e_{\lambda_{\ell -1}}/2 ~ .
\end{equation}
Let $\e_{\ell}' >0$ be a constant of equicontinuity for $\cGF$ corresponding to $\e_{\ell}$.

Next, choose  a partition of $V^{\ell -1}$ into disjoint clopen subsets,
$\ds \cW_{\ell}^* = \left\{W^{\ell}_1, \ldots , W^{\ell}_{\alpha'_{\ell}}\right\}$ where $w_0 \in W^{\ell}_1$ and each $W^{\ell}_i$ has diameter less than $\e_{\ell}'$.
As in the case $\ell =1$, this partition of $V^{\ell -1}$  need not be ``compatible'' with the dynamics of $\cGF$. The partition will be ``pruned'' using the coding map, as in the case $\ell =1$, to obtain a partition that is ``compatible'' with the dynamics of $\cGF$.

Extend the partition $\cW_{\ell}^*$ of $V^{\ell -1}$ to all of $W$, setting
$\ds \cW_{\ell} = \cW_{\ell}^1 \cup \cdots \cup \cW_{\ell}^{n_{\ell -1}}$
where $\ds \cW_{\ell}^i = \left\{ \gamma^{\ell-1}_i(\cW_{\ell}^j) \mid 1 \leq j \leq \alpha'_{\ell}\right\}$ is itself a partition of
$V^{\ell-1}_i = \gamma^{\ell-1}_i(V^{\ell -1})$ into clopen sets.

Note that by the choice of $\e_{\ell}'$ each set $V^{\ell-1}_i$ has diameter at most $\e_{\ell}$.
Define
\begin{equation}\label{eq-ellpartition}
 W^{\ell}_k = \gamma^{\ell-1}_i(W^{\ell}_j) ~ {\rm where} ~ k = j + (i-1) \cdot \alpha'_{\ell} ~, ~ 1 \leq j \leq \alpha'_{\ell} ~ , ~ 1 \leq i \leq  n_{\ell -1} ~ .
\end{equation}

That is, we relabel the collection $\cW_{\ell}$ using a lexicographical ordering.
Set $\beta_{\ell} = \alpha'_{\ell} \cdot n_{\ell -1}$, and define the code space $\cC_{\ell} = \left\{ 1,\ldots, \beta_{\ell}\right\}$.

The corresponding coding function is defined as before:
\begin{defn}\label{def-coding-ell} For $w \in W$,
the $C^{\ell}_w$-code of $u \in W$ is the function $C^{\ell}_{w,u} \colon \G_w^W \to \cC_{\ell}$ defined as:
for $\gamma \in \G_w^W$ set $C^{\ell}_{w,u}(\gamma) = k$ if $\gamma(u) \in W^{\ell}_k$.
Then define
\begin{equation}\label{eq-codedef-ell}
V^{\ell} = \left\{ u \in W^{\ell}_1 \subset V^{\ell -1} \mid C^{\ell}_{w_0,u}(\gamma) = C^{\ell}_{w_0 , w_0}(\gamma) ~ {\rm for ~ all} ~ \gamma \in \G_{w_0}^W \right\} ~ .
\end{equation}
Note that $\gamma(V^{\ell}) \subset W$ for all $\gamma \in \G_{w_0}^W$.
\end{defn}

Let $\eta_{\ell} = \min\left\{d_{\fX}(W^{\ell}_{k},W^{\ell}_{k'}) \mid 1 \leq k \neq k' \leq \beta_{\ell}\right\}>0$.
Let $\delta_{\ell} > 0$ be a constant of equicontinuity for $\cGF$ corresponding to $\eta_{\ell}$.
Then the following results are proved exactly as for the case $\ell = 1$.

\begin{lemma} \label{lem-locconstant-ell}
If $u,v \in W$ with $d_{\fX}(u,v) < \delta_{\ell}$ then $C^{\ell}_{w,u}(\gamma) = C^{\ell}_{w,v}(\gamma)$ for all $\gamma \in \G_w^W$.
Hence, the function $C^{\ell}_w$ defined by $C^{\ell}_w(u) = C^{\ell}_{w,u}$ is locally constant, and so $V^{\ell}$ is open.
\end{lemma}

\begin{lemma}\label{lem-fixset-ell}
Let $\gamma \in \G_{w_0}^W$ and set $V^{\ell}_{\gamma} = \gamma(V^{\ell})$. If $V^{\ell} \cap V^{\ell}_{\gamma} \ne \emptyset$,  then $V^{\ell} = V^{\ell}_{\gamma}$.
\end{lemma}

\begin{lemma} \label{lem-equalsets-ell}
Let $\gamma, \sigma\in \G_W^W$, and   $V^{\ell}_{\gamma} = \gamma(V^{\ell})$, $V^{\ell}_{\sigma} = \sigma(V^{\ell})$.
If $d_{\fX}(V^{\ell}_\gamma, V^{\ell}_{\sigma} ) < \delta_{\ell}$, then
$V^{\ell}_{\gamma} = V^{\ell}_{\sigma}$.
\end{lemma}

The completion of    the $\ell$-th stage of the induction proceeds as for $\ell =1$. Note that $V^{\ell}$ is an open set, hence the
 minimality of $\F$ implies that for each $\xi \in W$, we have $\cO(\xi) \cap V^{\ell} \ne \emptyset$, hence there exists $\sigma \in \G_W^W$ with $\xi' = \sigma(\xi) \in V^{\ell}$. So for $\gamma = \sigma^{-1}$ we have $\xi \in V^{\ell}_{\gamma}$.

  Thus, the set $\left\{\gamma(V^{\ell}) \mid \gamma \in \G_{w_0}^W \right\}$ of all $\G_{w_0}^W$-translates of $V^{\ell}$ is an open cover of the clopen set $W$.   As $W$ is compact, the cover admits a finite subcover
$\cV_{\ell} \equiv \left\{V^{\ell}_1, \ldots , V^{\ell}_{n_{\ell}}\right\}$.

Moreover, $V^{\ell}_i \cap V^{\ell}_j = \emptyset$ for $i \ne j$ by Lemma~\ref{lem-equalsets}.
As $\cV_{\ell}$ is a finite covering of the compact set $W$ by disjoint open sets, each of the sets $V^{\ell}_i$  is also closed, hence is clopen.

From the definition of $V^{\ell}$ and the coding function $C^{\ell}_{w,u}$ we have $V^{\ell} \subset W^{\ell}_1$.
Moreover, if $V^{\ell}_{\gamma} = \gamma(V^{\ell}) \cap V^{\ell -1} \ne \emptyset$, then
$V^{\ell}_{\gamma} \subset V^{\ell -1}$ by Lemma~\ref{lem-locconstant-ell}, following that of Lemma~\ref{lem-locconstant}.
Thus, the collection
$$
\cV_{\ell}^* = \left\{ V^{\ell}_{\gamma}  \mid \gamma \in \G_{w_0}^W ~, ~ V^{\ell}_{\gamma} \cap V^{\ell -1} \ne \emptyset \right\}
$$
  is a partition of $V^{\ell -1}$ by   clopen sets. Hence there is a finite collection
$\left\{\gamma^{\ell}_1, \ldots , \gamma^{\ell}_{\alpha_{\ell}}\right\} \subset \G_{w_0}^W$
so that for $V^{\ell}_j = V^{\ell}_{\gamma_j} = \gamma_j(V^{\ell})$ we have
$\cV_{\ell}^* = \left\{ V^{\ell}_j   \mid 1 \leq j \leq \alpha_{\ell}\right\}$ is a partition of $V^{\ell -1}$ by \emph{disjoint} sets.
As before,  we can assume that $\gamma_{\ell}^1 = *_{w_0}$ so that $V^{\ell}_1 = V^{\ell}$.

Recall that $\ds \cW_{\ell}^* = \left\{W^{\ell}_1, \ldots , W^{\ell}_{\alpha'_{\ell}}\right\}$ is the chosen partition of $V^{\ell -1}$ into disjoint clopen subsets, so by the definition of $V^{\ell}$, we have $\gamma^{\ell}_j(V^{\ell}) \subset W^{\ell}_k$ where
$1 \leq k = C^{\ell}_{w_0,w_0}(\gamma^{\ell}_j) \leq \alpha'_{\ell}$. Thus, $\cV_{\ell}^*$ is a refinement of the partition $\cW_{\ell}^*$ of $V^{\ell -1}$. In particular, note that $\alpha_{\ell} \geq \alpha_{\ell}'$.

Finally, extend the partition $\cV_{\ell}^*$ of $V^{\ell -1}$ to all of $W$, setting
$\cV_{\ell} = \cV_{\ell}^1 \cup \cdots \cup \cV_{\ell}^{n_{\ell -1}}$
where $ \cV_{\ell}^i = \{\gamma^{\ell-1}_i(V^{\ell}_j) \mid 1 \leq j \leq \alpha_{\ell} \}$ is a partition of
$V_{\ell-1}^i = \gamma^{\ell-1}_i(V^{\ell -1}) $ into clopen sets.
Note that by the choice of $\e_{\ell}'$ each set $\gamma^{\ell-1}_i(V^{\ell}_j)$ has diameter at most $\e_{\ell}$.

Set $n_{\ell} = \alpha_{\ell} \cdot n_{\ell -1}$, and for $1 \leq k \leq n_{\ell}$
relabel the collection $\cV_{\ell}$ using a lexicographical ordering,
$$ V^{\ell}_k = \gamma^{\ell-1}_i(V^{\ell}_j) ~ {\rm where} ~ k = j + (i-1) \cdot \alpha_{\ell} ~, ~ 1 \leq j \leq \alpha_{\ell} ~ , ~ 1 \leq i \leq n_{\ell -1} ~ .$$

Thus, the collection $\cV_{\ell}$ is a finite partition of $W$ by clopen subsets of diameter less that $\e_{\ell}$, which refines the initial partition $\cW_{\ell}$. Set $w_i^{\ell} = \gamma_i^{\ell}(w_0)$.

This concludes the general inductive step of the construction of the partitions
$\ds \cV_{\ell} \equiv \left\{ V^{\ell}_k \mid 1 \leq k \leq n_{\ell}\right\}$.

\section{Equicontinuity and Thomas tubes}\label{sec-tubes}

Let $\fM$ be an equicontinuous matchbox manifold. In this section, we show how
the partitions $\cV_{\ell}$ of the transversal space $\fX$,  for $\ell \geq 1$, introduced in Section~\ref{sec-codes}, give rise to a ``presentation'' of $\fM$ by what we call the \emph{Thomas tubes}. In  Section~\ref{sec-projections}, we derive the solenoidal structure of $\fM$ from this data.

Suppose that the equicontinuous matchbox manifold $\fM$ is a \emph{Cantor bundle}, $\pi_0 \colon \fM \to M_0$, as discussed in Section~\ref{sec-intro}. That is, we assume  there exists a compact   manifold $M_0$, finitely presented group $\G = \pi_1(M_0 , b_0)$, Cantor set $\fF_b = \pi_0^{-1}(b)$, and a minimal, equicontinuous  action   $\varphi \colon \G \times \fF_b \to \fF_b$ so that $\fM$ is homeomorphic to the suspension of the action $\varphi$. Consider   $V^{\ell} \subset \fF_b \cong \fT_*$,  and   introduce the subgroup
$$\G_{\ell} \equiv \left\{\gamma \in \G \mid  \varphi(\gamma)(V^{\ell}) = V^{\ell} \right\} ~ .$$
Then the  ``Thomas tube''  associated to $V^{\ell}$ is homeomorphic to  the suspension of the action of $\G_{\ell}$ on $V^{\ell}$, which is a Cantor bundle over the covering space $M_{\ell} \to M_0$ associated to $\G_{\ell} \subset \G$.

 The complication which arises for the more general case  is that one requires a ``foliated product theorem'' in place of the suspension construction, as the leaves of $\F$ are not given as covering  spaces of a given fixed base manifold $M_0$. We introduce the notion of \emph{foliated microbundles} in the context of matchbox manifolds to obtain such a product theorem.

 \subsection{Foliated microbundles}\label{subsec-microbundles}

The ``foliated microbundle'' associated to a leaf in a foliated space is one of the most basic concepts, originating with the first works of Reeb (see Milnor \cite{Milnor2009} for a discussion).
 Its construction   is a generalization of that for   the holonomy map $h_{\cI_{\gamma}}$ associated to a leafwise curve $\gamma$, as in Section~\ref{subsec-admissible}, in that it follows essentially the same procedure, but \emph{uniformly} for all paths in a given leaf.
We give this construction for foliated spaces;   then for the case where $L_x$ is dense and $\F$ is equicontinuous, it provides a framework for analyzing the structure of $\fM$.

Recall that $w_0 \in int(\fT_1)$ is the fixed base-point of Section~\ref{sec-codes}.
Let $x_0 = \tau_1(w_0) \in U_1$ and $L_0$ be the leaf through $x_0$. Let $h_{\F, x_0} \colon \pi_1(L_0 , x_0) \to \cGF^{w_0}$ denote the holonomy homomorphism of $L_0$, whose kernel $\cK_0 \subset \pi_1(L_{x_0} , x_0)$ of $h_{\F,x_0}$ is a normal subgroup.
Let $\Pi \colon \wtL_{0} \to L_0$ be the covering associated to $\cK_0$ and choose $\wtx_0 \in \wtL_0$ such that $\pi(\wtx_0) = x_0$. By definition, given any closed path $\wtgamma \colon [0,1] \to \wtL_0$ with basepoint $\wtx_0 = \wtgamma(0) = \wtgamma(1)$, the image of $\wtgamma$ in $L_0$ has trivial germinal holonomy as a leafwise path in $\fM$.
It follows that the holonomy map  defined by a path $\wtgamma$  in $\wtL_0$ starting at $\wtx_0$  is  determined by the endpoint $\wtgamma(1)$.

The construction of the foliated microbundle associated to $L_0$ begins with the  selection of  a collection of points in $\wtL_0$ which are ``sufficiently dense in $\wtL_0$'' to capture  the holonomy of $L_0$. This is assured by choosing a suitably fine net in $L_0$ and then lifting this to a net in $\wtL_0$.

\begin{defn} \label{def-net}
Let $(X, d_X)$ be a complete separable metric space. Given
$0 < e_1 < e_2$, a subset $\cN \subset X$ is an \emph{$(e_1 , e_2)$-net} (or \emph{Delone set}) if:
\begin{enumerate}
\item  $\cN$ is $e_1$-separated: for all $y \ne z \in \cN$, $e_1 \leq d_{X}(y,z)$;
\item  $\cN$ is $e_2$-dense: for all $x \in X$, there exists some $z \in \cN$ such that $d_{X}(x,z) \leq e_2$.
\end{enumerate}
\end{defn}
It is a standard fact that given a separable, complete metric space $X$ and any $e_2 > 0$, there exists $0 < e_1 < e_2$ and a $(e_1 , e_2)$-net $\cN \subset X$.

Recall that $\eFU$ defined by (\ref{eq-leafdiam}) was chosen so that every leafwise disk of radius $\eFU$ is contained in a metric ball of $\fM$ of radius $\eU/2$. That is, for all $y \in \fM$, $D_{\F}(y, \eFU) \subset D_{\fM}(y, \eU/2)$.

Let $e_2 = \eFU/4$, then choose $\cN_0 \subset L_0$   an $(e_1 , e_2)$-net for $L_0$
for some $0 < e_1 < \eFU/4$.
We can assume without loss of generality that $x_0 \in \cN_0$. Condition~(\ref{def-net}.2) implies that the collection of leafwise open disks
$\ds \{B_{\F}(z , \eFU/2) \mid z \in \cN_0 \}$ is a covering of $L_0$.

For each $z \in \cN_0$, choose an index $1 \leq i_z \leq \nu$ so that $B_{\fM}(z, \eU) \subset U_{i_z}$. Without loss, we can assume that $B_{\fM}(x_0, \eU) \subset U_{1}$.
Then note that for all $z' \in D_{\F}(z, \eFU)$, we have $z' \in  B_{\fM}(z, \eU/2)$ so
the triangle inequality implies that
\begin{equation}\label{eq-containment}
D_{\F}(z', \eFU) \subset B_{\fM}(z', \eU/2) \subset B_{\fM}(z, \eU) \subset U_{i_z} ~ .
\end{equation}

\begin{lemma}\label{lem-newcover}
The collection $\{U_{i_z} \mid z \in \cN_0\}$ is a subcover for $\fM$, with Lebesgue number $\eU/3$.
\end{lemma}
\proof
Let $y \in \fM$, then $L_0$ is dense so there exists $y' \in L_0$ with $d_{\fM}(y,y') < \eU/6$.
Let $z \in \cN_0$ with $d_{\F}(y', z) < \eFU/2$.
Then $y' \in B_{\fM}(z, \eU/2)$,
hence $y \in B_{\fM}(z, \eU) \subset U_{i_z}$ by (\ref{eq-containment}).
%In fact, for all $y'' \in B_{\fM}(y, \eU/3)$ this shows that $y'' \in   B_{\fM}(z, \eU) \subset U_{i_z}$ so that
Therefore, $B_{\fM}(y, \eU/3) \subset U_{i_z}$.
\endproof

Let $\wtcN_0 = \Pi^{-1}(\cN_0)$ which is a $(e_1 , e_2)$-net for $\wtL_0$ with the Riemannian metric lifted from $L_0$.
The points of $\wtcN_0$ are denoted by $\wtz$, where $\Pi(\wtz) = z \in \cN_0$.
In particular, $\wtx_0 \in \wtcN_0$ as $\Pi(\wtx_0) = x_0 \in \cN_0$.

For each $\wtz \in \wtcN_0$, let $z = \Pi(\wtz)$ and set $\wtU_{\wtz} = \oU_{i_z} \times \{\wtz\}$.
For $(x, \wtz) \in \wtU_{\wtz}$ define  $\Pi \colon \wtU_{\wtz} \to \oU_{i_z}$ by $\Pi(x, \wtz) = x$.
For $\wtz \ne \wtz' \in \wtcN_0$ with $\Pi(\wtz) = \Pi(\wtz') = z$,
the sets $\wtU_{\wtz}$ and $\wtU_{\wtz'}$ are   disjoint by definition, though their projections to $\fM$ agree.

For $\wtz \in \wtcN_0$ and $\wty = (x, \wtz)  \in \wtU_{\wtz}$, let
$\wtcP_{\wtz}(\wty) = \cP_{i_z}(x) \times \{\wtz\}$ denote the plaque of $\wtU_{\wtz}$ containing $\wty$. If $x \in \cP_{i_z}(z)$ then we abuse notation and identify $\wtcP_{\wtz}(\wty)$ with  the plaque of $\wtL_0$ containing $\wtz$.
Note that $B_{\wtL_0}(\wtz , \e^{\F}_{\cU}) \subset \wtcP_{\wtz}(\wtz)$ for each $\wtz \in \wtcN_0$,
so the collection
$\ds \{\wtcP_{\wtz}(\wtz) \mid \wtz \in \cN_0 \}$ is a covering of $\wtL_0$.

One thinks of the plaques $\wtcP_{\wtz}(\wtz)$ as ``convex tiles'', and the collection $\{ \wtcP_{\wtz}(\wtz) \mid \wtz \in \wtcN_0\}$ as a ``tiling'' of $\wtL_0$. The interiors of the plaques need not be disjoint, so this is not a proper tiling in the usual sense (for example see \cite{AP1998, BG2003, FHK2002, Senechal1995}, or \cite[\S 11.3.C]{CandelConlon2000}). In particular, the combinatorics of the covering of $\wtL_0$ by plaques is not a consequence of the geometry of the ``tiles'', but rather is determined by the dynamical properties of the leaf $L_0$. (See \cite{AH1996,Hurder1994} for a discussion of this point of view.)
The net $\wtcN_0$ can also be used to generate Voronoi decompositions of $\wtL_0$, as in the work \cite{CHL2011a}.

The \emph{foliated microbundle} of $\wtL_0$ is the foliated space
$\ds \wtfN_0 = \left\{ \cup_{\wtz \in \wtcN_0} ~ \wtU_{\wtz}\right\} {\Big \slash} \sim$ ,
where $\wty \in \wtU_{\wtz}$ and $\wty' \in \wtU_{\wtz'}$ are identified
if $\Pi(\wty) = \Pi(\wty')$ and $\wtcP_{\wtz}(\wtz) \cap \wtcP_{\wtz'}(\wtz') \ne \emptyset$.
Let $\wtF$ denote the foliation whose leaves are the path components of $\wtfN_0$.

For each $\wtz \in \wtcN_0$, let $\wtfT_{\wtz} = \fT_{i_{\wtz}}$.
The composition
$\ds \wtvp_{\wtz} \equiv \vp_{i_z} \circ \Pi \colon \wtU_{\wtz} \to [-1,1]^n \times \fT_{\wtz}$
defines a foliated coordinate chart on $\wtfN_0$ for $\wtF$.
Let $\wtpi_{\wtz} \colon \wtU_{\wtz} \to \fT_{\wtz}$ be the normal coordinate,
and $\wtlambda_{\wtz} \colon \wtU_{\wtz} \to [-1,1]^n$ be the leafwise coordinate.

Given $\wtz \in \wtcN_0$, subset $V \subset \fT_{\wtz}$ and $\xi \in [-1,1]^n$, we obtain a local section for $\wtF$ by
\begin{equation}
\wttau_{\wtz , \xi} \colon V \to \wtU_{\wtz} ~ , ~ \wttau_{\wtz, \xi}(w) = \wtvp_{\wtz}^{-1}(\xi, w) = (\vp_{i_z} ^{-1}(\xi, w) , \wtz)
\end{equation}

 The  foliated microbundle   can be viewed as constructing, a uniform setting,    all of the holonomy maps for paths in the leaf $\wtL_0$.
To be precise,  we say that a path $\wtgamma \colon [0,1] \to \wtL_0$ is  \emph{nice}, if there exists a partition
$a = s_0 < s_1 < \cdots < s_{\alpha} = b$ such that for each $0 \leq \ell \leq \alpha$, the restriction
$\wtgamma \colon [s_{\ell}, s_{\ell + 1}] \to \wtL_0$ is a geodesic segment between points
$\wtz_{\ell} = \wtgamma(s_{\ell}), \wtz_{\ell +1} = \wtgamma(s_{\ell +1})\in \wtcN_0$ with $d_{\F}(\wtz_{\ell}, \wtz_{\ell +1}) < \eFU$.
Then $\cI = (i_{\wtz_{0}}, \ldots , i_{\wtz_{\alpha}})$ is an admissible sequence for \emph{both} $\wtF$ and $\F$, and so defines holonomy maps $\wth_{\cI}$ for $\wtF$ and $h_{\cI}$ for $\F$.
Clearly, $\wth_{\cI}$ is just the lift of $h_{\cI}$, and $h_{\cI}$ is the holonomy map for the leafwise path $\gamma = \Pi \circ \wtgamma$ constructed in Section~\ref{sec-holonomy}.
As before, we note that $\wth_{\cI}$ depends only on the endpoints of $\cI$. For $\wtz \in \wtcN_0$ let $\wth_{\wtz}$ denote the holonomy along some nice path $\wtgamma_{\wtz}$ from $\wtx_0$ to $\wtz$, considered as a transformation of the space $\wtfT$, which is the disjoint union of the local transversals  $\fT_{\wtz}$.
Let $h_{\wtz}$ denote the holonomy along the projected path, $\gamma_{\wtz} = \Pi \circ \wtgamma_{\wtz}$.

Recall that $W \subset B_{\fX}(w_0, \dTU/2) \subset \fT_1$ is the clopen neighborhood of $w_0$ chosen in Section~\ref{sec-codes}.
\begin{lemma}
Let $\wtgamma \colon [0,1] \to \wtL_0$ be a nice path with $\wtgamma(0) = \wtx_0$. Then $W \subset D(\wth_{\cI})$.
\end{lemma}
\proof
This follows directly from Proposition~\ref{prop-uniformdom} and the definition of $\dTU$.
\endproof

 \subsection{Thomas tubes}

For each $\ell \geq 1$, let $V^{\ell} \subset W$ be the clopen subset defined by (\ref{eq-codedef-ell}).
For $\wtz \in \wtcN_0$ define
\begin{equation}\label{def-basicblocks}
 V^{\ell}_{\wtz} =   h_{\wtz}(V^{\ell}) \subset \fT_{i_{\wtz}} ~ , ~  \wtV^{\ell}_{\wtz} = \wth_{\wtz}(V^{\ell}) =  V^{\ell}_{\wtz} \times \{\wtz\} \subset \fT_{\wtz}  ~ .
\end{equation}
The union of the sets $\wtV^{\ell}_{\wtz}$ is just the saturation of $\wtV^{\ell}$ under the action of the pseudogroup $\cG_{\wtF}$.

Lemma~\ref{lem-fixset-ell} implies that if $\wtV^{\ell}_{\wtz} \cap \wtV^{\ell}_{\wtz'} \ne \emptyset$, then
$V^{\ell}_{\wtz} = V^{\ell}_{\wtz'}$ although this need not imply that $\wtz = \wtz'$ if $L_0$ has non-trivial germinal holonomy. The sets
$\wtV^{\ell}_{\wtz}$ and $\wtV^{\ell}_{\wtz'}$ are   disjoint if $\wtz \ne \wtz'$.

Also, introduce the local coordinate chart saturation of each of these sets:
\begin{equation}\label{eq-coordinatecovers}
\fU^{\ell}_{\wtz} =  \pi_{i_{\wtz}}^{-1}(V^{\ell}_{\wtz}) \subset \oU_{i_{\wtz}} ~, ~ \wtfU^{\ell}_{\wtz} = \wtpi_{\wtz}^{-1}(\wtV^{\ell}_{\wtz}) = \fU^{\ell}_{\wtz}  \times \{\wtz\} \subset \wtU_{\wtz}    ~ .
\end{equation}
Then $\fU^{\ell}_{\wtz} $ is the union of the plaques in $\oU_{i_{\wtz}}$ through the points of $V^{\ell}_{\wtz}$.

\begin{defn} \label{def-thomas}
The \emph{Thomas tube} associated with  $V^{\ell}$ is the subset of the microbundle $\wtfN_0$,
\begin{equation} \label{eq-thomas}
\wtfN_{\ell} ~ = ~ \bigcup_{\wtz \in \wtcN_0} ~ \wtfU^{\ell}_{\wtz}
\end{equation}
\end{defn}
In the case that $\ell = \ell_0$ note that  $\wtfN_{\ell_0} = \wtfN_0$. For all $\ell \geq \ell_0$, the image $\Pi(\wtfN_{\ell}) \subset \fM$ is the saturation by $\F$ of the clopen set $V^{\ell}$, hence $\Pi(\wtfN_{\ell}) = \fM$.

Note that $\wtfN_{\ell}$ is a (non-compact) foliated space whose leaves $\wtL$ are coverings of corresponding leaves of $\F$. That is, the restriction
$\Pi \colon \wtL \to L$ is a smooth covering map, which is a local isometry for the induced leafwise metric on $\wtfN_{\ell}$.
Also, the leaf space for $\wtfN_{\ell} $  is homeomorphic to    $V^{\ell}$ by construction, and
 for $\ell' > \ell$ the inclusion $V^{\ell'} \subset V^{\ell}$ induces a natural inclusion $\wtfN_{\ell'} \subset \wtfN_{\ell}$.

\section{Solenoidal structure for equicontinuous foliations}\label{sec-projections}

In this section, we show that if $\fM$ is an equicontinuous matchbox manifold, then it has  a presentation   as an inverse limit, and thus is homeomorphic to a generalized solenoid as in  (\ref{eq-solenoid}).

The strategy of the proof begins with an observation, that if we assume that $\fM$ is homeomorphic to a solenoid $\cS$, then  the bonding maps $p_{\ell+1} \colon M_{\ell +1} \to M_{\ell}$ induce, for each $\ell \geq 0$,  a map $q_{\ell} \colon \fM \to M_{\ell}$ which is a fibration \cite{McCord1965}. For each $x \in M_{\ell}$ the   fiber   $K_{\ell}(x) \equiv q_{\ell}^{-1}(x)$ is an embedded Cantor set in $\fM$, and the fibration structure implies that the family of  Cantor sets $\{K_{\ell}(x) \mid x \in M_{\ell}\}$ form what we call a \emph{Cantor foliation transverse to $\F$}. Moreover, the property $q_{\ell} = p_{\ell +1} \circ q_{\ell +1 }$ of a solenoid imply that the  Cantor foliations associated to $q_{\ell}$ and $q_{\ell + 1}$ are naturally related.
We show that given these consequences, then $\fM$ is homeomorphic to a solenoid.

We begin with a definition. Recall that we assume there  is   a fixed regular covering $\{U_{i} \mid 1 \leq i \leq \nu\}$ of $\fM$ by foliation charts, as in Definition~\ref{def-regcover}, with charts $\vp_i \colon \oU_i \to [-1,1]^n \times \fT_i$ where $\fT_i \subset \fX$ is a clopen subset. By construction, each chart admits a foliated  extension $\wtvp_i \colon \wtU_i \to (-2,2)^n \times \fT_i$ where $\oU_i \subset \wtU_i \subset \fX$ is an open neighborhood of $\oU_i$ and $\wtvp_i | \oU_i = \vp_i$.

\begin{defn}\label{def-cantorfol}
Let $\fM$ be a matchbox manifold. We say that $\fM$ admits  a \emph{Cantor foliation $\cH$ transverse to $\F$} if there exists an equivalence relation $\approx$ on $\fM$ such that:
\begin{enumerate}
\item for $x \in \fM$ the class $H_x = \{y \in \fM \mid y \approx x\}$ is a Cantor set;
\item for each $x \in U_i$ with   $w = \pi_i(x) \in \fT_i$ there exists a clopen neighborhood $w \in V_x \subset \fT_i$
and a homeomorphism $\Phi_x \colon [-1,1]^n \times V_x \to \wtU_i$ such that, for each $\xi \in [-1,1]^n$, the image $\Phi_x(\{\xi\} \times V_x) \subset \oU_i$ is a complete equivalence class.
\end{enumerate}
The leaves of the ``foliation'' $\cH$ are   defined to be the equivalence classes of $\approx$.
\end{defn}
We call  $V_x$ the model space for $\cH$ at $x$.  For a standard foliation, the space $V_x$ would be homeomorphic to  $(-1,1)^n$, while for a Cantor foliation, it is a homeomorphic to a Cantor set.

Condition~\ref{def-cantorfol}.1 implies the leaves of $\cH$ are Cantor sets, and Condition~\ref{def-cantorfol}.2 states that these leaves are ``vertical'' segments for a regular coordinate chart, after reparametrization by the maps $\Phi_x$. In other words, every point $x\in \fM$ admits what is sometimes called in the foliation literature, a ``bi-foliated neighborhood'', where the leaves of $\F$ correspond to the ``horizontal'' Euclidean slices of $[-1,1]^n \times V_x$, and the leaves of $\cH$ correspond to the ``vertical'' Cantor set slices.

 For example, if $\pi \colon \fM \to M$ is a Cantor bundle over a compact manifold $M$, then the fibers of $\pi$ define a   Cantor foliation of $\fM$
 which is transverse to the foliation $\F$ of $\fM$. As a Cantor bundle need not be a solenoid, the existence of the transverse foliation $\cH$ is clearly not sufficient to show that $\fM$ is homeomorphic to  a solenoid. What is required, in addition, is that there exists a sequence $\{\cH_{\ell}  \mid \ell \geq \ell_0\}$ of nested Cantor foliations, which in our situation is provided by constructing Cantor foliations \textit{adapted} to the Thomas tubes of Definition~\ref{def-thomas}.

For $\ell \geq 1$,  let $\wtfN_{\ell}$ be the    Thomas tube with transversal model $V^{\ell}$, with notation as in Section~\ref{sec-tubes}.
Recall that the foliated space $\wtfN_{\ell}$   contains the holonomy covering $\wtL_0$ of the leaf $L_0 \subset \fM$ corresponding to $w_0$.
Also, for $\ell' > \ell$ we have $\wtfN_{\ell'} \subset  \wtfN_{\ell}$ is a foliated subspace.

\begin{defn}\label{def-adapted}
We say that a transverse Cantor foliation $\cH_{\ell}$ on $\fM$ is \emph{adapted} to $\wtfN_{\ell}$ if, for each $\wtz \in \wtN_0$ and $x \in \fU^{\ell}_{\wtz} \subset \fM$, we can choose $V_x = V^{\ell}_{\wtz}$ in Definition~\ref{def-cantorfol}.2.
\end{defn}

Note that if $\cH_{\ell}$ is adapted to $\wtfN_{\ell}$, then for each foliated coordinate chart $\wtfU^{\ell}_{\wtz}$ for $\wtF$, the leaves of $\cH_{\ell}$ form complete transversals to the image of the restriction  $\Pi \colon \wtfU^{\ell}_{\wtz} \to \fU^{\ell}_{\wtz}$.
It follows that $\cH_{\ell}$ induces a transverse Cantor foliation $\wtcH_{\ell}$ on $\wtfN_{\ell}$. Actually, this is evident from the definitions as well.

Given such $\cH_{\ell}$,  let $\approx_{\ell}$ denote the equivalence relation on   $\fM$ defined by its leaves, and by a small abuse of notation, we also let $\approx_{\ell}$ denote the corresponding   equivalence relation on  $\wtfN_{\ell}$.

Observe that for $\ell' > \ell$, the restriction of $\approx_{\ell}$ to the foliated subspace $\wtfN_{\ell'} \subset  \wtfN_{\ell}$ defines an equivalence relation, which is denoted by $\approx_{\ell'}$. The Cantor foliation $\wtcH_{\ell'}$ of $\wtfN_{\ell'}$
defined by   $\approx_{\ell'}$ is the lift of an adapted  transverse Cantor foliation $\cH_{\ell'}$ on $\fM$. The model sets $V_x'$ for $\cH_{\ell'}$ are given by the collection of translates   $\{V^{\ell'}_{\wtz} \mid \wtz \in \wtN_0\}$.

If $\pi \colon \fM \to M_0$ is an equicontinuous Cantor bundle over a compact manifold $M_0$  then the fibers of $\pi$ define a Cantor foliation of $\fM$ which adapts to each Thomas tube $\wtfN_{\ell}$, for $\ell \geq 1$.  For example,  this is the case studied in in the work  \cite{Clark2002} by the first author, for $M_0 = \mT^n$.
 In  general, the existence of an adapted  transverse Cantor foliation on $\fM$ is not ``obvious'', though in fact one can show:

\begin{thm}[\cite{CHL2011a}] \label{thm-foliate}
Let $\fM$ be an equicontinuous matchbox manifold.
Then for some $\ell_0 \geq 1$,   there exists a transverse Cantor foliation $\cH_{\ell_0}$ on $\fM$ adapted to  $\wtfN_{\ell}$.
\end{thm}

The idea of the proof of  this result  is straightforward enough. For each foliated coordinate chart, $\oU_i$ or a subchart
$\fU^{\ell}_{\wtz}  \subset \oU_i$, there is a natural ``vertical'' foliation whose leaves are the images of the transversals
$\tau_{\wtz , \xi}$. The problem is that on the overlap of two charts, these vertical foliations need not match up, as  the requirement on a foliation chart for $\F$ is that the horizontal plaques match up. The exception is when $\fM$ is given with a fibration structure, then the coordinates can be chosen to be adapted to the fibration structure, and so the fibers of the bundle are compatible on overlaps.

For the general case, the idea is then to subdivide the horizontal plaques into small enough regions, and restrict the   diameters of the model set $V_x = V^{\ell}_{\wtz}$, so that the   vertical leaves become sufficiently close on overlaps, so that they can be made compatible on overlaps. More precisely, one constructs a uniform triangulation of the leaves of $\F$ on $\fM$ so that the triangles have sufficiently small diameter and in ``general position'', so that they are stable in transverse directions, for small perturbations. Then, the vertical foliations are defined using barycentric coordinates based on each simplex in the triangulation. The functions $\Phi_x$ introduced in Condition~\ref{def-cantorfol}.2 are the adjustments to the vertical foliation needed to make the foliations match up. The requirement that the images of the maps $\Phi_x$ be allowed to take values in the open neighborhood $\wtU_i$ is due to the fact that on the boundary points of $\oU_i$, the leaves of the  foliation $\cH$ need not have constant horizontal coordinate   $\lambda_i$.

A uniform triangulation of the leaves (satisfying the required stability conditions above) is constructed as the Delaunay simplicial complex associated to a very fine Voronoi tessellation of the leaves. The proof that all this can be done is quite tedious,   and  uses only ``elementary techniques'', along with effective estimates in each stage of the process. The details as given in \cite{CHL2011a} are quite lengthy and  involved.

 \medskip

 Given Theorem~\ref{thm-foliate}, we     complete the proof of Theorem~\ref{thm-main2}.
 Let $\fM$ be an equicontinuous matchbox manifold.
Assume that $\ell_0$ is such that  there exists an adapted transverse Cantor foliation $\cH_{\ell_0}$ on $\fM$. Let $\approx_{\ell}$ denote the restricted equivalence relation on $\fM$ adapted to  $\wtfN_{\ell}$.   For each $\ell \geq \ell_0$, introduce the quotient spaces $\wtM_{\ell} ~ \equiv ~ \wtfN_{\ell}/\approx_{\ell}$ and $M_{\ell} \equiv  \fM /\approx_{\ell}$. Given a point $x \in \fM$, let $[x]_{\ell}  \in M_{\ell}$ denote its $\approx_{\ell}$ equivalence class.

\begin{prop}\label{prop-structure1} Let $\fM$ be an equicontinuous matchbox manifold, with $\ell_0$ as specified in Theorem~\ref{thm-foliate}. Then for all  $\ell \geq \ell_0$,    $M_{\ell}$ is a closed $n$-dimensional topological manifold.
\end{prop}
\proof
For each $x \in \fM$, the  equivalence class $[x]_{\ell}$ is compact, for  the quotient topology, $M_{\ell}$ is a Hausdorff topological space. As $\fM$ is compact and connected, $M_{\ell}$ is also compact and connected.

For $\wtz \in \wtcN_0$ and $x \in \fU^{\ell}_{\wtz} \cap U_{i_{\wtz}}$,  recall that    $\cP_{\wtz}(x)$ is an open plaque of $\F$ containing  $x$.
The restriction $Q_{\ell} \colon \cP_{\wtz}(x) \to M_{\ell}$ is a homeomorphism onto its image by Condition~\ref{def-cantorfol}.2. We define    the composition
\begin{equation}
\phi_{\wtz} ~ \colon ~ (-1,1)^n \cong  (-1,1)^n  \times \pi_{i_{\wtz}}(x) \subset (-1,1)^n \times V^{\ell}_{\wtz} \stackrel{\vp_{i_{\wtz}}}{\longrightarrow} \cP_{\wtz}(x) \stackrel{Q_{\ell}}{\longrightarrow}  Q_{\ell}(\cP_{\wtz}(x) )
\end{equation}
which is a coordinate neighborhood of $[x]_{\ell}$.  Moreover, if  there exists $\wtz'$ such that  $x \in \fU^{\ell}_{\wtz} \cap  \fU^{\ell}_{\wtz'} \cap U_{i_{\wtz}}$, then the change of coordinate map $\phi_{\wtz'} \circ \phi_{\wtz}$ is a homeomorphism, as each of the maps $\Phi_{i_{\wtz}}$ and $\Phi_{i_{\wtz'}}$ in  Condition~\ref{def-cantorfol}.2 are homeomorphisms.
 Thus, $M_{\ell}$ has the structure of a topological manifold.
\endproof

We observe a basic point about the transverse foliations $\wtcH_{\ell}$ induced on the Thomas tube $\wtcH_{\ell}$.
\begin{lemma}\label{lem-product}
The inclusion map induces   a homeomorphism $\iota_{\ell} \colon \wtL_0 \cong \wtcH_{\ell}/\approx_{\ell}$.
\end{lemma}
\proof
Each leaf  $\wtL \subset \wtcH_{\ell}$  intersects each transversal leaf of $\wtcH_{\ell}$ in exactly one point, so the map $\iota_{\ell}$ induced by the inclusion is a 1-1 onto map. The equivalences classes for $\approx_{\ell}$ in $\wtcH_{\ell}$ are   compact, so the quotient topology is Hausdorff. The    map $\iota_{\ell}$ is continuous as it is induced by the inclusion of plaques, hence it is a homeomorphism.
\endproof

Let  $Q_{\ell} \colon \fM \to M_{\ell}$, given by $Q_{\ell}(x) = [x]_{\ell}$, denote the quotient map.

\begin{lemma}\label{lem-structure2} For all  $\ell \geq \ell_0$,    the restriction $\wtpi_{\ell} = Q_{\ell} | \wtL_0 \colon \wtL_0 \to M_{\ell}$ is a covering map.
\end{lemma}
\proof
The map $\wtpi_{\ell} \colon \wtL_0 \to M_{\ell}$ is a homeomorphism when restricted to each plaque $\wtcP_{\wtz}(\wtz)$ so the map is a local homeomorphism of a complete Hausdorff topological space, hence is a covering map.
\endproof

\begin{lemma}\label{lem-structure3}  For all  $\ell' > \ell \geq \ell_0$,
  there is a   covering map $q_{\ell', \ell} \colon M_{\ell'} \to M_{\ell}$ such that, for $q_{\ell} \equiv   q_{\ell, \ell_0}$, the   maps satisfy
$q_{\ell'} = q_{\ell} \circ q_{\ell', \ell}$.
\end{lemma}
\proof
For $\ell' > \ell \geq \ell_0$ define $\wtpi_{\ell} = q_{\ell} \circ  \Pi  \colon \wtL_0 \to M_{\ell}$. By the definitions  of the equivalence relations $\approx_{\ell}$ on $\fM$ and $\wtcN_0$,  the   map $\wtpi_{\ell}$ can be factored as a composition
\begin{equation}
\wtL_0 ~ \longrightarrow ~ \wtcN_0 ~ \longrightarrow ~ \fM /\approx_{\ell'} \, \equiv \, M_{\ell'} ~ \longrightarrow ~ \fM /\approx_{\ell} \, \equiv \, M_{\ell}
\end{equation}

Define $q_{\ell, \ell'} \colon \fM /\approx_{\ell'} \to \fM /\approx_{\ell}$ to be the natural quotient map, then $q_{\ell'} = q_{\ell} \circ q_{\ell', \ell}$ follows.
\endproof

\medskip

Recall that the inverse limit of the sequence of maps
$\ds
\left\{ q_{\ell,\ell+1} \colon M_{\ell+1} \to M_{\ell} \mid \ell \geq \ell_0 \right\}
$
is the topological space
\begin{equation}\label{eq-factoring}
\cS\{q_{\ell,\ell'} \colon M_{\ell'} \to M_{\ell}\} ~ \equiv ~ \left\{ \omega = (\omega_{\ell_0} , \omega_{\ell_0 +1} , \ldots) \in \prod_{\ell = \ell_0}^{\infty} ~ M_{\ell} ~ | ~ q_{\ell, \ell+1}(\omega_{\ell +1}) = \omega_{\ell} \right\}
\end{equation}

\medskip

\begin{prop}\label{prop-invlim} Let $\fM$ be an equicontinuous matchbox manifold, with $\ell_0$ as specified in Theorem~\ref{thm-foliate}. Then there is a homeomorphism $q \colon \fM \to \cS\{q_{\ell,\ell'} \colon M_{\ell'} \to M_{\ell}\}$ of foliated spaces.
\end{prop}
\proof
For $x \in \fM$ define
$$q(x) = ([x]_{\ell_0} , [x]_{\ell_0 +1}, \ldots ) \in \cS\{q_{\ell,\ell'} \colon M_{\ell'} \to M_{\ell}\}$$
which is  well-defined   by Lemma~\ref{lem-structure3}.
The map to each factor, $x \mapsto q_{\ell}(x) = [x]_{\ell}$ is continuous by Lemma~\ref{lem-structure2} and (\ref{eq-factoring}), hence the map $q(x)$ is  continuous.

Finally, let $x, y \in \fM$ such that $q(x) = q(y)$.
Then $ q_{\ell}(x) = q_{\ell}(y)$ for all $\ell \geq \ell_0$.
That is, $x \approx_{\ell} y$ for all $\ell \geq \ell_0$. Define
\begin{equation}\label{eq-mu}
\mu_{\ell} \equiv \max ~ \left\{  d_{\fM}(x,y) \mid x \approx_{\ell} y ~ , ~ x,y \in \fM \right\}
\end{equation}

  For all $\wtz \in \wtN_0$ the diameter of $V^{\ell}_{\wtz} \subset \fT_* \subset \fX$ is bounded above by $\e_{\ell}$ by the inductive construction in Section~\ref{sec-codes} of these sets. By Lemma~\ref{lem-modtau}, the diameter of the sections  $\tau_{i_{\wtz}}(V^{\ell}_{\wtz})$ are bounded above by $\rt(\e_{\ell})$.  As $\e_{\ell} \to 0$ as $\ell \to \infty$,  we also have that their diameters $\rt(\e_{\ell}) \to 0$ as $\ell \to \infty$. Finally, there is a finite collection of maps $\Phi_{i_{\wtz}}$ which arise in the Condition~\ref{def-cantorfol}.2 for the fixed $\ell_0$, hence they have a uniform modulus of continuity.
  Thus  $\mu_{\ell} \to 0$ as  $\ell \to \infty$ and consequently, the map $q$ is injective.
 \endproof

Combining the above results, we obtain:
\begin{thm}\label{thm-imasolenoid}
Let $\fM$ be an equicontinuous matchbox manifold, let $\ell_0$ be defined by  Theorem~\ref{thm-foliate}, and let $M_{\ell}$ be defined as above. Then $\fM$ is homeomorphic   to the solenoid $\cS\{q_{\ell,\ell'} \colon M_{\ell'} \to M_{\ell}\}$  defined by the bonding maps $q_{\ell, \ell +1} \colon M_{\ell+1} \to M_{\ell}$. \hfill $\Box$
\end{thm}

\section{Homogeneous matchbox manifolds}\label{sec-mccord}

An  {equicontinuous} matchbox manifold $\fM$ has the structure of a solenoid by   Theorem~\ref{thm-imasolenoid}, although it need not be a McCord solenoid.    In this section, we consider the case where $\fM$ is homogeneous, and therefore is equicontinuous and without germinal holonomy. The homogeneous  hypothesis, and     Corollary~\ref{cor-effros} of the  Effros Theorem,  implies special normality properties for the conjugation actions of the fundamental groups in the solenoidal tower (see (\ref{eq-tower}) below), which then implies  Theorem~\ref{thm-main1}, that $\fM$ is homeomorphic to a  McCord solenoid.

Note that Theorem~\ref{thm-main1}    follows directly from Theorem~3 in Fokkink and Oversteegen  \cite{FO2002}, that a homogeneous solenoid is McCord, and the proof we give is ``essentially'' the same. We include the proof for matchbox manifolds here, as the key idea  follows naturally from     our previous results.

Let $\fM$ be an homogeneous matchbox manifold.
We follow the notations of the previous sections.
Recall that $x_0 \in \fM$ is the fixed base-point, and $L_0$ is the leaf in $\fM$ containing $x_0$.
As $\F$ has no holonomy, we can identify its holonomy covering $\wtL_0$ with $L_0$.
We assume that the equivalence relations $\approx_{\ell}$ as in Section~\ref{sec-projections} have been defined for all $\ell \geq \ell_0$.
Then the basepoint for $M_{\ell}$ is the equivalence class $[x_0]_{\ell}$.
For $\ell \geq \ell_0$, set $H_{\ell} = \pi_1(M_{\ell}, [x_0]_{\ell})$.

The bonding maps of the solenoid $\ds \cS\{q_{\ell,\ell'} \colon M_{\ell'} \to M_{\ell}\}$ induce  homomorphisms of fundamental groups,
\begin{equation}\label{eq-bonding}
q_{\ell} = q_{\ell_0, \ell} \colon H_{\ell} \to H_{\ell_0} \quad , \quad p_{\ell} \colon H_{\ell} \to H_{\ell -1} \quad , \quad q_{\ell} =q_{\ell -1}   \circ  p_{\ell}
\end{equation}
so we obtain a tower   of groups for $\ell > \ell_0$,
\begin{equation}\label{eq-tower}
\cdots \longrightarrow H_{\ell +1} \stackrel{p_{\ell +1}}{\longrightarrow} H_{\ell} \longrightarrow \cdots \longrightarrow H_{\ell_0}
\end{equation}

 Let $\cH_{\ell} = q_{\ell}(H_{\ell}) \subset \cH_{\ell_0} \equiv H_{\ell_0}$ which results in a descending chain of subgroups of finite index,
 \begin{equation}\label{eq-towersubgroups}
\cdots \subset \cH_{\ell +1} \subset \cH_{\ell} \subset \cdots \subset \cH_{\ell_0+1} \subset \cH_{\ell_0}
\end{equation}
Each manifold $M_{\ell}$ is then naturally homeomorphic to the covering of $M_{\ell_0}$ defined by the subgroup $\cH_{\ell}$, and the homeomorphism type of $\ds \cS\{q_{\ell,\ell'} \colon M_{\ell'} \to M_{\ell}\}$ is determined by the chain (\ref{eq-towersubgroups}).

The claim of Theorem~\ref{thm-main1} is that $\fM$ is \emph{homeomorphic} to a McCord solenoid, and this is an essential point, as we use a well-known criteria for when two  inverse limit spaces with the same base space $M_{\ell_0}$ are homeomorphic.

\begin{thm}\label{thm-equivalence}
Suppose that $\fM \cong \cS\{q_{\ell,\ell'} \colon M_{\ell'} \to M_{\ell}\}$ with subgroups $\cH_{\ell} \subset \cH_{\ell_0}$ for $\ell > \ell_0$ defined as above. Suppose there is given a second chain of subgroups of finite index,
\begin{equation}\label{eq-towersubgroups2}
\cdots \subset \cH_{\nu +1}' \subset \cH_{\nu}' \subset \cdots \subset \cH_{\nu_0+1}' \subset \cH_{\nu_0}' \equiv \cH_{\ell_0}
\end{equation}
such that there  exists $\ell_1 \geq \ell_0$ so that for every $\ell \geq \ell_1$ there exists $\nu_{\ell} \geq \nu_0$ with $\cH_{\nu_{\ell}}' \subset \cH_{\ell}$, and
for every $\nu \geq \nu_0$ there exists $\ell_{\nu} \geq \ell_1$ with $\cH_{\ell_{\nu}} \subset \cH_{\nu}'$. Then the inverse limit space defined by the covering spaces   $\wtM_{\nu} \to M_{\ell_1}$ associated to the chain of subgroups (\ref{eq-towersubgroups2}) is homeomorphic to $\fM$.
\end{thm}
The proof of this result and its applications can be found in many sources \cite{ClarkHurder2011,FO2002, Mardesic2000,McCord1965,Rogers1970, RogersTollefson1971, Schori1966}.

 Thus, to show Theorem~\ref{thm-main1}, it suffices to produce $\ell_1 \geq \ell_0$ and  a chain of normal subgroups   $\cN_{i} \subset \cH_{\ell_1}$ which satisfy the criteria of Theorem~\ref{thm-equivalence}. The existence of these normal subgroups   follows from a geometric argument using the homogeneous hypothesis.

First, we recall a basic notion of group theory. Let $H \subset G$ be a subgroup, then the \emph{normal core} of $H$, or the \emph{core} for short, is the largest   subgroup $N \subset H$ so that $N$ is normal in $G$. We use the notation $C_G(H) = N$. It may happen that $C_G(H)$ is the trivial subgroup.
If $H$ has finite index in $G$, then for a set $\{g_1, \ldots , g_m\} \subset G$ consisting of  a representative of each residue class of $G/H$, then the core of $H$ is the subgroup
$$C_G(H) = \bigcap_{1 \leq i \leq m} ~ g_i^{-1} H g_i$$
Thus, if $G$ is an infinite group and $H$ has finite index in $G$, then $C_G(H)$ is always an infinite group, with finite index in $H$.

 For $\ell \geq \ell_0$ introduce  the sections
\begin{equation}\label{eq-sectionsell}
\fS_{\ell} \equiv   \vp_1^{-1}(0, V^{\ell})  \subset \cT_1 \subset \oU_1 ~ .
\end{equation}
Note that $x_0 \in \fS_{\ell} \subset \fS_{\ell_0}$  , and for any $x \in \fS_{\ell}$ we have $[x]_{\ell} = [x_0]_{\ell}$.
For $\ell' > \ell \geq \ell_0$, the covering maps $q_{\ell, \ell'} \colon M_{\ell'} \to M_{\ell}$ are induced by expanding the equivalence classes of $\approx_{\ell'}$ to those of $\approx_{\ell}$.

 Corollary~\ref{cor-effros} implies  there exists   $\delta_{\fM}$   so that
for any $x,y \in \fM$ with $d_{\fM}(x,y) < \delta_{\fM}$, there is a
homeomorphism $\theta: \fM \rightarrow \fM$ with   $h(x)=y$ and $d_{\cH}(\theta, id_{\fM}) \leq \eU/4$.

Recall the constants $\e_{\ell}$ defined in (\ref{eq-epsilonell}) with  $ \dX(V^{\ell}) < \e_{\ell}$ and where $\e_{\ell} \to 0$ monotonically.
Let $\ell_1 \geq \ell_0$ be chosen so that  $\fS_{\ell_1} \subset B_{\fM}(x_0, \delta_{\fM})$.
It follows that for any $\xi \in \fS_{\ell_1}$ there exists a homeomorphism $\theta: \fM\rightarrow \fM$ with $d_{\cH}(\theta, id_{\fM}) <\eU/4$ and $\theta(\xi)=x_0$.   Note that this condition implies that the map $\theta$ is  homotopic to map commuting with the quotient projections   $Q_{\ell} \colon \fM \to M_{\ell}$ for $\ell \geq \ell_1$, which is the condition used in the work \cite{FO2002}.

\begin{prop}\label{prop-cores}
Let $\ell  \geq \ell_1$ then there exists $\ell' \geq   \ell$   such that
$\ds \cH_{\ell'} \subset C_{\cH_{\ell_1}}(\cH_{\ell})$.
\end{prop}
\proof
 Recall   the constants  $\eta_{\ell} = \min\left\{d_{\fX}(W^{\ell}_{k},W^{\ell}_{k'}) \mid 1 \leq k \neq k' \leq \beta_{\ell}\right\}>0$ defined in  Section~\ref{sec-codes}.

The subgroup $\cH_{\ell}$ has finite index in $\cH_{\ell_1}$, so we can choose $\{g_1, \ldots , g_{m_{\ell}}\} \subset H_{\ell_1}$ consisting of  a representative of each residue class of $\cH_{\ell_1}/\cH_{\ell}$.
Let $\{\gamma_1, \ldots , \gamma_{m_{\ell}}\}$ be leafwise paths so that for $1 \leq i \leq m_{\ell}$ we have:
$$\gamma_i(0) = x_0    ~, ~ \gamma_i(1) = \xi_i \in  \fS_{\ell_1} ~, ~    [\gamma_i]_{\ell_1} = g_i \in  H_{\ell_1}$$
Let $ h_i$ denote the holonomy transformation  defined by the path  $\gamma_i$. Note that by Proposition~\ref{prop-uniformdom} and the choice of $\ell_0$, we can assume   that  $V^{\ell_1} \subset V^{\ell_0} \subset D(h_i)$ for $1 \leq i \leq m$.
Note that the inverse map $h_i^{-1}$ is defined by transport along the reverse path $\gamma_i^{-1}(t) = \gamma_i(1-t)$ and that
$[\gamma^{-1}_i]_{\ell_1} = g_i^{-1}  \in  H_{\ell_1}$.

For $1 \leq i \leq m_{\ell}$,    choose a homeomorphism $\theta_i: \fM \rightarrow \fM$ with   $\theta_i(\xi_i)=x_0$ and $d_{\cH}(\theta_i, id_{\fM}) \leq \eU/4$.

Let $\delta_{\ell} > 0$ be the constant of Lemma~\ref{lem-locconstant-ell} such that $B_{\fX}(w_0, \delta_{\ell}) \subset V^{\ell} \subset \fT_{\ell}$.
Set $\e_{\ell}^* = \rp(\delta_{\ell})$,  where the modulus function   $\rp(\e)$ for   the transverse coordinate projections was defined in  Lemma~\ref{lem-modpi}.

Each map $\theta_i $   for $1 \leq i \leq m_{\ell}$ is uniformly continuous, so there exists $\delta_{\ell}^* >0$ so that if $\xi, \xi' \in \fS_{\ell_1}$ satisfies $d_{\fM}(\xi, \xi') < \delta_{\ell}^*$ then $d_{\fM}(\theta_i(\xi), \theta_i(\xi')) < \e_{\ell}^*$.

Now choose $\ell' \geq \ell$ so that  $\fS_{\ell'} \subset B_{\fM}(x_0, \delta_{\ell}^*)$.
 We claim that $\ds \cH_{\ell'} \subset C_{\cH_{\ell_1}}(\cH_{\ell})$.

  Given $g_i \in \{g_1, \ldots , g_{m_{\ell}}\}$ and $b \in    H_{\ell'}$ we show that $b \in g_i^{-1} H_{\ell} g_i$. Equivalently, we show  that  the class $g_i b g_i^{-1} \in H_{\ell_1}$   is represented by a path $\sigma_i \colon [0,1] \to \fM$ such that $\sigma_i(0) = x_0$ and $\sigma_i(1) = \xi  \in \fS_{\ell}$ and thus $g_i b g_i^{-1} \in H_{\ell}$.

Choose a leafwise path $\gamma_b \colon [0,1] \to \fM$ so that:
$$\gamma_b(0) = x_0    ~, ~ \gamma_a(1) = \xi_b \in  \fS_{\ell'} ~, ~   [\gamma_b]_{\ell'}  = b \in  H_{\ell'}$$

  The endpoint $\xi_b \in  \fS_{\ell'} \subset \fS_{\ell_1}$ so we can define a path $\gamma_i'$ which shadows  $\gamma_i$ as in
  Lemma~\ref{lem-domainconst},  with
  $\gamma_i'(0) = \xi_b$ and $\gamma_i'(1) = \xi_i'$. Note that  $\xi_b \approx_{\ell'} x_0$ and    holonomy transport along any path preserves the coding decomposition   $\cV_{\ell'}$, hence $\xi_i' = h_i(\xi_b) \approx_{\ell'} h_i(x_0) = \xi_i$.

Form the concatenation $\sigma_i^* \equiv \gamma_i' * \gamma_b * \gamma_i^{-1} \colon [0,1] \to L_0$ which satisfies
$\sigma_i ^*(0) = \xi_i \in \fS_{\ell_1}$ and $\sigma_i^*(1) = \xi_i' \in \fS_{\ell_1}$.
As the endpoints of the path $\sigma_i ^*$ are $\approx_{\ell'}$ equivalent, we obtain a class
$[\sigma_i ^*]_{\ell'} \in \pi_1(M_{\ell'}, [\xi_i]_{\ell'})$ which is a representative for the lift of the element $g_i b g_i^{-1}  \in H_{\ell_1}$.

Now define the   leafwise path $\sigma_i = \theta_i \circ \sigma_i^*$ for each $1 \leq i \leq m_{\ell}$. Note that
$$\sigma_i(0) = \theta_i(\gamma_i(1)) = \theta_i(\xi_i) = x_0 \quad , \quad \sigma_i(1) = \theta_i(\gamma_i^*(1)) = \theta_i(\xi_i') = \xi_i''$$
As $ \xi_i, \xi_i' \in \fS_{\ell'} \subset B_{\fM}(x_0, \delta_{\ell}^*)$, by the choice of $\ell'$,  $\delta_{\ell}^*$ and $\e_{\ell}^*$ we have that $\xi_i'' \approx_{\ell} x_0$.

Moreover, the leafwise paths  satisfy $d_{\fM}(\sigma_i(t) , \sigma_i^*(t)) < \eU/4$ for all $0 \leq t \leq 1$ hence
by Lemma~\ref{lem-domainconst},   they determine the same holonomy transformations on their common domain and are homotopic when projected to $M_{\ell}$.
\endproof

Now  apply Proposition~\ref{prop-cores} inductively. Let $\ell_1 \geq \ell_0$ be defined  as above. Take $\ell = \ell_1 + 1$ and let $\ell_2 = \ell'$ be defined using Proposition~\ref{prop-cores}. Then repeat, let $\ell = \ell_2$ and set $\ell_3 = \ell'$, and so forth.  Define the normal subgroups of $\cH_{\ell_1}$ by
$\cN_{i} =  C_{\cH_{\ell_1}}(\cH_{\ell_i})$ where $\cH_{\ell_{i+1}} \subset \cN_i$.

By Theorem~\ref{thm-equivalence}, this completes the proof of Theorem~\ref{thm-main1}.

\section{Two Non-homogeneous Examples} \label{sec-examples}

Every   equicontinuous matchbox manifold   is homeomorphic to a solenoid by Theorem~\ref{thm-main2}, but it need not be a McCord solenoid and  homogeneous. In this section, we give two  general constructions of examples to illustrate this point. Before giving the examples, we first establish a basic result.

\begin{prop}\label{prop-equicex}
An $n$-dimensional solenoid  is an equicontinuous foliated space.
\end{prop}
\begin{proof}
Let
 $\ds \ds \cS = \lim_{\leftarrow} ~ \{p_{\ell+1} \colon M_{\ell +1} \to M_{\ell}\}$  be an  $n$-dimensional solenoid. Recall that by definition, we assume that each bonding maps $p_{\ell}$ is a   proper covering.

For ${\ell} \in \mN $, let $q_{\ell} \colon  \cS\rightarrow M_{\ell}$ denote projection onto the ${\ell}$--th factor.
Fix a point $x_0$ in $M_0$ and let ${\widehat q_0} \colon ({\widehat M_0},{\widehat x_0})\rightarrow (M_0,x_0)$ be the universal covering.
Now let $U_0$ be a neighborhood of $x_0$ homeomorphic to an open ball that is evenly covered by ${\widehat q_0}$ and let ${\widehat U_0}$ be the component of ${\widehat q_0}^{-1}(U_0)$ that contains ${\widehat x_0}$, and for $u \in U_0$, let ${\widehat u}$ denote the point of ${\widehat U_0}$ with ${\widehat q_0}({\widehat u})=u.$  We see that $ q_0^{-1}(x_0) \times U_0$ is homeomorphic to the open set $q_0^{-1}(U_0)$ in $\cS$ by considering the homeomorphism $h \colon q_0^{-1}(x_0) \times U_0\rightarrow q_0^{-1}(U_0)$  constructed as follows. For a given $f = \langle f_{\ell}\rangle \in q_0^{-1}(x_0)$, let ${\widehat q_{\ell}}[f] \colon ( {\widehat M_0},{\widehat x_0})\rightarrow (M_{\ell},f_{\ell})$ be the universal covering satisfying $p_1 \circ \cdots \circ p_{\ell}\circ {\widehat q_{\ell}}[f]= {\widehat q_0}.$ Then
$$ h((f,u))= \langle {\widehat q_{\ell}}[f]({\widehat u}) \rangle   $$
is the desired homeomorphism. The fiber $ q_0^{-1}(x_0)$ is homeomorphic to a Cantor set, and so if we choose an appropriate open
cover of $M_0$ consisting of sets evenly covered by ${\widehat q_0}$, we see that $\cS$ meets the definition of a matchbox manifold.
Also notice that ${\widehat M_0}$ is path connected and that for a given
$f = \langle f_{\ell}\rangle \in q_0^{-1}(x_0)$, $\langle{\widehat q_{\ell}}(f)({\widehat M_0}) \rangle$ is a path connected and dense subset of $\cS$. Thus, the path components of $\cS$ are dense  and  $\cS$ is a minimal foliated space. (See also \cite[Lemma 11 ff.]{FO2002}.)

To show that the foliation of $\F$ is equicontinuous, we show equicontinuity with respect to a pseudogroup $\cG_{\F}$ determined by a
foliation atlas of the form $\{q_0^{-1}(V_1),\dots,q_0^{-1}(V_N)\},$ where $\{V_1,\dots,V_N\}$ is an open cover of $M_0$ and the inverse
of each chart is of the form $h_i \colon  q_0^{-1}(x_i) \times V_i\rightarrow q_0^{-1}(V_i)$ as above, where each $x_i \in V_i$ and
$x_i \neq x_j$ for $i \neq j.$ Thus, $\cT_i = q_0^{-1}(x_i)$ and $\cT_*$ are subspaces of $\cS$. We identify the transverse space $\fX$ with $\cT_*$, but in $\fX$ we view the distances between points in distinct $\cT_i$ as being $1$. Now let $0<\e<1.$ For ${\ell}\in \mN$ let
$$\e_{\ell} = \max \{ {\rm diam} (q_{\ell}^{-1}(x)) \,|\,x\in M_{\ell} \}.$$  Then $\e_{\ell}\rightarrow 0$.
Choose $k$ with $\e_k< \e$ and let $d$ the degree of the covering $p_1 \circ \cdots \circ p_k.$ For each $i\in\{1,\dots,N\}$ let
$$\{y_1^i,\dots,y_d^i\}= (p_1 \circ \cdots \circ p_k)^{-1}(x_i)$$ and define $\delta>0$ to be the minimum of all the distances
in $\fX$ between the compact pairwise disjoint sets $q_k^{-1}(y^i_j)$ for $ (i,j)\in \{1,\dots,N\} \times \{1,\dots,d\}.$

Suppose that $w=\langle w_{\ell}\rangle, z= \langle z_{\ell} \rangle \in \cT_r$ are within $\delta$ and that both are contained in the domain of $g \in \cG_{\F}.$
As $w$ and $z$ are within $\delta$, we must have that $w_k=z_k=y^i_j$ for some $ (i,j)\in \{1,\dots,N\} \times \{1,\dots,d\}.$
Let $\cI=(i_0,\dots,i_{\alpha})$ be the admissible sequence associated to $g$. Let $\gamma_w , \gamma_z \colon [0,1]\rightarrow \cS$ be
paths from $w$ to $g(w)$ and $z$ to $g(z)$ covered by $(\cI,w)$ and $(\cI,z)$ respectively and constructed such that $q_0 \circ \gamma_w = q_0 \circ \gamma_z$. As the covering $p_1 \circ \cdots \circ p_k$ lifts the path $q_0 \circ \gamma_w = q_0 \circ \gamma_z$ to the paths $q_k\circ \gamma_w$ and   $q_k\circ \gamma_z$ that agree on the initial point $y_j^i$ in $M_k,$ the paths must agree on their endpoints  $q_k\circ g(w)=q_k\circ \gamma_w(1)=q_k\circ \gamma_z(1)=q_k\circ g(z).$ Thus, the distance between $g(w)$ and $g(z)$ is less than or equal to $\e_k <\e,$ as required to show that $\cG_{\F}$ is equicontinuous.
\end{proof}

\subsection{Non-homogeneous solenoids}\label{subsec-nhs}

Observe that Proposition~\ref{prop-equicex}  implies that the existence of an $n$-solenoid without holonomy which is not homogeneous yields
an equicontinuous, minimal, matchbox manifold that is  not homogeneous.

The first example of a solenoid that is not homogeneous was provided
by Schori~\cite{Schori1966}. This example is formed by taking a
specific sequence of non-normal three--to--one coverings of
orientable surfaces of increasing genus.

A simpler construction    of   a non-homogeneous solenoid was constructed by Fokkink and Oversteegen   in \cite[Theorem 35]{FO2002}, which gives an    example with    \emph{simply-connected} path components. We briefly describe this example.

Let  $\mathcal{S}_p$ denote  the
$p$--adic solenoid of dimension one, considered as an abelian
topological group.  The example can be described as the orbit
space of an involution $I$ on $\mathcal{S}_3 \times
\mathcal{S}_{35}$. To describe $I$, consider $\mathcal{S}_3 \times
\mathcal{S}_{35}$ the inverse limit of an inverse sequence of tori
$\mT^2 \equiv \mR^2/\mZ^2$ with single bonding map represented by the matrix
$\left(
                                 \begin{array}{cc}
                                   3 & 0 \\
                                   0 & 35 \\
                                 \end{array}
\right).$
The involution $(x,y) \mapsto (x+\frac{1}{2},-y)$ of $\mT^2$
  then induces the involution $I$ on $\mathcal{S}_3 \times
\mathcal{S}_{35}$. As the orbit space of the involution of the torus
described above is the Klein bottle, this space could also be
described as a solenoid over Klein bottles. This example is similar
to an example of Rogers and Tollefson~\cite{RogersTollefson1971},
which could be described as the orbit space of the analogously
defined involution on $S^1 \times \mathcal{S}_2$.

\subsection{Matchbox manifolds with holonomy}\label{subsec-mmwh}

We give an   construction of a   class of examples of equicontinuous matchbox manifolds      having leaves with infinite germinal holonomy groups. The method is  very general though abstract.

Let $\Lambda_1 = \langle g_1, \ldots , g_k \rangle$ be a finitely generated group, $K_1$ a Cantor set with metric $d_1$, and let $\rho_1 \colon \Lambda_1 \to {\bf Homeo}(K_1)$ define a minimal equicontinuous action of $\Lambda_1$ on $K_1$.

Let $K_0 \subset [0,1]$ be the ``standard'' middle thirds Cantor set, with coordinate $0 \leq t \leq 1$, with $0 \in K_0$. Let $d_0$ be the metric inherited from the interval $[0,1]$.    Let $\Lambda_1$ act on $K_0 \times K_1$ via the second coordinate, where for $g \in \Lambda_1$ and $(x,y) \in K_0 \times K_1$ we set $g   (x,y) = (x, g   y)$.  Now let
$$K = (K_0 \times K_1)/\{0\} \times K_1$$
where we collapse the ``vertical slice''  $\{0\} \times K_1$ to a point, denoted by $w_0 \in K$. Note that $K$ is again a Cantor set.
The action of $\Lambda_1$ on $K_0 \times K_1$ descends to a continuous action on $K$.

Define the \emph{warp product  metric} $d_K$ on $K$ by setting, for $(x,y), (x', y')  \in K_0 \times K_1$ and letting $[x,y], [x',y'] \in K$ denote their equivalence classes,
\begin{equation}\label{eq-newmetric}
d_K([x,y], [x',y']) ~ = ~  d_0(x,x')  + \max\{x,x'\} \cdot d_1(y,y')
\end{equation}
 Then the induced action of $\Lambda_1$ on $K$ is equicontinuous for this metric, but not minimal.

 Let $K_2$ be a Cantor set, and $\phi_2 \colon \mZ \to {\bf Homeo}(K_2)$ be any minimal equicontinuous action. Choose a homeomorphism $\Phi \colon K_2 \to K$ and let $\phi \colon \mZ \to  {\bf Homeo}(K)$ be the conjugate homeomorphism. Then the action $\phi$ of $\mZ$ on $K$
 is also equicontinuous, as $K$ is compact.

 Let $M = \Sigma_{k+1}$ be a surface of genus $m = k+1$, choose a basepoint $b_0 \in M$ and a surjection
 $$\Lambda \equiv \pi_1(M, b_0) \to \mZ^{*m} = \langle f_1, \ldots f_k, f_{k+1} \rangle$$
 onto the free group on $m$ generators. Then define a surjection of $ \mZ^{*m}$ onto the free product $\Lambda_1 * \mZ$, sending $f_i \mapsto g_i$, for $1 \leq i \leq k$, and $f_m \mapsto 1 \in \mZ$. In this way, we obtain a minimal equicontinuous action of $\Lambda$ on $K$.
 Furthermore, note that the subgroup of $\Lambda$ corresponding to the first generators, $\langle f_1, \ldots f_k \rangle$ fixes the point $w_0 \in K$, yet acts non-trivially on open neighborhoods in $K$ of this point.

 Suspend the action thus constructed of the fundamental group $\Lambda$ on $K$ to obtain a $2$-dimensional matchbox manifold, which is equicontinuous, hence minimal, and has very large infinite holonomy group for the leaf determined by the point $w_0$ in the transversal $K$.

\section{Codimension one}\label{sec-codim1}

If $\fM$ is a minimal matchbox
manifold, any homeomorphic copy of $\fM$ that  occurs as an
invariant subset of a foliation $\F$ of the same leaf dimension as $\fM$
must in fact be a minimal set of $\F$ since $\fM$ would then be the closure of each
of its leaves.

 As follows from a famous theorem of Denjoy and its generalization to
foliations, any sufficiently smooth foliation of the $2$-dimensional
torus is either minimal or has compact (circular) leaves.
Reeb~\cite{Reeb1961} conjectured that sufficiently smooth
codimension one foliations on closed manifolds could not have
exceptional leaves. However, Sacksteder and Schwartz~\cite{SS1964}
constructed a $C^{1}$ codimension one foliation on a closed
$3$-manifold that has a $2$-dimensional minimal set. This was
improved to a $C^{\infty}$ example by Sacksteder
in~\cite{Sacksteder1964} and to an analytic example by Hirsch
in~\cite{Hirsch1975}. Thus, smoothness alone poses no obstacle to
the existence of exceptional minimal sets in codimension one.
Reeb~\cite{Reeb1961} and Sacksteder~\cite{Sacksteder1965} did,
however, find added conditions on the foliation that eliminate the
possibility of exceptional leaves. Perhaps most notable of these
conditions is that the foliation be defined by a locally
free action of a connected Lie group ~(\cite{Sacksteder1965}).

\medskip

Alternatively, one could view the Denjoy  dichotomy purely
topologically. The exceptional Denjoy minimal sets of foliations of
the torus are known to be not homogeneous; see, e.g.,~\cite[Example
4]{AHO1991}. Hence, any homogeneous minimal set of a foliation of
the torus is a circle or the torus itself. As indicated below,  our
results imply that this statement generalizes in a natural way.
Thus, while smoothness does not force the regularity of minimal
sets, a natural topological condition does.

We recall the definition of an orientable foliated
space~\cite[Definition 11.2.14]{CandelConlon2000}.

\begin{defn}\label{orientable}
A smooth foliated space $X$ is \emph{orientable} if its tangent
bundle is an orientable vector bundle over $X$.
\end{defn}

\begin{thm}\label{thm-Denjdichot}
Let $\F$ be a smooth codimension one transversely orientable
foliation of a closed orientable manifold $M$. If $\fM$ is a
homogeneous minimal set of $\F$, then $\fM$ is a manifold.
\end{thm}

\begin{proof}
Suppose $\fM$ is a homogeneous minimal set of $\F$ that is neither a
closed leaf nor all of $M$. Then $\fM$ is an exceptional minimal
set, which in the codimension one case is well known to have the
structure of a matchbox manifold in which the transverse space can
be taken to be the Cantor set; see, e.g.,~\cite[III, Theorem
7]{CN1985}. As $\fM$ is homogeneous, Theorem~\ref{thm-main1} applies to allow us
to conclude that $\fM$ is a McCord solenoid. By our assumption that
$\F$ is transversely orientable and that $M$ is orientable, the
plane field associated to $\F$ is orientable, see, e.g.,~\cite[II,
Theorem 5]{CN1985}. It then follows that the manifolds $M_i$ formed
by collapsing the tubes in Section~\ref{sec-projections} are orientable. By Lemma
2 of~\cite{ClarkFokkink2004}, $\fM$ does not embed (even
topologically) in $M$, a contradiction.
\end{proof}

However, there are many examples of codimension $n \geq 2$
foliations with exceptional homogeneous minimal sets. For example,
it is not difficult to construct a smooth flow on a three manifold
without fixed points that has the dyadic solenoid as a minimal set.
As the Denjoy exceptional minimal set shows, without the condition
of homogeneity the above theorems fail.

This definition implies that an orientable foliated space admits a
foliation atlas in which all the leafwise transition maps have
Jacobians with positive determinant. If the matchbox manifold $\fM$
is orientable when regarded as a foliated space and if $\fM$ is at
the same time a McCord solenoid, then there exist an inverse limit
expansion for $\fM$ in which all the manifolds are orientable. As a
consequence of this and the impossibility of codimension one
embeddings of solenoids as shown in~\cite{ClarkFokkink2004}, we
obtain the following corollary, which is a generalization of the
result of Prajs~\cite{Prajs1990} that any homogeneous continuum in
$\mR^{n+1}$ which contains an $n$-cube is an $n$-manifold.

\begin{cor}\label{cor-codimen1a}
Any orientable homogeneous $n$-dimensional matchbox manifold
embedded in a closed, orientable $(n+1)$-dimensional manifold is
itself a manifold.
\end{cor}

\bigskip

\section{Problems}\label{sec-remarks}

We state   three open problems, motivated by the results of this work and   \cite{ClarkHurder2011}. The first two are in the spirit of Corollary~\ref{cor-codimen1a}, as they concern the consequences of a matchbox manifold being the minimal set for a smooth foliation.

\begin{prob}\label{prob1}
Let $\fM$ be an equicontinuous matchbox manifold embedded as a
minimal set of a $C^2$-foliation $\F$ of a closed manifold. Show that there exists a finite foliated covering    $\Pi \colon \widetilde{\fM} \to \fM$,
as in Definition~\ref{def-finiteholo}, such that  $\widetilde{\fM}$ is a McCord solenoid.
\end{prob}

In~\cite{ClarkHurder2011} we find embeddings of solenoids as minimal sets of smooth foliations. In all of these examples, the Galois groups of the covers in a
presentation of the solenoid are abelian. This leads to the following problem, which is an even stronger form of Problem~\ref{prob1}.

\begin{prob}\label{prob2}
Let $\cS$ be a McCord solenoid embedded in a $C^2 $-foliation $\F$ of a compact manifold.
Show that $\cS$ admits a presentation in which all the covers are abelian.
\end{prob}

Finally, we formulate some questions about the relationship between a matchbox manifold and its group of homeomorphisms.
Note that Fokkink and Oversteegen   ask a related question  at the conclusion of their work \cite{FO2002}.
Define the  normal closed topological  subgroup of ${\bf  Homeo}(\fM)$ consisting  of all leaf-preserving homeomorphisms
 $${\bf Inner}(\fM) = \{ h \in {\bf  Homeo}(\fM) \mid h(L) = L ~ {\rm for ~ all} ~ L \subset \fM\} ~ .$$
 In analogy with   geometric group theory constructions,     introduce the  group  of \emph{outer automorphisms} of a matchbox manifold $\fM$, which  is the quotient topological group
 \begin{equation}
{\bf Out}(\fM) = {\bf  Homeo}(\fM)/{\bf Inner}(\fM)
\end{equation}

 One can think of ${\bf Out}(\fM)$ as the group of automorphisms of the leaf space $\fM$, and thus  should reflect many aspects of the space $\fM$ -- its topological, dynamical and algebraic properties. Very little is known, in general, concerning some basic questions in higher dimensions:

\begin{prob}\label{prob3}
Let $\fM$ be a matchbox manifold with foliation $\F_{\fM}$.
\begin{enumerate}
\item If ${\bf Out}(\fM)$ is not discrete, must it act transitively?  If not, what are the examples?
\item If ${\bf Out}(\fM)$ is   discrete and infinite, what conditions on $\fM$ imply that it is finitely generated?
\item Suppose that $\fM$ is minimal and expansive, must ${\bf Out}(\fM)$ be discrete?
\item For what hypotheses on $\fM$ must  ${\bf Out}(\fM)$ be a finite group?
\end{enumerate}
\end{prob}

%%%%%%%%%%%%%%%%%%%%%%%%%%%%%%%%

\end{document}